\documentclass{amsart}
\usepackage{amsmath,amssymb,graphicx,setspace,verbatim,url}
\usepackage{color}
\usepackage{hyperref}

\input xy
\xyoption{all}
\newtheorem{prop}{Proposition}[section]
\newtheorem{coro}[prop]{Corollary}
\newtheorem{thm}[prop]{Theorem}
\newtheorem{lemma}[prop]{Lemma}

\newtheorem{induction}[prop]{Induction}

\begin{document}
\title{The number of string C-groups of high rank}
\author[P. J. Cameron]{Peter J. Cameron}
\address{Peter J. Cameron,
School of Mathematics and Statistics, 
University of St Andrews,
North Haugh,
St Andrews, Fife KY16 9SS,
UK
}
\email{pjc20@st-andrews.ac.uk}

\author[M. E. Fernandes]{Maria Elisa Fernandes}
\address{Maria Elisa Fernandes,
Center for Research and Development in Mathematics and Applications, Department of Mathematics, University of Aveiro, Portugal
}
\email{maria.elisa@ua.pt}

\author[D. Leemans]{Dimitri Leemans}
\thanks{This research was supported by an Action de Recherche Concertée of the Communauté Française Wallonie Bruxelles 	and by the Center for Research and Development 
in Mathematics and Applications (CIDMA) through the Portuguese 
Foundation for Science and Technology 
(FCT - Fundação para a Ciência e a Tecnologia), 
references UIDB/04106/2020 and UIDP/04106/2020. Dimitri Leemans also thanks warmly the Université Libre de Bruxelles for awarding a Chaire Internationale IN to Maria Elisa Fernandes to support a one-month stay during which this work was finalized.}
\address{Dimitri Leemans, Universit\'e Libre de Bruxelles, D\'epartement de Math\'ematique, C.P.216 - Alg\`ebre et Combinatoire, Boulevard du Triomphe, 1050 Brussels, Belgium
}
\email{leemans.dimitri@ulb.be}

\date{}
\maketitle

\begin{abstract}
If $G$ is a transitive group of degree $n$ having a string C-group of rank $r\geq (n+3)/2$, then $G$ is necessarily the symmetric group $S_n$.
We prove that if $n$ is large enough, up to isomorphism and duality, the number of string C-groups of rank $r$ for $S_n$ (with $r\geq (n+3)/2$) is the same as the number of string C-groups of rank $r+1$ for $S_{n+1}$. 
This result and the tools used in its proof, in particular the rank and degree extension, imply that if one knows the string C-groups of rank $(n+3)/2$ for $S_n$ with $n$ odd, one can construct from them all string C-groups of rank $(n+3)/2+k$ for $S_{n+k}$ for any positive integer $k$. 
The classification of the string C-groups of rank $r\geq (n+3)/2$ for $S_n$ is thus reduced to classifying string C-groups of rank $r$ for $S_{2r-3}$.
A consequence of this result is the complete classification of all string C-groups of $S_n$ with rank $n-\kappa$ for $\kappa\in\{1,\ldots,6\}$,  when $n\geq 2\kappa+3$, which extends previously known results.
The number of string C-groups of rank $n-\kappa$, with $n\geq 2\kappa +3$, of this classification gives the following sequence of integers indexed by $\kappa$ and starting at $\kappa = 1$:
 $$(1,1,7,9,35,48)$$ 
 This sequence of integers is new according to the On-Line Encyclopedia of Integer Sequences. It will be available as sequence number A359367.

\end{abstract}


\noindent \textbf{2000 Math Subj. Class:} 20B30, 52B11.

\section{Introduction}\label{intrud}
String C-group representations (or string C-groups for short) of a group $G$ consist of the group $G$ and an ordered set $S$ of generators of $G$ that are all involutions and that satisfy two properties called the string property and the intersection property (see Section~\ref{Cgroups} for the definitions). 
As observed in~\cite{arp}, string C-groups are in one-to-one correspondence with abstract regular polytopes. 
Starting from an abstract regular polytope, one can choose a flag of the polytope and obtain from it a set of involutory automorphisms that generate the group of automorphisms of the polytope. In the other direction, the usual way to reconstruct an abstract regular polytope from a string C-group is to use the so-called Tits algorithm, introduced by Jacques Tits in~\cite{Tits57}.
String C-groups are also interesting in the study of the subgroup structure of Coxeter groups as they are smooth quotients of Coxeter groups with a linear diagram.

The number of generators is the \emph{rank} of the string C-group. It is easy
to see that the string generators form an independent generating set for the
group $G$. The size of a minimal generating set for a permutation group of
degree $n$ is at most $n-1$, a result of Whiston~\cite{Whiston2000} who shows that an independent generating set of size $n-1$ necessarily generates the whole symmetric group $S_n$.
A string C-group of rank $n-1$ for $S_n$ can be obtained by taking the well known set of generators $\{(i,i+1): i \in \{1, \ldots, n-1\}\}$ consisting of consecutive transpositions. It was already known by Moore~\cite{Moore1896} and together with $S_n$, it gives a string C-group of maximal rank of $S_n$. This string C-group corresponds to the $(n-1)$-simplex.

In 2010, Fernandes and Leemans started working on string C-groups of high rank for $S_n$. They looked at the classification of all string C-groups of $S_n$ for $5\leq n \leq 9$ available in~\cite{LVatlas}.
These data suggested the following three theorems proved in~\cite{fl}.
In the first theorem, they proved that the group $S_n$ has, up to isomorphism and duality, a unique string C-group of rank $n-1$ for $n\geq 5$. It is Moore's representation mentioned above.
In the second theorem, they proved that $S_n$ has, up to isomorphism and duality, a unique string C-group of rank $n-2$ for $n\geq 7$.
Finally, in the third theorem, they proved that for $n\geq 4$, the group $S_n$ has string C-groups of rank $r$ for each $3\leq r\leq n-1$.

In 2011, together with Mixer, they started looking at alternating groups $A_n$ and managed to produce string C-groups of rank $\lfloor (n-1)/2 \rfloor$ for $n\geq 12$~\cite{flm,flm2}.  They then conjectured that these representations are of the highest possible rank for $A_n$ when $n\geq 12$.
They also pushed the classifications of string C-groups further to obtain all string C-groups of $S_n$ for $n\leq 14$ (see Table~\ref{sn} where the numbers of such representations are given for ranks $n-k$ with $1\leq k\leq 6$ and $5\leq n\leq 16$).
\begin{table}
\begin{tabular}{|c|c|c|c|c|c|c|c|c||}
\hline
$S_n$&Rk $n-1$&Rk $n-2$&Rk $n-3$&Rk $n-4$&Rk $n-5$&Rk $n-6$\\
\hline 
$S_5$&1&4&&&&\\
$S_6$&1&4&2&&&\\
$S_7$&1&1&7&35&&\\
$S_8$&1&1&11&36&68&\\
$S_9$&1&1&7&7&37&129\\
$S_{10}$&1&1&7&13&52&203\\
$S_{11}$&1&1&7&9&25&43\\
$S_{12}$&1&1&7&9&40&75\\
$S_{13}$&1&1&7&9&35&41\\
$S_{14}$&1&1&7&9&35&54\\
$S_{15}$&1&1&7&9&35&48\\
$S_{16}$&1&1&7&9&35&48\\
\hline
\end{tabular}
\caption{The number of pairwise nonisomorphic string C-groups of rank $n-k$ for $S_n$ with $1\leq k \leq 6$ and $5\leq n\leq 16$.}\label{sn}
\end{table}
Looking at the computational results, they proved that $S_n$ has, up to isomorphism and duality, 7 (resp. 9) string C-groups of rank $n-3$ (resp. $n-4$) when $n\geq 9$ (resp. 11)~\cite{extension}.

In 2016, Cameron, Fernandes, Leemans and Mixer proved that a string C-group for a transitive imprimitive group of degree $n$ has rank at most $\lfloor(n+2)/2\rfloor$ and that for a primitive group,  other than $S_n$ or $A_n$, the rank is bounded above by $n/2$. 
 They then proceeded to prove the conjecture above, what they now call the ``Aveiro Theorem''~\cite[Theorem 1.1]{2017CFLM}, which states that the maximal rank of a string C-group for $A_n$ is $\lfloor (n-1)/2\rfloor$ when $n\geq12$. 
 
 Similarly to what happens to the symmetric groups, Fernandes and Leemans proved in~\cite{2019fl} that the range
 for alternating groups is the interval from 3 to $\lfloor (n-1)/2\rfloor$ when $n\geq 12$. A major tool used in that proof is the Rank Reduction theorem by Brooksbank and Leemans~\cite[Theorem 1.1]{2019bl}.


As corollary of the aforementioned results, if $G$ is a transitive group of degree $n$ having a string C-group representation of rank $r\geq (n+3)/2$, then $G$ is necessarily the symmetric group $S_n$.
The authors then conjectured that if $n$ is large enough, up to isomorphism and duality, the number of string C-groups of rank $r$ for $S_n$, with $r\geq (n+3)/2$, is the same as the number of string C-group representations of rank $r+1$ for $S_{n+1}$ (see~\cite[Conjecture 1]{survey}).
The main aim of this paper is to prove this conjecture.

Let $\mathcal S(n,r)$ be the set of all string C-group representations of rank $r$ for $S_n$.
Define a relation $\sim$ on $\mathcal S(n,r)\times \mathcal S(n,r)$ by saying that for any elements $P, Q \in\mathcal S(n,r)$, $P\sim Q$ if and only if $P$ is isomorphic to $Q$ or to the dual of $Q$. The relation $\sim$ is an equivalence relation. 
Let $\Sigma^{\kappa}(n) = \mathcal S(n,n-\kappa)/\sim$. As mentioned before, the results of~\cite{fl,extension} give the following sequence.
\begin{center}
\begin{tabular}{l}
$|\Sigma^1(n)|=1$ for $n\geq 5$\\
$|\Sigma^2(n)|=1$ for $n\geq 7$\\
$|\Sigma^3(n)|=7$ for $n\geq 9$\\
$|\Sigma^4(n)|=9$ for $n\geq 11$\\
\end{tabular}
\end{center}
In addition, relying on computational results, it was conjectured that  
\begin{center}
\begin{tabular}{l}
$|\Sigma^5(n)|=35$ for $n\geq 13$ and \\
$|\Sigma^6(n)|=48$ for $n\geq 15$. 
\end{tabular}
\end{center}
In this paper we prove the following result, which implies the above conjectures.
\begin{thm}\label{main}
For each fixed integer $\kappa\geq 1$, there exists an integer $c_\kappa$ such that, for all $n\geq 2\kappa +3$, $|\Sigma^\kappa(n)| = c_\kappa$.
%
%
\end{thm}

This theorem and the tools used in its proof, in particular the rank and degree extension described in Section~\ref{RD}, imply that if one knows the string C-groups of rank $(n+3)/2$ for $S_n$ with $n$ odd, one can construct from them all string C-groups of rank $(n+3)/2+k$ for $S_{n+k}$ for any positive integer $k$. The classification of the string C-groups of rank $r\geq (n+3)/2$ for $S_n$ is thus reduced to classifying string C-groups of rank $r$ for $S_{2r-3}$.


The paper is organised as follows.
In Section~\ref{Cgroups}, we give the basic definitions about string C-groups.
We also include in that section some results on sesqui-extensions of string C-groups which were introduced in \cite{flm} and developed in \cite{flm2}.

In Section~\ref{T} we collect the known results on the upper bounds for the rank of string C-groups $\Gamma=(G,S)$ where $G$ is a transitive permutation group. We include in that section the classification of the highest rank string C-groups for the primitive groups of degree $n$ other than $S_n$ and $A_n$, and the classification of the string C-groups of highest rank for transitive imprimitive groups.

The study of string C-groups of high rank, specially for the alternating and symmetric  groups of degree $n$, was accomplished using permutation representation graphs of these groups on $n$ points. 
For a better understanding of these representations we use  a subgraph called fracture graph \cite{extension}. The concept of fracture graph turns out to be essential to the classification of the string C-groups for the symmetric group $S_n$ with ranks $n-3$ and $n-4$. This concept was developed further in \cite{2017CFLM} where a new subgraph appeared in the proof of the  ``Aveiro Theorem'', the 2-fracture graph.
In Section~\ref{frac} we classify all possible permutation representation graphs of string C-groups,  which have rank at least $(n-1)/2$ and admit a 2-fracture graph.
This classification includes string C-groups of the highest rank for alternating groups of odd degree. So it gives an important contribution
to the classification of the string C-groups of the highest rank for alternating groups. This classification has not yet been accomplished.

At the end of Section~\ref{frac}, we introduce the notion of a split in a permutation representation graph of a string group generated by involutions (sggi).
To be a little bit more precise, a permutation representation graph has a split if it has a fracture graph but no 2-fracture graph.

In Section~\ref{SSplit} we analyse the string C-groups that have a permutation representation graph with a split. The section begins with the definition of a perfect split. 
The main result here, is the classification of the sggi's  for $S_n$ of rank  $r\geq n/2$ that have a permutation representation graph with splits but no perfect splits.

In Section~\ref{RD}, we introduce a rank and degree operation (which we call the r\&d-extension) that, given a sggi of degree $n$ and rank $r$, constructs a sggi of degree $n+1$ and rank $r+1$ provided the sggi given has a perfect split.
We also describe this operation on string C-groups, having as main concern to give sufficient conditions to ensure that the r\&d extension of a string C-group is a string C-group.

In Section~\ref{se} we give special attention  to sesqui-extensions of alternating or symmetric groups and introduce the concepts of proper and improper sesqui-extensions.

In Section~\ref{tranS} we consider transitive string C-groups having a perfect split and show that when the rank is sufficiently large the r\&d-extension is possible.

Finally, in Section~\ref{theTh} we use the r\&d-extension to determine a one-to-one correspondence between $\Sigma^{\kappa}(n)$ and $\Sigma^{\kappa}(n+1)$ when $n\geq 2\kappa + 3$ and therefore prove Theorem~\ref{main}.

\section{String C-groups}\label{Cgroups}

A pair $\Gamma:=(G,S)$ is a \emph{string group generated by involutions} (for short, sggi) if $G := \langle S \rangle$ and $S := \{\rho_0, \ldots, \rho_{r-1}\}$ is an ordered set of involutions
satisfying the following property, called the {\em commuting property}.
\[\forall i,j\in\{0,\ldots, r-1\}, \;|i-j|>1\Rightarrow (\rho_i\rho_j)^2=1\]
If in addition $S$ satisfies the following property called the {\em intersection property}, that is,
\[\forall J, K \subseteq \{0,\ldots,r-1\}, \langle \rho_j \mid j \in J\rangle \cap \langle \rho_k \mid k \in K\rangle = \langle \rho_j \mid j \in J\cap K\rangle,\]
then $\Gamma$ is a {\em string C-group representation} or {\em string C-group} for short. In this case $G$ is a quotient of a Coxeter group with the following linear diagram (where $p_i$ is the order of $\rho_{i-1}\rho_i,\,i\in\{1,\ldots,r-1\}$).
$$ \xymatrix@-1pc{*{\bullet}  \ar@{-}[rr]^{p_1} && *{\bullet} \ar@{-}[rr]^{p_2} && *{\bullet}  \ar@{-}[rr]^{p_3} && *{\bullet} \ar@{.}[rr] && *{\bullet} \ar@{-}[rr]^{p_{r-2}} && *{\bullet} \ar@{-}[rr]^{p_{r-1}} &&*{\bullet} }$$
The above diagram is called the \emph{Coxeter diagram} of $\Gamma$ and the sequence $\{p_1,\,\ldots,\,p_{r-1}\}$ is  the  \emph{(Schl\"afli) type} of $\Gamma$. 
The \emph{rank} of a string C-group (or of a sggi) $\Gamma$ is the size of $S$. The group $G$ is called {\em the group of} $\Gamma$.
If $G$ is a transitive (resp. primitive, imprimitive) group, we say that $\Gamma$ is a transitive (resp. primitive, imprimitive) string C-group.
The \emph{dual} of a sggi $\Gamma=(G,\{\rho_0, \ldots, \rho_{r-1}\})$ is the sggi $(G,\{\rho_{r-1}, \ldots, \rho_0\})$.

Two string C-groups $\Gamma = (G,S)$ and $\Phi=(H,T)$ are {\em isomorphic} if $G\cong H$ and there exists an isomorphism $\varphi:G\rightarrow H$ such that $\varphi(S)=T$. 
Most of the results presented here are up to isomorphism and duality, meaning that we make no difference between string C-groups that are dual to each other.

Let $\Gamma:=(G,S)$ be a sggi with $S:=\{\rho_0,\,\ldots,\,\rho_{r-1}\}$. 
We let
\begin{eqnarray*}
G_{i_1,\ldots,i_k} &:=& \langle \rho_j : j \notin \{i_1,\ldots,i_k\}\rangle;\\
G_{\{i_1,\ldots,i_k\}} &:=& \langle \rho_j : j \in \{i_1,\ldots,i_k\}\rangle;\\
G_{<i} &:=& \langle \rho_0,\ldots, \rho_{i-1}\rangle\qquad (i\neq 0);\\
G_{>i} &:=& \langle \rho_i+1,\ldots, \rho_{r-1}\rangle\qquad (i\neq r-1).
\end{eqnarray*}
(Note the distinction between the first two definitions.)

Thus $G_{<0}$ and $G_{>r-1}$ are the trivial group.
We then let
\begin{eqnarray*}
\Gamma_{i_1,\ldots,i_k} &:=&(G_{i_1,\ldots,i_k}, \{\rho_j : j \notin \{i_1,\ldots,i_k\}\});\\
\Gamma_{\{i_1,\ldots,i_k\}} &:=&(G_{\{i_1,\ldots,i_k\}},\{ \rho_j : j \in \{i_1,\ldots,i_k\}\});\\
\Gamma_{<i} &:=& (G_{<i},\{ \rho_0,\ldots, \rho_{i-1}\})\qquad (i\neq 0);\\
\Gamma_{>i} &:=& (G_{<i},\{ \rho_i+1,\ldots, \rho_{r-1}\})\qquad (i\neq r-1).
\end{eqnarray*}
The {\em maximal parabolic subgroups} of $\Gamma$ are the subgroups $G_i$ with $i\in S$.

 Also, for $i,j \in  \{0,\ldots,\,r-1\}$, we denote 
 \begin{eqnarray*}
\Gamma_i&:=&(G_i,\{ \rho_j\mid j\in \{0, \ldots, r-1\}\setminus \{ i\}\});\\
\Gamma_{i,j}&:=&(\Gamma_i)_j.
\end{eqnarray*}
Usually our sggi or string C-group $\Gamma = (G,S)$ will be such that $G$ is a subgroup of the symmetric
group $S_n$; we call the integer $n$ the \emph{degree} of $\Gamma$.

The following result shows that when $\Gamma_0$ and $\Gamma_{r-1}$ are string C-groups, the intersection property for $\Gamma$ is verified by checking only one condition.

\begin{prop}\cite[Proposition 2E16]{arp}\label{arp}
Let $\Gamma:=(G,\{\rho_0,\ldots,\rho_{r-1}\})$ be a sggi.
Suppose that $\Gamma_0$ and $\Gamma_{r-1}$ are string C-groups. 
Then $\Gamma$ is a string C-group if and only if $G_0\cap G_{r-1} = G_{0,r-1}$.
\end{prop}
Let  $\Gamma$ be a sggi of degree $n$ generated by a set of $r$ involutions $S:=\{\rho_0,\ldots,\rho_{r-1}\}$.
The \emph{permutation representation graph} of $\Gamma$ is a $r$-edge-labeled multigraph  with $n$ vertices and with an $i$-edge $\{a,\,b\}$ whenever $a\rho_i=b$ with $a\neq b$. 

A pair $\Gamma:=(G,S)$ is a sggi if and only if its permutation representation graph satisfies the following properties:
\begin{enumerate}
\item The graph induced by edges of label $i$ is a matching;
\item Each connected component of the graph induced by edges of labels $i$ and $j$, for $|i-j| \geq 2$, is a single vertex, a single edge, a double edge, or a square with alternating labels.
\end{enumerate}

The term sesqui-extension was first introduced in \cite{flm}. Let us recall its meaning.
Let $\Gamma := (G,\{\rho_0,\ldots,\rho_{r-1}\})$ be a sggi, and let $\tau$ be an involution in a supergroup of $G$ such that $\tau \not \in G$ and $\tau$ commutes with all of $G$.  
For a fixed $k$, we define the set $S^* := \{\rho_i \tau^{\eta_i}\mid i\in \{0,\,\ldots,\,r-1\}\}$ where $\eta_i = 1$ if $i=k$ and 0 otherwise. Let $G^*:=\langle S^*\rangle$.
The {\it sesqui-extension} of $\Gamma$ with respect to $\rho_k$ and $\tau$ is the sggi $\Gamma^* := (G^*,S^*)$.

\begin{prop} \label{sesqui0}\cite[Proposition 3.3]{flm}
If $\Gamma$ is a string C-group, and $\Gamma^*$ is a sesqui-extension of $\Gamma$ with respect to the first generator, then
$\Gamma^*$ is a string C-group.
\end{prop}

\begin{lemma}\cite[Lemma 5.4]{flm2}\label{sesqui}  
If  $\Gamma^*$ is a sesqui-extension of $\Gamma$ with respect to $\rho_k$, then the following hold:
\begin{enumerate}
\item $G^*\cong G$ or $G^*\cong G\times \langle \tau\rangle\cong G\times C_2$. 
\item If the identity element of $G$ can be written as a product of generators involving an odd number of $\rho_k$'s, then $G^*\cong G\times \langle \tau\rangle$.
\item If $G$ is a finite permutation group, $\tau$ and $\rho_k$ are odd permutations, and all other $\rho_i$ are even permutations, then $G^*\cong G$.
\item Whenever $\tau\notin  G^*$, $\Gamma$ is a string C-group if and only if $\Gamma^*$ is a string C-group.
\end{enumerate}
\end{lemma}
We now extend Proposition~\ref{sesqui0}.
\begin{prop}\label{sesqui1}
Let $\Gamma^*$ be a sesqui-extension of a sggi $\Gamma$ with respect to $\rho_0$.
Then $\Gamma^*$ is a string C-group if and only if $\Gamma$ is a string C-group.
\end{prop}
\begin{proof}
This proposition holds trivially when the rank of $\Gamma$ is $1$.
Suppose by induction that it is true for rank smaller than $r$.
Let $\Gamma^*$ be a string C-group of rank $r$.


As $\Gamma_0=\Gamma^*_0$, $\Gamma_0$ is a string C-group.
Now either  $\Gamma^*_{r-1}\cong \Gamma_{r-1}$ or $\Gamma^*_{r-1}$ is a proper sesqui-extension with respect to $\rho_0$ of  $\Gamma_{r-1}$. Then, by induction, $\Gamma_{r-1}$ is a string C-group.

We have
$$G_{0,r-1}=G^*_{0,r-1}=G^*_0\cap G^*_{r-1}=G_0\cap G^*_{r-1}=G_0\cap G_{r-1}.$$
Hence $\Gamma$ is a string C-group.

The converse follows from Proposition~\ref{sesqui0}.
\end{proof}


In this paper $D_n$ denotes the dihedral group of degree $n$ and order $2n$.

We give in the following proposition a series of permutation representation graphs that do not satisfy the intersection property. These graphs, when appearing as subgraphs of larger permutation representation graphs, will permit us later in the paper to prove that several sggi's we will look at do not satisfy the intersection property.

\begin{prop}\label{IPfails}
If $\Gamma:=(G,\{ \rho_i\mid i\in I\})$ is a sggi that has one of the following permutation representation graphs, then $\Gamma$ is not a string C-group.
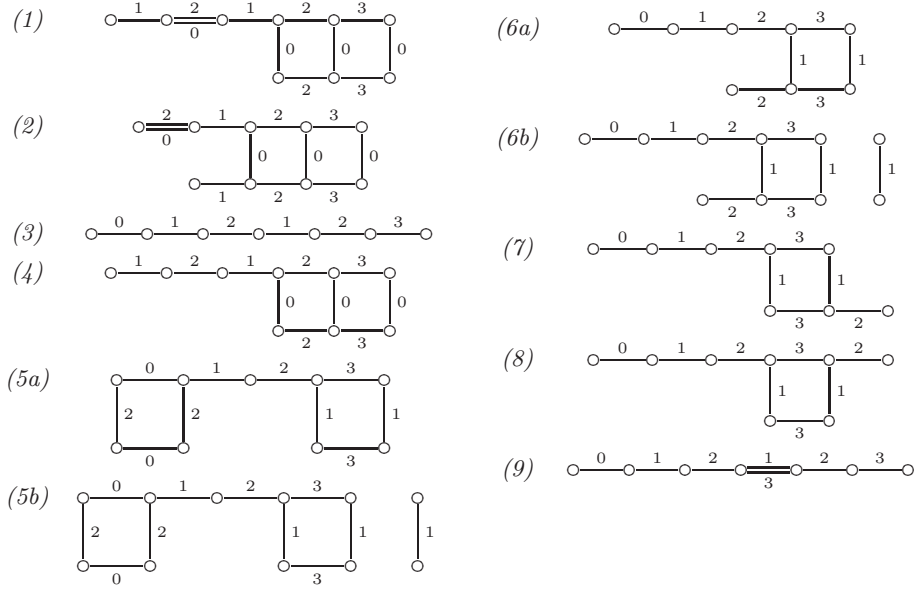
\begin{center}
\begin{small}
\begin{tabular}{cc}
\begin{tabular}{cc}
(1) &$\xymatrix@-0.7pc{  *+[o][F]{} \ar@{-}[r]^1 & *+[o][F]{}  \ar@{=}[r]^2_0 & *+[o][F]{} \ar@{-}[r]^1 & *+[o][F]{} \ar@{-}[r]^2 & *+[o][F]{} \ar@{-}[r]^3& *+[o][F]{} \\
 & & &*+[o][F]{} \ar@{-}[r]_2  \ar@{-}[u]_0&  *+[o][F]{} \ar@{-}[u]_0 \ar@{-}[r]_3 & *+[o][F]{}  \ar@{-}[u]_0  } $\\
(2) &$\xymatrix@-0.7pc{   *+[o][F]{}  \ar@{=}[r]^2_0 & *+[o][F]{} \ar@{-}[r]^1 & *+[o][F]{} \ar@{-}[r]^2 & *+[o][F]{} \ar@{-}[r]^3 & *+[o][F]{} \\
  & *+[o][F]{} \ar@{-}[r]_1 &*+[o][F]{} \ar@{-}[r]_2  \ar@{-}[u]_0&  *+[o][F]{} \ar@{-}[u]_0 \ar@{-}[r]_3 & *+[o][F]{}  \ar@{-}[u]_0 } $\\
(3)&$\xymatrix@-0.7pc{  *+[o][F]{} \ar@{-}[r]^0 & *+[o][F]{}  \ar@{-}[r]^1 & *+[o][F]{} \ar@{-}[r]^2 & *+[o][F]{} \ar@{-}[r]^1 & *+[o][F]{} \ar@{-}[r]^2&*+[o][F]{} \ar@{-}[r]^3 &*+[o][F]{} }$\\
(4) &$\xymatrix@-0.7pc{  *+[o][F]{} \ar@{-}[r]^1 & *+[o][F]{}  \ar@{-}[r]^2 & *+[o][F]{} \ar@{-}[r]^1 & *+[o][F]{} \ar@{-}[r]^2 & *+[o][F]{} \ar@{-}[r]^3& *+[o][F]{} \\
 & & &*+[o][F]{} \ar@{-}[r]_2  \ar@{-}[u]_0&  *+[o][F]{} \ar@{-}[u]_0 \ar@{-}[r]_3 & *+[o][F]{} \ar@{-}[u]_0  } $\\
(5a) &$\xymatrix@-0.35pc{ *+[o][F]{}  \ar@{-}[r]^0   & *+[o][F]{}  \ar@{-}[r]^1 & *+[o][F]{}  \ar@{-}[r]^2& *+[o][F]{}  \ar@{-}[r]^3& *+[o][F]{}    \\
  *+[o][F]{}  \ar@{-}[r]_0\ar@{-}[u]_2& *+[o][F]{}   \ar@{-}[u]_2&     &  *+[o][F]{}  \ar@{-}[r]_3\ar@{-}[u]_1& *+[o][F]{}  \ar@{-}[u]_1 } $\\
  (5b) &$\xymatrix@-0.35pc{ *+[o][F]{}  \ar@{-}[r]^0   & *+[o][F]{}  \ar@{-}[r]^1 & *+[o][F]{}  \ar@{-}[r]^2& *+[o][F]{}  \ar@{-}[r]^3& *+[o][F]{} & *+[o][F]{}   \\
  *+[o][F]{}  \ar@{-}[r]_0\ar@{-}[u]_2& *+[o][F]{}   \ar@{-}[u]_2&     &  *+[o][F]{}  \ar@{-}[r]_3\ar@{-}[u]_1& *+[o][F]{}  \ar@{-}[u]_1& *+[o][F]{}  \ar@{-}[u]_1}$ \\
\end{tabular}
&
\begin{tabular}{cc}
(6a) &$\xymatrix@-0.6pc{ *+[o][F]{}  \ar@{-}[r]^0& *+[o][F]{} \ar@{-}[r]^1& *+[o][F]{}   \ar@{-}[r]^2 & *+[o][F]{}  \ar@{-}[r]^3& *+[o][F]{}  \\
 & & *+[o][F]{}  \ar@{-}[r]_2& *+[o][F]{}  \ar@{-}[r]_3\ar@{-}[u]_1 & *+[o][F]{} \ar@{-}[u]_1 }$\\
(6b) &$\xymatrix@-0.6pc{ *+[o][F]{}  \ar@{-}[r]^0& *+[o][F]{} \ar@{-}[r]^1& *+[o][F]{}   \ar@{-}[r]^2 & *+[o][F]{}  \ar@{-}[r]^3& *+[o][F]{} & *+[o][F]{}   \\
 & & *+[o][F]{}  \ar@{-}[r]_2& *+[o][F]{}  \ar@{-}[r]_3\ar@{-}[u]_1 & *+[o][F]{} \ar@{-}[u]_1& *+[o][F]{}  \ar@{-}[u]_1 }$\\
(7) &$\xymatrix@-0.6pc{ *+[o][F]{}  \ar@{-}[r]^0 & *+[o][F]{} \ar@{-}[r]^1& *+[o][F]{}   \ar@{-}[r]^2& *+[o][F]{}  \ar@{-}[r]^3& *+[o][F]{} &\\
 & && *+[o][F]{}  \ar@{-}[r]_3 \ar@{-}[u]_1& *+[o][F]{}  \ar@{-}[r]_2\ar@{-}[u]_1& *+[o][F]{}   }$ \\
(8)&$\xymatrix@-0.6pc{ *+[o][F]{}  \ar@{-}[r]^0 & *+[o][F]{} \ar@{-}[r]^1 & *+[o][F]{}   \ar@{-}[r]^2 & *+[o][F]{}  \ar@{-}[r]^3& *+[o][F]{} \ar@{-}[r]^2&*+[o][F]{} \\
& && *+[o][F]{}  \ar@{-}[r]_3 \ar@{-}[u]_1& *+[o][F]{} \ar@{-}[u]_1&   }$\\
(9)&$\xymatrix@-0.7pc{  *+[o][F]{} \ar@{-}[r]^0 & *+[o][F]{}  \ar@{-}[r]^1 & *+[o][F]{} \ar@{-}[r]^2 & *+[o][F]{} \ar@{=}[r]^1_3 & *+[o][F]{} \ar@{-}[r]^2&*+[o][F]{} \ar@{-}[r]^3 &*+[o][F]{} }$\\
&\\
&\\
&\\
\end{tabular}
\end{tabular}
\end{small}
\end{center}
\end{prop}
\begin{proof}
Let $\Gamma$ be a sggi having one of the permutation representation graphs given in this proposition.
In what follows we show that, by Proposition~\ref{arp}, $\Gamma$ does not satisfy the intersection property.\\

\begin{tabular}{cl}
(1) &$G_0\cong S_6\times S_3$; $G_3\cong S_7\times C_2$;  $G_3\cap G_0\cong S_5\times C_2$; $G_{0,3} \cong D_{10}$\\[5pt]
(2) & $G_0\cong S_5\times S_4$; $G_3\cong S_7\times C_2$; $G_3\cap G_0\cong S_4\times S_3$; $G_{0,3} \cong D_{12}$\\[5pt]
(3) & $G_0\cong S_6$;  $G_3\cong S_6$;  $G_0\cap G_3\cong S_5$; $G_{0,3}\cong D_5$\\[5pt]
(4) & $G_0\cong S_6\times S_3$; $G_3 \cong S_7\times C_2$; $G_3\cap G_0\cong S_5\times C_2$; $G_{0,3} \cong D_{10}$\\[5pt]
(5a) & $G_0\cong G_3 \cong  S_7\times C_2$, $G_3\cap G_0\cong S_5\times C_2\times C_2$; $G_{0,3}\cong D_{10}$\\[5pt]
(5b) & $G_0\cong G_3 \cong  S_7\times C_2$, $G_3\cap G_0\cong S_5\times C_2\times C_2$; $G_{0,3}\cong D_{10}$\\[5pt]
(6a) &$G_0 \cong S_7$; $G_3\cong S_6\times C_2$; $G_3\cap G_0\cong S_5\times C_2$; $G_{0,3} \cong D_{10}$\\[5pt]
(6b) &$G_0 \cong S_7$; $G_3\cong S_6\times C_2$; $G_3\cap G_0\cong A_5\times C_2$; $G_{0,3} \cong D_{10}$\\[5pt]
(7) &$G_0 \cong  S_7$; $G_3\cong S_5\times S_3$; $G_3\cap G_0\cong S_4\times S_3$; $G_{0,3} \cong D_{12}$\\[5pt]
(8) &$G_0 \cong  S_7$; $G_3\cong S_5\times S_3$; $G_3\cap G_0\cong S_4\times S_3$; $G_{0,3}\cong D_{12}$\\[5pt]
(9)& $G_0\cong A_6$;  $G_3\cong S_6$; $G_3\cap G_0\cong A_5$; $G_{0,3}\cong D_5$
\end{tabular}

\end{proof}
\section{Transitive String C-groups of high rank}\label{T}
We recall in this section what is known about maximal ranks of transitive string C-groups.
Transitive groups are either primitive or imprimitive.
Theorem~\ref{altn} gives the maximal ranks of alternating groups, Theorem~\ref{bigimp} gives the string C-groups of highest possible ranks for imprimitive groups and Theorem~\ref{bigprim} deals with primitive groups that are not alternating groups nor symmetric groups.

\begin{thm}\cite[Theorem 1.1]{2017CFLM}\label{altn}
The maximal rank of a string C-group for $A_n$ is $3$ if $n=5$; $4$ if $n=9$; $5$ if $n=10$; $6$ if $n=11$ and $\lfloor\frac{n-1}{2}\rfloor$ if $n\geq 12$.
Moreover, if $n=3, 4, 6, 7$ or $8$, the group $A_n$ does not admit a string C-group.
\end{thm}

The string C-groups for $A_n$ with rank greater than $\lfloor\frac{n-1}{2}\rfloor$ are, up to duality, the ones listed in Table~\ref{an} (see \cite{flm}).

\begin{table}
\[\begin{array}{|c|c|c|c|}
\hline
\# & n &r&\emph{Permutation Representation Graph}\\
\hline
(1) & 5& 3 &\xymatrix@-1.3pc{*+[o][F]{}  \ar@{-}[rr]^1 && *+[o][F]{} \ar@{-}[rr]^0&& *+[o][F]{} \ar@{-}[dll]_1\\
  &&*+[o][F]{}  \ar@{-}[rr]_0 \ar@{-}[u]^2&&*+[o][F]{}  \ar@{-}[u]_2}\\
(2)& 5& 3 &\xymatrix@-1.3pc{*+[o][F]{}  \ar@{-}[rr]^1 && *+[o][F]{} \ar@{-}[rr]^0&& *+[o][F]{}\\
  &&*+[o][F]{}  \ar@{-}[rr]_0 \ar@{-}[u]^2&&*+[o][F]{}  \ar@{=}[u]_1^2}\\
(3)& 10& 5 &\xymatrix@-1.3pc{*+[o][F]{} \ar@{-}[rr]^0 && *+[o][F]{}  \ar@{-}[rr]^1&&*+[o][F]{} \ar@{=}[rr]^0_2 && *+[o][F]{} \ar@{-}[rr]^1 && *+[o][F]{} \ar@{-}[rr]^2&&*+[o][F]{}  \ar@{-}[rr]^3 && *+[o][F]{} \ar@{-}[rr]^4&& *+[o][F]{} \ar@{-}[rr]^3&&*+[o][F]{}  \ar@{-}[rr]^4&&*+[o][F]{} }\\
(4) & 10& 5 &\xymatrix@-1.3pc{*+[o][F]{} \ar@{=}[rr]^0_2 && *+[o][F]{}  \ar@{-}[rr]^1&&*+[o][F]{} \ar@{-}[rr]^0 && *+[o][F]{} \ar@{-}[rr]^1 && *+[o][F]{} \ar@{-}[rr]^2&&*+[o][F]{}  \ar@{-}[rr]^3 && *+[o][F]{} \ar@{-}[rr]^4&& *+[o][F]{}\\
&& && && && &&  &&*+[o][F]{}  \ar@{-}[rr]_4 \ar@{-}[u]^2&&*+[o][F]{}  \ar@{=}[u]_3^2}\\
(5)& 10& 5 &\xymatrix@-1.3pc{*+[o][F]{} \ar@{-}[rr]^0 && *+[o][F]{}  \ar@{-}[rr]^1&&*+[o][F]{} \ar@{=}[rr]^0_2 && *+[o][F]{} \ar@{-}[rr]^1 && *+[o][F]{} \ar@{-}[rr]^2&&*+[o][F]{}  \ar@{-}[rr]^3 && *+[o][F]{} \ar@{-}[rr]^4&& *+[o][F]{}\\
&& && && && &&  &&*+[o][F]{}  \ar@{-}[rr]_4 \ar@{-}[u]^2&&*+[o][F]{}  \ar@{=}[u]_3^2}\\
(6)& 10& 5 &\xymatrix@-1.3pc{*+[o][F]{} \ar@{=}[rr]^0_2  && *+[o][F]{}  \ar@{-}[rr]^1&&*+[o][F]{} \ar@{-}[rr]^0 && *+[o][F]{} \ar@{-}[rr]^1 && *+[o][F]{} \ar@{-}[rr]^2&&*+[o][F]{}  \ar@{-}[rr]^3 && *+[o][F]{} \ar@{-}[rr]^4&& *+[o][F]{} \ar@{-}[rr]^3&&*+[o][F]{}  \ar@{-}[rr]^4&&*+[o][F]{} }\\
(7)& 11& 6 &\xymatrix@-1.3pc{*+[o][F]{} \ar@{=}[rr]^0_2  && *+[o][F]{}  \ar@{-}[rr]^1&&*+[o][F]{} \ar@{-}[rr]^0 && *+[o][F]{} \ar@{-}[rr]^1 && *+[o][F]{} \ar@{-}[rr]^2&&*+[o][F]{}  \ar@{-}[rr]^3 && *+[o][F]{} \ar@{-}[rr]^4&& *+[o][F]{} \ar@{-}[rr]^5&&*+[o][F]{}  \ar@{-}[rr]^4&&*+[o][F]{}  \ar@{=}[rr]^5_3&& *+[o][F]{}}\\
(8)& 11& 6 &\xymatrix@-1.3pc{*+[o][F]{} \ar@{-}[rr]^0  && *+[o][F]{}  \ar@{-}[rr]^1&&*+[o][F]{} \ar@{=}[rr]_2^0 && *+[o][F]{} \ar@{-}[rr]^1 && *+[o][F]{} \ar@{-}[rr]^2&&*+[o][F]{}  \ar@{-}[rr]^3 && *+[o][F]{} \ar@{-}[rr]^4&& *+[o][F]{} \ar@{-}[rr]^5&&*+[o][F]{}  \ar@{-}[rr]^4&&*+[o][F]{}  \ar@{=}[rr]^5_3&& *+[o][F]{}}\\
(9)& 11& 6 &\xymatrix@-1.3pc{*+[o][F]{} \ar@{-}[rr]^0&& *+[o][F]{}  \ar@{-}[rr]^1&&*+[o][F]{} \ar@{=}[rr]_2^0 && *+[o][F]{} \ar@{-}[rr]^1 && *+[o][F]{} \ar@{-}[rr]^2&&*+[o][F]{}  \ar@{-}[rr]^3 && *+[o][F]{} \ar@{-}[rr]^4&& *+[o][F]{} \ar@{=}[rr]^5_3&&*+[o][F]{}  \ar@{-}[rr]^4&&*+[o][F]{}  \ar@{-}[rr]^5&& *+[o][F]{}}\\
\hline
\end{array} \]
\caption{Permutation representation graphs of string C-groups of rank $\geq n/2$ for $A_n$.}\label{an}
\end{table}

\begin{thm}\label{bigimp}
Let $\Gamma$ be a transitive imprimitive string C-group of degree $n$ and rank $r>n/2$.
The permutation representation graph of $\Gamma$ is one of the graphs of Table~\ref{BI}.
\end{thm}
\begin{proof}
This is a direct consequence of~\cite[Theorem 1.2]{2016CFLM}. 
\end{proof}
Observe that, as pointed out to us by Mark Mixer, in~\cite[Theorem 1.2]{2016CFLM}, one graph was missing, namely the graph (1) with $x=r$ of Table~\ref{BI}.
\begin{small}
\begin{table}[htbp]
\[\begin{array}{|cccc|}
\hline

(1) & \xymatrix@-1.4pc{
*+[o][F]{}\ar@{=}[dd]_x^1\ar@{-}[rr]^{r-1}&&*+[o][F]{}\ar@{=}[dd]_x^1\ar@{.}[rr]&&*+[o][F]{}\ar@{-}[rr]^3 \ar@{=}[dd]_x^1 &&*+[o][F]{}   \ar@{=}[dd]_x^1\ar@{-}[rr]^2&& *+[o][F]{}   \ar@{-}[dd]^x \\
&&&&&&&&\\
*+[o][F]{}\ar@{-}[rr]_{r-1}*+[o][F]{}&&*+[o][F]{}\ar@{.}[rr]&&*+[o][F]{}\ar@{-}[rr]_3&&*+[o][F]{} \ar@{-}[rr]_2&& *+[o][F]{}\\
x &=&0 &\textrm{or} &r&&&&\\ }
&
(3)&  \xymatrix@-1.4pc{
  &&*+[o][F]{}   \ar@{-}[dd]_0\ar@{-}[rr]^2\ar@{-}[ddrr]^3&& *+[o][F]{}   \ar@{-}[dd]^0 \ar@{-}[rr]^1&&  *+[o][F]{}   \ar@/_{2.5pc}/@{-}[lllllldd]_3\\
&&&&\\
*+[o][F]{}\ar@{-}[rr]_1&&*+[o][F]{} \ar@{-}[rr]_2\ar@/_{3.3pc}/@{-}[uurr]_3&& *+[o][F]{} }\\
(2)& \quad \xymatrix@-1.3pc{*+[o][F]{}  \ar@/_{.5pc}/@{=}[dd]_0^3\ar@{-}[rr]^2  &&*+[o][F]{}  \ar@{=}[rr]^3_1&& *+[o][F]{}\ar@{-}[rr]^2&& *+[o][F]{} \ar@/_{.7pc}/@{=}[dd]_4^0\\
&&&&&&\\
*+[o][F]{} \ar@/_{.5pc}/@{-}[rrrr]_1\ar@/^{3pc}/@{-}[uurr]^4\ar@{-}[rr]_2&&*+[o][F]{} \ar@{-}[rr]_3 \ar@{=}[uull]^1_4\ar@{-}[uu]_0 &&*+[o][F]{} \ar@{=}[uu]_4^0\ar@{-}[rr]_2&& *+[o][F]{} \ar@/_{.7pc}/@{=}[uu]_1^3}
&
(4) &\xymatrix@-1.4pc{
*+[o][F]{}\ar@{=}[rr]^1_2 \ar@{=}[dd]_0^3 &&*+[o][F]{}   \ar@{-}[dd]^0\ar@{-}[rr]^3&& *+[o][F]{}   \ar@{-}[dd]_0 \\
&&&&\\
*+[o][F]{} \ar@/_{4pc}/@{-}[rrrruu]^1\ar@{-}[rr]_2&&*+[o][F]{} \ar@{=}[rr]_1^3&& *+[o][F]{} }
\\
&&&\\
\hline
\end{array}\] 
\caption{Permutation representation graphs of transitive imprimitive string C-groups of degree $n$ with rank $\ge \frac{n+1}{2}$.
}\label{BI}
\end{table}
\end{small}

\begin{thm}\cite[Proposition 3.2] {2016CFLM}\label{bigprim}
Let $G$ be a primitive permutation group of degree $n$ which is not $S_n$ or $A_n$.
Let $\Gamma := (G,S)$ be a string C-group of rank $r$.
Then  $r< n/2$ or the permutation representation graph of $\Gamma$ is one of those listed in Table~\ref{primPolys}.
\end{thm}

\begin{table}[htbp]
\[\begin{array}{|cc|}
\hline
\begin{array}{cc}
(1)& \xymatrix@-1.3pc{*+[o][F]{}  \ar@{-}[rr]^2 &&*+[o][F]{}  \ar@{-}[rr]^3 && *+[o][F]{} \ar@{-}[rr]^4&& *+[o][F]{}\\
&&&&&&\\
&&*+[o][F]{}  \ar@{-}[rr]^3  \ar@{-}[uu]^1   &&*+[o][F]{}  \ar@{-}[rr]^4 \ar@{-}[uu]^1&&*+[o][F]{}  \ar@{-}[uu]_1\\
&&&&&&\\
&&*+[o][F]{}  \ar@{-}[rr]_3  \ar@{-}[uu]^0  &&*+[o][F]{}  \ar@{-}[rr]_4 \ar@{-}[uu]^0 \ar@/^.2pc/@{-}[uurr]_2&&*+[o][F]{}  \ar@{-}[uu]_0 \ar@/_.2pc/@{-}[uull]_2}\\
(2)&\xymatrix@-1.3pc{*+[o][F]{} \ar@{=}[rr]^0_2 && *+[o][F]{}  \ar@{-}[rr]^1&&*+[o][F]{} \ar@{-}[rr]^0 && *+[o][F]{} \ar@{-}[rr]^1 && *+[o][F]{} \ar@{-}[rr]^2&&*+[o][F]{} }\\
(3)& \xymatrix@-1.3pc{*+[o][F]{} \ar@{-}[rr]^0 && *+[o][F]{}  \ar@{-}[rr]^1&&*+[o][F]{} \ar@{=}[rr]^0_2 && *+[o][F]{} \ar@{-}[rr]^1 && *+[o][F]{} \ar@{-}[rr]^2&&*+[o][F]{} }\\

(4)&\xymatrix@-1.3pc{*+[o][F]{}  \ar@{-}[rr]^1 &&*+[o][F]{}  \ar@{-}[rr]^0 && *+[o][F]{} \\
*+[o][F]{} \ar@{=}[u]_0^2 \ar@{-}[rr]_1&&*+[o][F]{} \ar@{-}[rr]_2  &&*+[o][F]{} \ar@{-}[u]_1}\\
\end{array}&
\begin{array}{cc}
(5)&\xymatrix@-1.3pc{*+[o][F]{}  \ar@{-}[rr]^0 &&*+[o][F]{}  \ar@{-}[rr]^1 && *+[o][F]{} \ar@{-}[rr]^0&& *+[o][F]{} \ar@{-}[dll]_1\\
&&  &&*+[o][F]{}  \ar@{-}[rr]_0 \ar@{-}[u]^2&&*+[o][F]{}  \ar@{-}[u]_2}\\
(6) &\xymatrix@-1.3pc{*+[o][F]{}  \ar@{=}[rr]^0_2 &&*+[o][F]{}  \ar@{-}[rr]^1 && *+[o][F]{} \ar@{-}[rr]^0&& *+[o][F]{}\\
&&  &&*+[o][F]{}  \ar@{=}[rr]_0^1 \ar@{-}[u]^2&&*+[o][F]{}  \ar@{-}[u]_2}\\
(7)&\xymatrix@-1.3pc{*+[o][F]{}  \ar@{-}[rr]^2  \ar@/^1pc/@{-}[rrrr]^1 &&*+[o][F]{}  \ar@{-}[rr]^3&& *+[o][F]{} \\
&&&&\\
*+[o][F]{} \ar@{=}[uu]_0^3  \ar@{-}[uurr]^1\ar@{-}[rr]_2&&*+[o][F]{} \ar@{=}[rr]_3^1 \ar@{-}[uu]^0  &&*+[o][F]{} \ar@{=}[uu]_2^0}\\
(8)&\xymatrix@-1.3pc{*+[o][F]{}  \ar@{-}[rr]^0 &&*+[o][F]{}  \ar@{-}[rr]^1 && *+[o][F]{} \ar@{-}[rr]^0&& *+[o][F]{}\\
&&  &&*+[o][F]{}  \ar@{=}[rr]_0^1 \ar@{-}[u]^2&&*+[o][F]{}  \ar@{-}[u]_2}\\
\end{array}\\
\hline
\end{array}\]
\caption{Permutation representation graphs of primitive string C-groups of degree $n$ whose group is neither $S_n$ nor $A_n$,  with rank $\ge n/2$.}\label{primPolys}
\end{table}

\section{String C-groups of high rank having a 2-fracture graph}\label{frac}

Let $\Gamma := (G,\{\rho_0,\,\ldots,\,\rho_{r-1}\})$ be a sggi such that $G$ acts faithfully on a set $\{1,\,\ldots,\,n\}$.
Suppose that for every $i\in \{0, \ldots, r-1\}$, the subgroup $G_i$ 
has at least one more orbit than $G$.
Then, for each $i$, the involution $\rho_i$ permutes a pair of points lying in
different $G_i$-orbits. Choosing one such transposition of $\rho_i$ for each $i$, and
regarding them as the edges of a graph on the vertex set $\{1,\ldots,n\}$,
we obtain a graph with $r$ edges that we call a \emph{fracture graph} for $\Gamma$ \cite{extension}.
A fracture graph is a spanning forest of the permutation representation graph of $\Gamma$.


\begin{prop}\label{int}
Let $G$ be a transitive group of degree $n$. Let $\Gamma := (G,\{\rho_0,\ldots, \rho_{r-1}\})$ be a string C-group of rank $r$. Let $\kappa = n - r$.
If $n\geq 2\kappa+3$ then $G\cong S_n$ and $\Gamma$ admits a fracture graph.
\end{prop}
\begin{proof}
First observe that $n\geq 2\kappa+3\Leftrightarrow r\geq \frac{n+3}{2}$. Then by Theorems~\ref{altn} to~\ref{bigprim}, $G\cong S_n$.
It remains to prove that $\Gamma$ admits a fracture graph.
If $n\leq 11$ (and therefore $\kappa\leq 4$), this is a direct consequence of the main theorem of~\cite{extension}. 
Assume that $n\geq 12$.
Suppose $\Gamma$ is such that there exists an $i\in\{0, \ldots, r-1\}$ with $G_i$ transitive. Then $\Gamma_i$ is a string C-group of rank $r-1\geq  \frac{n+1}{2}$. As $G_i$ is transitive, by Theorems~\ref{altn} to~\ref{bigprim}, this inequality is satisfied only if $n \equiv 2 \mod 4$, $r-1=n/2+1$ and $\Gamma_i$ has, up to a duality, the permutation representation graph (1) given in Table~\ref{BI}  (with rank $r-1$). Particularly $\rho_i$ is either the first or the last generator, it
must break the blocks of imprimitivity for $G$ to be isomorphic to $S_n$, and it must commute with all but one existing generators, which is impossible.
Hence every $G_i$ is intransitive and $\Gamma$ admits a fracture graph.
\end{proof}

If it happens that for each $i$ we can find two transpositions of $\rho_i$ such that, for each transposition, its points are in
different $G_i$-orbits, then taking an $i$-edge between each of these pairs of points we obtain a graph on $n$ vertices with $2r$ edges that we call a
\emph{$2$-fracture graph}. 

In \cite[Section 4]{2017CFLM} a series of results for a transitive string C-group $\Gamma$ whose  permutation representation graph $\mathcal{G}$ admits a 2-fracture graph were proven; we need these to prove our main result.
In what follows, for $\{i,\ldots, j\}\subseteq \{0,\ldots,\,r-1\}$, let $\mathcal{G}_{\{i,\ldots,j\}}$ denote the permutation representation graph of $\Gamma_{\{i,\ldots,j\}}$.
In the graphs below the dashed edges represent edges of $\mathcal{G}$ not belonging to the 2-fracture graph.

\begin{prop}\cite[Proposition 4.1]{2017CFLM}\label{2fpaths}
If $e=\{v,w\}$ is an $i$-edge of a $2$-fracture graph $\mathcal Q$ of $\Gamma$, then any path from $v$ to $w$ in $\mathcal{G}$, which does not contain $e$, must contain another $i$-edge.
\end{prop}


\begin{prop}\cite[Proposition 4.5]{2017CFLM}\label{swap}
If there is a cycle $C$ in $\mathcal{G}$ containing exactly two $i$-edges $e_1, e_2$, such that $e_1$ is in a $2$-fracture graph $\mathcal{Q}$ and $e_2$ is not, then there is another $2$-fracture graph $\mathcal{Q}'$ obtained by removing $e_1$ and adding $e_2$.
\end{prop}

\begin{prop}\cite[Proposition 4.7]{2017CFLM}\label{movingsquares}
Let $q_{i,j}$ be an alternating square with a vertex $v$ that is $l$-adjacent to a vertex $w$ in a $2$-fracture graph  $\mathcal{Q}$.  If $l$ is not consecutive with $i$, then the square can be moved to include the edge $\{v,w\}$.  That is, there is another $2$-fracture graph $\mathcal{Q}'$ obtained from  $\mathcal{Q}$ by changing exactly two edges, as pictured below.   Furthermore $\mathcal{Q}'$ does not have more alternating squares than $\mathcal{Q}$.
$$ \xymatrix@-1.3pc{*+[o][F]{}  \ar@{-}[rr]^j \ar@{-}[dd]^i && *+[o][F]{v}  \ar@{-}[rr]^l \ar@{-}[dd]^i&&*+[o][F]{w}\ar@{--}[dd]^i\\
 &&&&\\
 *+[o][F]{}\ar@{-}[rr]_j &&*+[o][F]{}\ar@{--}[rr]_l&& *+[o][F]{}}\qquad\longrightarrow\qquad\xymatrix@-1.3pc{*+[o][F]{}  \ar@{-}[rr]^j \ar@{--}[dd]^i && *+[o][F]{v}  \ar@{-}[rr]^l \ar@{-}[dd]^i&&*+[o][F]{w}\ar@{-}[dd]^i\\
 &&&&\\
 *+[o][F]{}\ar@{-}[rr]_j &&*+[o][F]{}\ar@{-}[rr]_l&& *+[o][F]{}}$$
\end{prop}

\begin{prop}\cite[Proposition 4.8]{2017CFLM}\label{changelabels}
If a $2$-fracture graph $\mathcal{Q}$ contains the subgraph on the left of the following figure, then there is a $2$-fracture graph $\mathcal{Q}'$ containing the subgraph on the right, such that  $\mathcal{Q}$ and  $\mathcal{Q}'$ differ only in two edges.
$$ \xymatrix@-1.3pc{*+[o][F]{}  \ar@{-}[rr]^{i-1} \ar@{-}[dd]_{i+1}&& *+[o][F]{}   \ar@{-}[dd]_{i+1}&&*+[o][F]{} &&\\
 &&&&\\
 *+[o][F]{}\ar@{-}[rr]_{i-1}&&*+[o][F]{}\ar@{-}[rr]_i\ar@{--}[uurr]^l && *+[o][F]{}\ar@{-}[rr]_l&& *+[o][F]{}\ar@{--}[uull]_i}\qquad\longrightarrow\qquad\xymatrix@-1.3pc{*+[o][F]{}  \ar@{-}[rr]^{i-1} \ar@{-}[dd]_{i+1}&& *+[o][F]{}   \ar@{-}[dd]_{i+1}&&*+[o][F]{} &&\\
 &&&&\\
 *+[o][F]{}\ar@{-}[rr]_{i-1}&&*+[o][F]{}\ar@{--}[rr]_i\ar@{-}[uurr]^l && *+[o][F]{}\ar@{--}[rr]_l&& *+[o][F]{}\ar@{-}[uull]_i} $$
\end{prop}

\begin{prop}\cite[Proposition 4.12]{2017CFLM}\label{frac1}
If the permutation representation graph of $\Gamma$ has a $2$-fracture graph but has no connected $2$-fracture graph, then it has a $2$-fracture graph that has at least one component which is a tree, and all the others having only one cycle (which is an alternating square).
\end{prop}
\begin{prop}\cite[Proposition 4.18 and Corollary 4.19]{2017CFLM}\label{2fracC}
Let $n\geq 9$. If $\Gamma$ has a connected $2$-fracture graph, the $2$-fracture graph is a either a tree or  one of the graphs (1) and (2) of Table~\ref{T2F}.
\end{prop}
{
In what follows we describe, up to duality, the transitive string C-groups of high rank admitting a 2-fracture graph.

\begin{prop}\label{frac2bis}
Let $G$ be a transitive group of degree $n\geq 9$ and let $\Gamma := (G,S)$ be a string C-group of rank $r$.
Suppose that $\Gamma$ has a $2$-fracture graph. Then $r\leq n/2$.
\end{prop}
\begin{proof}
Suppose that $\Gamma$ has a 2-fracture graph $\mathcal F$. By Propositions~\ref{frac1} and \ref{2fracC}, each connected component of $\mathcal F$ is either a tree or has at most a cycle, which is an alternating square.
Trees with $v$ vertices have $v-1$ edges while for the other components the number of edges and vertices coincide. So in total $\mathcal F$ has a number of vertices $n $ such that $2r = n-c$, with $c$ being the number of components that are trees. Hence $r\leq n/2$.
\end{proof}

\begin{prop}\label{frac2}
Let $G$ be a transitive group of degree $n$ and let $\Gamma := (G,\{\rho_0, \ldots, \rho_{r-1}\})$ be a string C-group of rank $r$.
Suppose that $\Gamma$ has a $2$-fracture graph. If  $r\geq \frac{n-1}{2}$ then $\Gamma$ has a permutation representation graph isomorphic to one of the graphs given in Table~\ref{T2F}.
\end{prop}

\begin{proof}
If $n\leq 8$ a computer search\footnote{See \url{https://leemans.dimitri.web.ulb.be/transitive/index.html} for the results of the search.} shows that all possible graphs are listed in Table~\ref{T2F}.
Let $n\geq 9$.
By assumption $\Gamma$ has a 2-fracture graph. Proposition~\ref{frac2bis} then implies that $r\leq n/2$.
Suppose that  $r= n/2$. In this case, $n$ is even and Proposition~\ref{frac1} implies that $\Gamma$ has a connected 2-fracture graph that is not a tree, hence by~\cite[Corollary 4.19]{2017CFLM}, either $n$ is even and $G\cong C_2\wr S_{n/2}$ or $n$ is odd and $G\cong C_2^{n/2-1}: S_{n/2}$ corresponding to the graphs (1), (2) of Table~\ref{T2F}. 

Now let us consider the case $r=\frac{n-1}{2}$. 
Let $\mathcal G$ be the permutation representation graph of $\Gamma$.
If $\Gamma$ has a connected 2-fracture graph, this graph has $2r = n-1$ edges and therefore it is a tree.
If there is no connected 2-fracture graph, by Proposition~\ref{frac1}, there exists a 2-fracture graph that has exactly one component which is a tree and all the others having exactly one cycle, which is an alternating square.
Suppose $\mathcal{F}$ is such a 2-fracture graph and assume that the number $s$ of squares of $\mathcal{F}$ is minimal. We deal with cases $s>0$ and $s=0$ separately.

Case $s>0$: As $\mathcal{F}$ contains all vertices of $\mathcal{G}$, any component of $\mathcal{F}$ is at distance one from another component of $\mathcal{F}$ in $\mathcal G$.
Consider two components of $\mathcal{F}$ at distance one. Let $\{v,w\}$ be an edge of $\mathcal{G}$  between two components of  $\mathcal{F}$.
By Proposition~\ref{movingsquares}, we may assume $\mathcal{F}$ contains one of the following subgraphs.

$$ \mbox{(i)} 
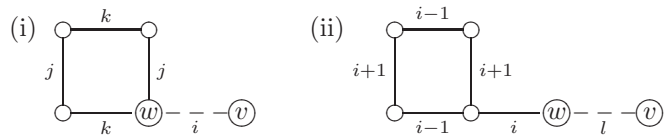
\xymatrix@-1.3pc{*+[o][F]{}  \ar@{-}[rr]^k \ar@{-}[dd]_j && *+[o][F]{}  \ar@{-}[dd]^j\\
 &&&&\\
 *+[o][F]{}\ar@{-}[rr]_k &&*+[o][F]{w}\ar@{--}[rr]_i&& *+[o][F]{v}} \qquad 
 \mbox{(ii)}  \xymatrix@-1.3pc{*+[o][F]{}  \ar@{-}[rr]^{i-1} \ar@{-}[dd]_{i+1} && *+[o][F]{}  \ar@{-}[dd]^{i+1}&&\\
 &&&&\\
 *+[o][F]{}\ar@{-}[rr]_{i-1}&&*+[o][F]{}\ar@{-}[rr]_i&&*+[o][F]{w}\ar@{--}[rr]_l&& *+[o][F]{v}} $$

If the label of the edge $\{w,v\}$, in (i), is nonconsecutive with at least one of the edges of the alternating squares, then by Propositions~\ref{2fpaths} and \ref{swap}, 
there exists a fracture graph with $s-1$ squares. 
$$
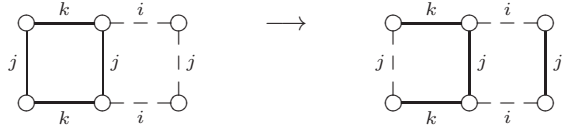
\xymatrix@-1.3pc{*+[o][F]{}  \ar@{-}[rr]^k \ar@{-}[dd]_j && *+[o][F]{}  \ar@{-}[dd]^j\ar@{--}[rr]^i&& *+[o][F]{}  \ar@{--}[dd]^j\\
 &&&&\\
 *+[o][F]{}\ar@{-}[rr]_k&&*+[o][F]{}\ar@{--}[rr]_i&&*+[o][F]{}} \qquad\longrightarrow\qquad \xymatrix@-1.3pc{*+[o][F]{}  \ar@{-}[rr]^k \ar@{--}[dd]_j && *+[o][F]{}  \ar@{-}[dd]^j\ar@{--}[rr]^i&& *+[o][F]{}  \ar@{-}[dd]^j\\
 &&&&\\
 *+[o][F]{}\ar@{-}[rr]_k &&*+[o][F]{}\ar@{--}[rr]_i&&*+[o][F]{}} $$
As this operation does not create another cycle in the 2-fracture graph (for otherwise, the 2-fracture graph would have at least three $j$-edges, as that cycle would have another $j$-edge by Proposition~\ref{2fpaths}), we get a contradiction with the assumption that the number of squares $s$ is minimal. Hence $|j-k|=2$, $j=i-1$ and $k=i+1$.

In case (ii), $l\in\{i\pm1\}$ otherwise we get a situation as in (i).

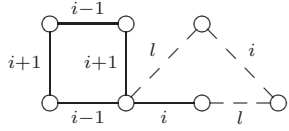
$$ \xymatrix@-1.3pc{*+[o][F]{}  \ar@{-}[rr]^{i-1} \ar@{-}[dd]_{i+1}&& *+[o][F]{}   \ar@{-}[dd]_{i+1}&&*+[o][F]{} &&\\
 &&&&\\
 *+[o][F]{}\ar@{-}[rr]_{i-1}&&*+[o][F]{}\ar@{-}[rr]_i\ar@{--}[uurr]^l && *+[o][F]{}\ar@{--}[rr]_l&& *+[o][F]{}\ar@{--}[uull]_i} $$
 
Let us now prove that in both cases, (i) and (ii), the connected component containing $v$ must be a tree.
Suppose that the component containing $v$ has a square. Then, by Proposition~\ref{movingsquares} there exists a 2-fracture graph such that the squares are at distance at most three in $\mathcal{G}$, as in the following figure.

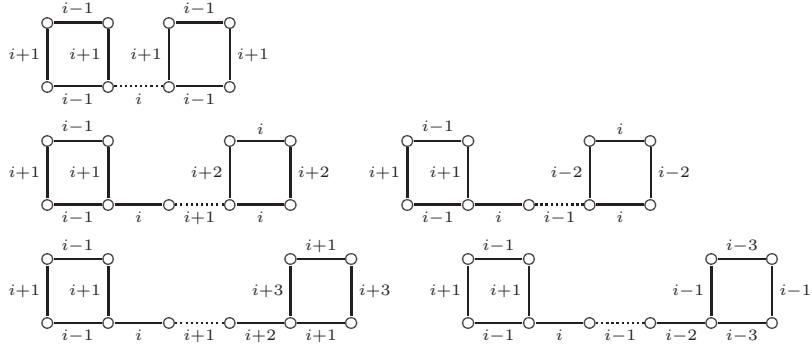
\begin{small}
\[\begin{array}{l}
 \xymatrix@-1.5pc{*+[o][F]{}  \ar@{-}[rr]^{i-1} \ar@{-}[dd]_{i+1}&& *+[o][F]{}  \ar@{-}[dd]_{i+1}&&   *+[o][F]{}\ar@{-}[rr] \ar@{-}[dd]_{i+1}\ar@{-}[rr]^{i-1}&&*+[o][F]{}\ar@{-}[dd]^{i+1}\\
 &&&& &&\\
 *+[o][F]{}\ar@{-}[rr]_{i-1}&& *+[o][F]{}\ar@{.}[rr]_i&& *+[o][F]{}\ar@{-}[rr]_{i-1}&&*+[o][F]{}}\\
 
  \xymatrix@-1.5pc{*+[o][F]{}  \ar@{-}[rr]^{i-1} \ar@{-}[dd]_{i+1}&& *+[o][F]{}  \ar@{-}[dd]_{i+1}&&&&   *+[o][F]{}\ar@{-}[rr] \ar@{-}[dd]_{i+2}\ar@{-}[rr]^i&&*+[o][F]{}\ar@{-}[dd]^{i+2}\\
 &&&& &&&&\\
 *+[o][F]{}\ar@{-}[rr]_{i-1}&& *+[o][F]{}\ar@{-}[rr]_i&& *+[o][F]{}\ar@{.}[rr]_{i+1}&& *+[o][F]{}\ar@{-}[rr]_i&&*+[o][F]{}}\quad 
 \xymatrix@-1.5pc{*+[o][F]{}  \ar@{-}[rr]^{i-1} \ar@{-}[dd]_{i+1}&& *+[o][F]{}  \ar@{-}[dd]_{i+1}&&&&   *+[o][F]{}\ar@{-}[rr] \ar@{-}[dd]_{i-2}\ar@{-}[rr]^i&&*+[o][F]{}\ar@{-}[dd]^{i-2}\\
 &&&& &&&&\\
 *+[o][F]{}\ar@{-}[rr]_{i-1}&& *+[o][F]{}\ar@{-}[rr]_i&& *+[o][F]{}\ar@{.}[rr]_{i-1}&& *+[o][F]{}\ar@{-}[rr]_i&&*+[o][F]{}}\\
 \xymatrix@-1.5pc{*+[o][F]{}  \ar@{-}[rr]^{i-1} \ar@{-}[dd]_{i+1}&& *+[o][F]{}  \ar@{-}[dd]_{i+1}&&&&&&   *+[o][F]{}\ar@{-}[rr] \ar@{-}[dd]_{i+3}\ar@{-}[rr]^{i+1}&&*+[o][F]{}\ar@{-}[dd]^{i+3}\\
 &&&& &&&&&&\\
 *+[o][F]{}\ar@{-}[rr]_{i-1}&& *+[o][F]{}\ar@{-}[rr]_i&& *+[o][F]{}\ar@{.}[rr]_{i+1}&&*+[o][F]{}\ar@{-}[rr]_{i+2}&& *+[o][F]{}\ar@{-}[rr]_{i+1}&&*+[o][F]{}}\quad \xymatrix@-1.5pc{*+[o][F]{}  \ar@{-}[rr]^{i-1} \ar@{-}[dd]_{i+1}&& *+[o][F]{}  \ar@{-}[dd]_{i+1}&&&&&&   *+[o][F]{}\ar@{-}[rr] \ar@{-}[dd]_{i-1}\ar@{-}[rr]^{i-3}&&*+[o][F]{}\ar@{-}[dd]^{i-1}\\
 &&&& &&&&&&\\
 *+[o][F]{}\ar@{-}[rr]_{i-1}&& *+[o][F]{}\ar@{-}[rr]_i&& *+[o][F]{}\ar@{.}[rr]_{i-1}&&*+[o][F]{}\ar@{-}[rr]_{i-2}&& *+[o][F]{}\ar@{-}[rr]_{i-3}&&*+[o][F]{}}
 \end{array}\]
\end{small}
In any case we get at least three edges of the 2-fracture graph with the same label, a contradiction.
So we may assume that the connected component containing $v$ is a tree.

Let $\nu$ be the number of vertices of that component.
Let us first consider the case where $\nu>1$.
In this case the component having a square cannot be as in (i) for, otherwise, the 2-fracture graph would have again at least three edges with same label.
Moreover the edge $\{v,w\}$ of $\mathcal{G}$, in (ii), must be a double edge with labels $i\pm 1$, otherwise, if $\{v,w\}$ is a single edge with label $i-1$ or $i+1$, this edge belongs to a $2$-fracture graph with less squares contradicting the minimality of $s$.
Consequently $\nu= 2$ and $\mathcal G$ must be the following graph.

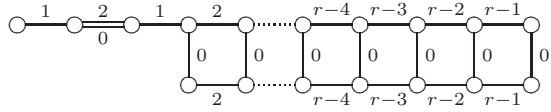
\[ \begin{array}{c} \xymatrix@-0.7pc{  *+[o][F]{} \ar@{-}[r]^1 & *+[o][F]{}  \ar@{=}[r]^2_0 & *+[o][F]{} \ar@{-}[r]^1 & *+[o][F]{} \ar@{-}[r]^2 & *+[o][F]{} \ar@{.}[r] & *+[o][F]{}  \ar@{-}[r]^{r-4}& *+[o][F]{}  \ar@{-}[r]^{r-3}  & *+[o][F]{}  \ar@{-}[r]^{r-2}  & *+[o][F]{}  \ar@{-}[r]^{r-1}  & *+[o][F]{}   &\\
 & & &*+[o][F]{} \ar@{-}[r]_2  \ar@{-}[u]_0&  *+[o][F]{} \ar@{-}[u]_0 \ar@{.}[r] & *+[o][F]{}  \ar@{-}[r]_{r-4} \ar@{-}[u]_0& *+[o][F]{}  \ar@{-}[r]_{r-3}\ar@{-}[u]_0 &  *+[o][F]{}  \ar@{-}[r]_{r-2}\ar@{-}[u]_0 &*+[o][F]{}  \ar@{-}[r]_{r-1}\ar@{-}[u]_0&  *+[o][F]{} \ar@{-}[u]_0  } 
\end{array}\]
When $n>7$, this graph does not give a permutation representation of a string C-group as $\Gamma_{\{0,1,2,3\}}$ is a sesqui-extension with respect to $\rho_0$ of the sggi having the permutation representation graph (1) of Proposition~\ref{IPfails}.
Thus, by Propositions~\ref{sesqui1} and \ref{IPfails},  $n=7$. This gives one possibility which is the graph (4) of Table~\ref{T2F}.

Now let $\nu=1$. 
Suppose we have situation (i) and there are two components with squares connected to vertex $v$, then we get the following possibilities.
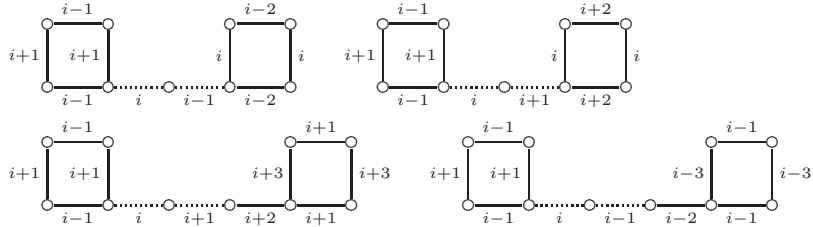
\begin{small}
\[\begin{array}{l}
 \xymatrix@-1.5pc{*+[o][F]{}  \ar@{-}[rr]^{i-1} \ar@{-}[dd]_{i+1}&& *+[o][F]{}  \ar@{-}[dd]_{i+1}&& &&  *+[o][F]{}\ar@{-}[rr] \ar@{-}[dd]_i\ar@{-}[rr]^{i-2}&&*+[o][F]{}\ar@{-}[dd]^i\\
 &&&&&&&&\\
 *+[o][F]{}\ar@{-}[rr]_{i-1}&& *+[o][F]{}\ar@{.}[rr]_i&& *+[o][F]{}\ar@{.}[rr]_{i-1}&& *+[o][F]{}\ar@{-}[rr]_{i-2}&&*+[o][F]{}}\quad
  \xymatrix@-1.5pc{*+[o][F]{}  \ar@{-}[rr]^{i-1} \ar@{-}[dd]_{i+1}&& *+[o][F]{}  \ar@{-}[dd]_{i+1}&& &&  *+[o][F]{}\ar@{-}[rr] \ar@{-}[dd]_i\ar@{-}[rr]^{i+2}&&*+[o][F]{}\ar@{-}[dd]^i\\
 &&&&&&&&\\
 *+[o][F]{}\ar@{-}[rr]_{i-1}&& *+[o][F]{}\ar@{.}[rr]_i&& *+[o][F]{}\ar@{.}[rr]_{i+1}&& *+[o][F]{}\ar@{-}[rr]_{i+2}&&*+[o][F]{}}\\
 
  \xymatrix@-1.5pc{*+[o][F]{}  \ar@{-}[rr]^{i-1} \ar@{-}[dd]_{i+1}&& *+[o][F]{}  \ar@{-}[dd]_{i+1}&&&&&&   *+[o][F]{}\ar@{-}[rr] \ar@{-}[dd]_{i+3}\ar@{-}[rr]^{i+1}&&*+[o][F]{}\ar@{-}[dd]^{i+3}\\
 &&&&&& &&&&\\
 *+[o][F]{}\ar@{-}[rr]_{i-1}&& *+[o][F]{}\ar@{.}[rr]_i&& *+[o][F]{}\ar@{.}[rr]_{i+1}&&*+[o][F]{}\ar@{-}[rr]_{i+2}&& *+[o][F]{}\ar@{-}[rr]_{i+1}&&*+[o][F]{}}\quad 
 \xymatrix@-1.5pc{*+[o][F]{}  \ar@{-}[rr]^{i-1} \ar@{-}[dd]_{i+1}&& *+[o][F]{}  \ar@{-}[dd]_{i+1}&&&&&&   *+[o][F]{}\ar@{-}[rr] \ar@{-}[dd]_{i-3}\ar@{-}[rr]^{i-1}&&*+[o][F]{}\ar@{-}[dd]^{i-3}\\
 &&&&&& &&&&\\
 *+[o][F]{}\ar@{-}[rr]_{i-1}&& *+[o][F]{}\ar@{.}[rr]_i&& *+[o][F]{}\ar@{.}[rr]_{i-1}&&*+[o][F]{}\ar@{-}[rr]_{i-2}&& *+[o][F]{}\ar@{-}[rr]_{i-1}&&*+[o][F]{}}
 \end{array}\]
\end{small}
The first two are excluded by the fact that we can reduce the number of squares.
The last two are excluded because they have at least three edges with the same label in the 2-fracture graph, a contradiction. Hence $\mathcal{F}$ has exactly two components one of which is an isolated vertex. The only possibility for $\mathcal G$ is then case (8) of Table~\ref{T2F} and situation (i) is fully analysed.

Suppose now we have situation (ii) and $\nu=1$.
If $v$ is at distance one from two components with a square then, as $\{w,v\}$ must be a double $(i\pm 1)$-edge, and both components must be as in (ii). Then we get the following graph.
$$ 
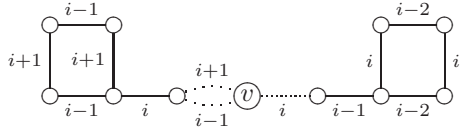
\xymatrix@-1.5pc{*+[o][F]{}  \ar@{-}[rr]^{i-1} \ar@{-}[dd]_{i+1}&& *+[o][F]{}  \ar@{-}[dd]_{i+1}&&&&&&&&   *+[o][F]{}\ar@{-}[rr] \ar@{-}[dd]_i\ar@{-}[rr]^{i-2}&&*+[o][F]{}\ar@{-}[dd]^i\\
 &&&&&&&& &&&&\\
 *+[o][F]{}\ar@{-}[rr]_{i-1}&& *+[o][F]{}\ar@{-}[rr]_i&& *+[o][F]{}\ar@/_{.3pc}/@{.}[rr]_{i-1}\ar@/^{.3pc}/@{.}[rr]^{i+1}&&*+[o][F]{v}\ar@{.}[rr]_i&& *+[o][F]{}\ar@{-}[rr]_{i-1}&& *+[o][F]{}\ar@{-}[rr]_{i-2}&&*+[o][F]{}}$$
But then the $2$-fracture graph has three $i$-edges, a contradiction.
Thus we have only one component with a square. Up to duality, the possibilities we get are graphs (5),  (6) or (7) of Table~\ref{T2F} with degree $7$ or the following graph with $n>7$.
$$ 
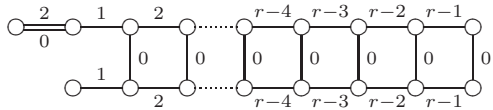
\xymatrix@-0.7pc{   *+[o][F]{}  \ar@{=}[r]^2_0 & *+[o][F]{} \ar@{-}[r]^1 & *+[o][F]{} \ar@{-}[r]^2 & *+[o][F]{} \ar@{.}[r] & *+[o][F]{}  \ar@{-}[r]^{r-4}& *+[o][F]{}  \ar@{-}[r]^{r-3}  & *+[o][F]{}  \ar@{-}[r]^{r-2}  & *+[o][F]{}  \ar@{-}[r]^{r-1}  & *+[o][F]{}   &\\
  & *+[o][F]{} \ar@{-}[r]^1 &*+[o][F]{} \ar@{-}[r]_2  \ar@{-}[u]_0&  *+[o][F]{} \ar@{-}[u]_0 \ar@{.}[r] & *+[o][F]{}  \ar@{-}[r]_{r-4} \ar@{-}[u]_0& *+[o][F]{}  \ar@{-}[r]_{r-3}\ar@{-}[u]_0 &  *+[o][F]{}  \ar@{-}[r]_{r-2}\ar@{-}[u]_0 &*+[o][F]{}  \ar@{-}[r]_{r-1}\ar@{-}[u]_0&  *+[o][F]{} \ar@{-}[u]_0  } $$
But this graph does not give a permutation representation of a string C-group for $n>7$ as $\Gamma_{\{0,1,2,3\}}$ is a sesqui-extension with respect to $\rho_0$ of the string C-group having the permutation representation graph (2) of Proposition~\ref{IPfails}. Hence by Proposition~\ref{sesqui1} and \ref{IPfails} the only possibility is graph (5) of Table~\ref{T2F} with  $n=7$.

Case $s=0$: In this case $\mathcal{F}$ is a tree. 
If any pair of adjacent edges of  $\mathcal{F}$ have consecutive labels, we  get a linear graph whose sequence of labels is $(0,1,0,1,2,3,2,3,\ldots,r-2,r-1,r-2,r-1)$. 
This sequence corresponds to a string C-group only if $n\in\{5,\,9\}$, giving graphs (9) and (10)  of Table~\ref{T2F}.
Indeed for $n>9$, $\Gamma_{\{0,1,2,3\}}$ is a sesqui-extension with respect to $\rho_0$ of a sesqui-extension with respect to $\rho_3$, of the sggi having the permutation representation graph (3) of Proposition~\ref{IPfails}.
 Thus when $n>9$, by Proposition~\ref{sesqui1} and \ref{IPfails} we get a sggi that does not satisfy the intersection property.

Suppose that $\mathcal{F}$ has adjacent edges with nonconsecutive labels.
Hence $\mathcal{G}$ contains an alternating square $q$.
By a similar argument as in the proof of Proposition~\ref{movingsquares} we may assume that the labels of $q$ are $i-1$ and $i+1$ and there exists an $i$-edge incident to $q$.
As $\mathcal{F}$ has exactly one component, three edges of $q$ belong to $\mathcal{F}$, and the label of an edge incident to $q$ must be consecutive either with $i-1$ or with $i+1$.
We may assume, up to duality, that $\mathcal{F}$ contains one of the following subgraphs.

$$
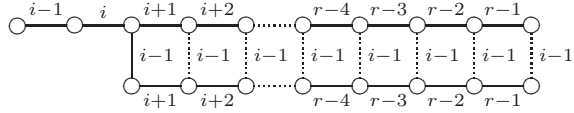
\xymatrix@-0.7pc{ *+[o][F]{} \ar@{-}[r]^{i-1} & *+[o][F]{} \ar@{-}[r]^i & *+[o][F]{} \ar@{-}[r]^{i+1}  & *+[o][F]{} \ar@{-}[r]^{i+2}& *+[o][F]{} \ar@{.}[r] & *+[o][F]{}  \ar@{-}[r]^{r-4}& *+[o][F]{}  \ar@{-}[r]^{r-3}  & *+[o][F]{}  \ar@{-}[r]^{r-2}  & *+[o][F]{}  \ar@{-}[r]^{r-1}  & *+[o][F]{}   &\\
 & &*+[o][F]{} \ar@{-}[r]_{i+1}  \ar@{-}[u]_{i-1}&  *+[o][F]{} \ar@{.}[u]_{i-1} \ar@{-}[r]_{i+2}  & *+[o][F]{} \ar@{.}[r]\ar@{.}[u]_{i-1} & *+[o][F]{}  \ar@{-}[r]_{r-4} \ar@{.}[u]_{i-1}& *+[o][F]{}  \ar@{-}[r]_{r-3}\ar@{.}[u]_{i-1} &  *+[o][F]{}  \ar@{-}[r]_{r-2}\ar@{.}[u]_{i-1} &*+[o][F]{}  \ar@{-}[r]_{r-1}\ar@{.}[u]_{i-1}&  *+[o][F]{} \ar@{.}[u]_{i-1}  }  
 $$
or 
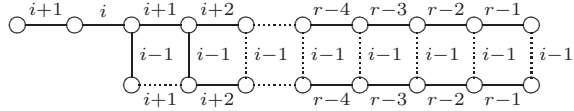
$$\xymatrix@-0.7pc{ *+[o][F]{} \ar@{-}[r]^{i+1}& *+[o][F]{} \ar@{-}[r]^i & *+[o][F]{} \ar@{-}[r]^{i+1}  & *+[o][F]{} \ar@{-}[r]^{i+2}& *+[o][F]{} \ar@{.}[r] & *+[o][F]{}  \ar@{-}[r]^{r-4}& *+[o][F]{}  \ar@{-}[r]^{r-3}  & *+[o][F]{}  \ar@{-}[r]^{r-2}  & *+[o][F]{}  \ar@{-}[r]^{r-1}  & *+[o][F]{}   &\\
 &&*+[o][F]{} \ar@{.}[r]_{i+1} \ar@{-}[u]_{i-1}& *+[o][F]{} \ar@{-}[u]_{i-1} \ar@{-}[r]_{i+2}  & *+[o][F]{} \ar@{.}[r]\ar@{.}[u]_{i-1} & *+[o][F]{}  \ar@{-}[r]_{r-4} \ar@{.}[u]_{i-1}& *+[o][F]{}  \ar@{-}[r]_{r-3}\ar@{.}[u]_{i-1} &  *+[o][F]{}  \ar@{-}[r]_{r-2}\ar@{.}[u]_{i-1} &*+[o][F]{}  \ar@{-}[r]_{r-1}\ar@{.}[u]_{i-1}&  *+[o][F]{} \ar@{.}[u]_{i-1}  }  
 $$

Now we consider two subcases: (a) there exists exactly one $i$-edge incident to $q$, or  (b) otherwise.

Suppose we have (a).
First, assume the only edges of $\mathcal{F}$ with nonconsecutive labels that are adjacent, are the edges of $q$. Then, up to duality, we get the graph (11) of Table~\ref{T2F} or the following graph which is a string C-group only if $n=7$ (graph (12) of  Table~\ref{T2F}). 
Indeed, when $n>7$, we get the following graph.
$$\xymatrix@-0.7pc{  *+[o][F]{} \ar@{-}[r]^1 & *+[o][F]{}  \ar@{-}[r]^2 & *+[o][F]{} \ar@{-}[r]^1 & *+[o][F]{} \ar@{-}[r]^2 & *+[o][F]{} \ar@{.}[r] & *+[o][F]{}  \ar@{-}[r]^{r-4}& *+[o][F]{}  \ar@{-}[r]^{r-3}  & *+[o][F]{}  \ar@{-}[r]^{r-2}  & *+[o][F]{}  \ar@{-}[r]^{r-1}  & *+[o][F]{}   &\\
 & & &*+[o][F]{} \ar@{-}[r]_2  \ar@{-}[u]_0&  *+[o][F]{} \ar@{-}[u]_0 \ar@{.}[r] & *+[o][F]{}  \ar@{-}[r]_{r-4} \ar@{-}[u]_0& *+[o][F]{}  \ar@{-}[r]_{r-3}\ar@{-}[u]_0 &  *+[o][F]{}  \ar@{-}[r]_{r-2}\ar@{-}[u]_0 &*+[o][F]{}  \ar@{-}[r]_{r-1}\ar@{-}[u]_0&  *+[o][F]{} \ar@{-}[u]_0  }  
 $$

But, by Propositions~\ref{sesqui1} and~\ref{IPfails},  this graph does not give a permutation representation of a string C-group when $n>7$ as $\Gamma_{\{0,1,2,3\}}$ is a sesqui-extension with respect to $\rho_0$ of the sggi having the permutation representation graph (4) of Proposition~\ref{IPfails}.

Next, assume that there is another pair of adjacent edges of $\mathcal{F}$ with nonconsecutive labels. By Proposition~\ref{movingsquares}, we get the following possibility.
$$
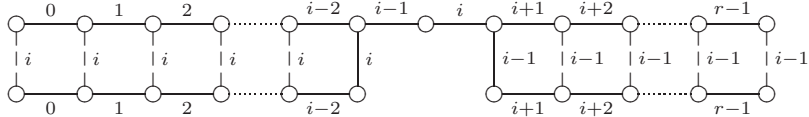
\xymatrix@-0.35pc{ *+[o][F]{}  \ar@{-}[r]^0 & *+[o][F]{} \ar@{-}[r]^1 & *+[o][F]{} \ar@{-}[r]^2 & *+[o][F]{} \ar@{.}[r] & *+[o][F]{}  \ar@{-}[r]^{i-2}   & *+[o][F]{}  \ar@{-}[r]^{i-1}  & *+[o][F]{}  \ar@{-}[r]^i& *+[o][F]{}  \ar@{-}[r]^{i+1}& *+[o][F]{}  \ar@{-}[r]^{i+2}&*+[o][F]{}  \ar@{.}[r]  & *+[o][F]{}  \ar@{-}[r]^{r-1}  & *+[o][F]{}   \\
*+[o][F]{}  \ar@{-}[r]_0 \ar@{--}[u]_i & *+[o][F]{} \ar@{-}[r]_1 \ar@{--}[u]_i & *+[o][F]{} \ar@{-}[r]_2  \ar@{--}[u]_i&  *+[o][F]{} \ar@{--}[u]_i \ar@{.}[r] & *+[o][F]{}  \ar@{-}[r]_{i-2} \ar@{--}[u]_i& *+[o][F]{}   \ar@{-}[u]_i&     &  *+[o][F]{}  \ar@{-}[r]_{i+1}\ar@{-}[u]_{i-1}& *+[o][F]{}  \ar@{-}[r]_{i+2}\ar@{--}[u]_{i-1} &*+[o][F]{} \ar@{.}[r]\ar@{--}[u]_{i-1}&*+[o][F]{}  \ar@{-}[r]_{r-1}\ar@{--}[u]_{i-1}&  *+[o][F]{} \ar@{--}[u]_{i-1} } 
 $$
 If $n=9$ we get the graph (5a) of Proposition~\ref{IPfails}  which is not a string C-group. If $n>9$, $\Gamma_{\{i-2,\,i-1,\,i,\, i+1\}}$ is a sesqui-extension  with respect to $\rho_{i-1}$ of a sggi $\Phi$ with permutation representation graph (5b). As in addition $\Gamma_{\{i-2,\,i-1,\,i,\, i+1\}}\cong \Phi$, by item (d) of  Proposition~\ref{sesqui} and Proposition~\ref{IPfails} the graph is not a permutation representation graph of a string C-group.

Suppose we have (b). Then there are more than one $i$-edge incident to $q$ and the only nonconsecutive edges of $\mathcal{F}$ which are adjacent are the edges of $q$.
Hence $\mathcal{G}$ must be, up to a duality,  one of the following graphs.
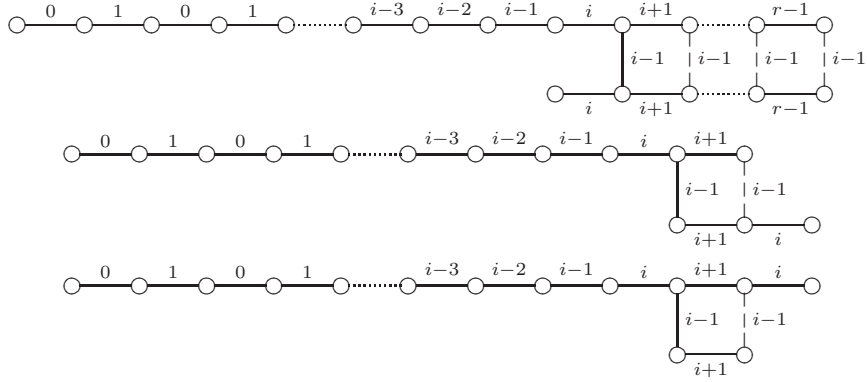
\[\begin{array}{c}
\xymatrix@-0.38pc{ *+[o][F]{}  \ar@{-}[r]^0& *+[o][F]{}  \ar@{-}[r]^1 &*+[o][F]{}  \ar@{-}[r]^0 & *+[o][F]{} \ar@{-}[r]^1 & *+[o][F]{} \ar@{.}[r]&*+[o][F]{}  \ar@{-}[r]^{i-3} &*+[o][F]{}  \ar@{-}[r]^{i-2} & *+[o][F]{} \ar@{-}[r]^{i-1} & *+[o][F]{}   \ar@{-}[r]^i & *+[o][F]{}  \ar@{-}[r]^{i+1}& *+[o][F]{}  \ar@{.}[r] &  *+[o][F]{}  \ar@{-}[r]^{r-1}  & *+[o][F]{}  \\
&& && && & &*+[o][F]{}  \ar@{-}[r]_i& *+[o][F]{}  \ar@{-}[r]_{i+1} \ar@{-}[u]_{i-1}& *+[o][F]{}  \ar@{.}[r]\ar@{--}[u]_{i-1} & *+[o][F]{}  \ar@{-}[r]_{r-1}\ar@{--}[u]_{i-1}&  *+[o][F]{} \ar@{--}[u]_{i-1}  }\\

\xymatrix@-0.38pc{ *+[o][F]{}  \ar@{-}[r]^0& *+[o][F]{}  \ar@{-}[r]^1 &*+[o][F]{}  \ar@{-}[r]^0 & *+[o][F]{} \ar@{-}[r]^1 & *+[o][F]{} \ar@{.}[r]&*+[o][F]{}  \ar@{-}[r]^{i-3} &*+[o][F]{}  \ar@{-}[r]^{i-2} & *+[o][F]{} \ar@{-}[r]^{i-1} & *+[o][F]{}   \ar@{-}[r]^i & *+[o][F]{}  \ar@{-}[r]^{i+1}& *+[o][F]{} &\\
&& && && & && *+[o][F]{}  \ar@{-}[r]_{i+1} \ar@{-}[u]_{i-1}& *+[o][F]{}  \ar@{-}[r]_i\ar@{--}[u]_{i-1}& *+[o][F]{}   } \\

\xymatrix@-0.38pc{ *+[o][F]{}  \ar@{-}[r]^0& *+[o][F]{}  \ar@{-}[r]^1 &*+[o][F]{}  \ar@{-}[r]^0 & *+[o][F]{} \ar@{-}[r]^1 & *+[o][F]{} \ar@{.}[r]&*+[o][F]{}  \ar@{-}[r]^{i-3} &*+[o][F]{}  \ar@{-}[r]^{i-2} & *+[o][F]{} \ar@{-}[r]^{i-1} & *+[o][F]{}   \ar@{-}[r]^i & *+[o][F]{}  \ar@{-}[r]^{i+1}& *+[o][F]{} \ar@{-}[r]^i &*+[o][F]{} \\
&& && && & && *+[o][F]{}  \ar@{-}[r]_{i+1} \ar@{-}[u]_{i-1}& *+[o][F]{} \ar@{--}[u]_{i-1}&   }
\end{array}\]

When $n=7$, we get the permutation  representation graphs (13),  (14) and (15) of Table~\ref{T2F}.
For $n>7$ only the permutation representation graph (13) is possible. Indeed for all the remaining possibilities
 $\Gamma_{\{i-2,\,i-1,\,i,\, i+1\}}$ is a sggi having one of the  permutation representation graphs (6a), (7) and (8) of Proposition~\ref{IPfails}.
\end{proof}

\begin{table}[htbp]

\[\begin{array}{|l|}
\hline
(1)\, C_2\times S_{n/2}:\; \xymatrix@-1.4pc{
*+[o][F]{}\ar@{--}[dd]_0\ar@{-}[rr]^{r-1}&&*+[o][F]{}\ar@{--}[dd]_0\ar@{.}[rr]&&*+[o][F]{}\ar@{-}[rr]^4\ar@{--}[dd]_0&&*+[o][F]{}\ar@{-}[rr]^3\ar@{--}[dd]_0&&*+[o][F]{}\ar@{-}[rr]^2 \ar@{-}[dd]_0 &&*+[o][F]{}   \ar@{-}[dd]^0\ar@{-}[rr]^1&& *+[o][F]{} \ar@{--}[dd]^0\\
&&&&&&&&&&&&&&\\
*+[o][F]{}\ar@{-}[rr]_{r-1}*+[o][F]{}&&*+[o][F]{}\ar@{.}[rr]&&*+[o][F]{}\ar@{-}[rr]_4&&*+[o][F]{}\ar@{-}[rr]_3&&*+[o][F]{}\ar@{-}[rr]_2&&*+[o][F]{} \ar@{-}[rr]_1&& *+[o][F]{} }\\
(2) \, C_2\wr S_{n/2}\, (\mbox{if }n/2 \mbox{ is even}) \mbox{ or } C_2^{n/2-1}:S_{n/2} \,(\mbox{if }n/2\mbox{ is odd}):\\
\hspace{3cm}\xymatrix@-1.4pc{
*+[o][F]{}\ar@{--}[dd]_0\ar@{-}[rr]^{r-1}&&*+[o][F]{}\ar@{--}[dd]_0\ar@{.}[rr]&&*+[o][F]{}\ar@{-}[rr]^5\ar@{--}[dd]_0&&*+[o][F]{}\ar@{-}[rr]^4\ar@{--}[dd]_0&&*+[o][F]{}\ar@{-}[rr]^3\ar@{--}[dd]_0&&*+[o][F]{}\ar@{-}[rr]^2 \ar@{-}[dd]_0 &&*+[o][F]{}   \ar@{-}[dd]^0\ar@{-}[rr]^1&& *+[o][F]{} \\
&&&&&&&&&&&&&&\\
*+[o][F]{}\ar@{-}[rr]_{r-1}*+[o][F]{}&&*+[o][F]{}\ar@{.}[rr]&&*+[o][F]{}\ar@{-}[rr]_5&&*+[o][F]{}\ar@{-}[rr]_4&&*+[o][F]{}\ar@{-}[rr]_3&&*+[o][F]{}\ar@{-}[rr]_2&&*+[o][F]{} \ar@{-}[rr]_1&& *+[o][F]{} }\\

(3)\,PGL(2,3):  \;  \xymatrix@-0.7pc{   *+[o][F]{}  \ar@{-}[r]^1 & *+[o][F]{}  \ar@{-}[r]^2 & *+[o][F]{} &  \\
   &  *+[o][F]{}  \ar@{-}[r]_2\ar@{-}[u]_0 &*+[o][F]{}\ar@{-}[r]_1 \ar@{-}[u]_0&*+[o][F]{}}  \\
   
 (4) \, S_7: \; \xymatrix@-0.7pc{  *+[o][F]{} \ar@{-}[r]^1 & *+[o][F]{}  \ar@{==}[r]^2_0 & *+[o][F]{} \ar@{-}[r]^1 & *+[o][F]{} \ar@{-}[r]^2 & *+[o][F]{} \\
 & & &*+[o][F]{} \ar@{-}[r]_2  \ar@{-}[u]_0&  *+[o][F]{} \ar@{-}[u]_0 } \\

(5)\, S_7:  \; \xymatrix@-0.7pc{   *+[o][F]{}  \ar@{==}[r]^2_0 & *+[o][F]{} \ar@{-}[r]^1 & *+[o][F]{} \ar@{-}[r]^2 & *+[o][F]{} \\
  & *+[o][F]{} \ar@{-}[r]^1 &*+[o][F]{} \ar@{-}[r]_2  \ar@{-}[u]_0&  *+[o][F]{} \ar@{-}[u]_0 }\quad
   
 (6)\, S_7:  \; \xymatrix@-0.7pc{   *+[o][F]{}  \ar@{==}[r]^2_0 & *+[o][F]{} \ar@{-}[r]^1 & *+[o][F]{} \ar@{-}[r]^2 & *+[o][F]{}& \\
  & &*+[o][F]{} \ar@{-}[r]_2  \ar@{-}[u]_0&  *+[o][F]{} \ar@{-}[u]_0\ar@{-}[r]^1 & *+[o][F]{}  }\quad

   (7)\, S_7:  \; \xymatrix@-0.7pc{   *+[o][F]{}  \ar@{==}[r]^2_0 & *+[o][F]{} \ar@{-}[r]^1 & *+[o][F]{} \ar@{-}[r]^2 & *+[o][F]{}\ar@{-}[r]^1 & *+[o][F]{}  \\
  & &*+[o][F]{} \ar@{-}[r]_2  \ar@{-}[u]_0&  *+[o][F]{} \ar@{-}[u]_0 &}\\[5pt]

(8) \,S_7:\; \xymatrix@-0.3pc{   *+[o][F]{}  \ar@{-}[r]^1 & *+[o][F]{}  \ar@{-}[r]^0 & *+[o][F]{}  \ar@{-}[r]^1 & *+[o][F]{}   &\\
   &  *+[o][F]{}  \ar@{-}[r]_0\ar@{-}[u]_2 &*+[o][F]{}  \ar@{--}[r]_1\ar@{-}[u]_2&  *+[o][F]{}   } \\
   
 (9)\, D_{5}: \; \xymatrix@-0.7pc{ *+[o][F]{}  \ar@{-}[r]^0 & *+[o][F]{}  \ar@{-}[r]^1 &*+[o][F]{}  \ar@{-}[r]^0 & *+[o][F]{} \ar@{-}[r]^1 & *+[o][F]{}  }  \quad

 (10)\,A_9: \; \xymatrix@-0.7pc{ *+[o][F]{}  \ar@{-}[r]^0 & *+[o][F]{}  \ar@{-}[r]^1 &*+[o][F]{}  \ar@{-}[r]^0 & *+[o][F]{} \ar@{-}[r]^1 & *+[o][F]{}  \ar@{-}[r]^2 & *+[o][F]{} \ar@{-}[r]^3 & *+[o][F]{}   \ar@{-}[r]^2  & *+[o][F]{}  \ar@{-}[r]^3  & *+[o][F]{}  } \\[5pt]

 (11)\, A_n\,{\emph or } \,S_n\, (n\geq 7): \; \xymatrix@-0.35pc{  *+[o][F]{} \ar@{-}[r]^1 & *+[o][F]{}  \ar@{-}[r]^0 & *+[o][F]{} \ar@{-}[r]^1 & *+[o][F]{} \ar@{-}[r]^2 & *+[o][F]{} \ar@{.}[r]   & *+[o][F]{}  \ar@{-}[r]^{r-2}  & *+[o][F]{}  \ar@{-}[r]^{r-1}  & *+[o][F]{}   &\\
 & & &*+[o][F]{} \ar@{-}[r]_2  \ar@{-}[u]_0&  *+[o][F]{} \ar@{--}[u]_0 \ar@{.}[r]  &  *+[o][F]{}  \ar@{-}[r]_{r-2}\ar@{--}[u]_0 &*+[o][F]{}  \ar@{-}[r]_{r-1}\ar@{--}[u]_0&  *+[o][F]{} \ar@{--}[u]_0  } \\

  (12) \,S_7:\; \xymatrix@-0.7pc{  *+[o][F]{} \ar@{-}[r]^1 & *+[o][F]{}  \ar@{-}[r]^2 & *+[o][F]{} \ar@{-}[r]^1 & *+[o][F]{} \ar@{--}[r]^2 & *+[o][F]{} \\
 & & &*+[o][F]{} \ar@{-}[r]_2  \ar@{-}[u]_0&  *+[o][F]{} \ar@{-}[u]_0} \\

 (13)\, A_n\,{\emph or } \,S_n \, (n\geq 7): \; \xymatrix@-0.35pc{ *+[o][F]{}  \ar@{-}[r]^0 & *+[o][F]{} \ar@{-}[r]^1 & *+[o][F]{} \ar@{-}[r]  \ar@{-}[r]^2 & *+[o][F]{}  \ar@{.}[r] & *+[o][F]{}  \ar@{-}[r]^{r-2}  & *+[o][F]{}  \ar@{-}[r]^{r-1}  & *+[o][F]{}  &\\
&*+[o][F]{}  \ar@{-}[r]_1 & *+[o][F]{} \ar@{-}[u]_0  \ar@{-}[r]_2  & *+[o][F]{}  \ar@{.}[r] \ar@{--}[u]_0 &  *+[o][F]{}  \ar@{-}[r]_{r-2}\ar@{--}[u]_0&*+[o][F]{}  \ar@{-}[r]_{r-1}\ar@{--}[u]_0&  *+[o][F]{} \ar@{--}[u]_0  } \\

 (14) \,S_7:\; \xymatrix@-0.35pc{ *+[o][F]{}  \ar@{-}[r]^0 &   *+[o][F]{}  \ar@{-}[r]^1 & *+[o][F]{}  \ar@{-}[r]^2 & *+[o][F]{}   &\\
 &  &  *+[o][F]{}  \ar@{-}[r]_2\ar@{--}[u]_0 &*+[o][F]{}  \ar@{-}[r]_1\ar@{-}[u]_0&  *+[o][F]{} } \quad

 (15) \,S_7:\; \xymatrix@-0.35pc{ *+[o][F]{}  \ar@{-}[r]^0 &   *+[o][F]{}  \ar@{-}[r]^1 & *+[o][F]{}  \ar@{-}[r]^2 & *+[o][F]{}   \ar@{-}[r]^1 & *+[o][F]{} \\
 &  &  *+[o][F]{}  \ar@{-}[r]_2\ar@{--}[u]_0 &*+[o][F]{} \ar@{-}[u]_0& } \\
   
\hline
\end{array}\]
\caption{Permutation representation graphs of transitive string C-groups of rank $r\geq\frac{n-1}{2}$ having a 2-fracture graph (edges not in the 2-fracture graph are dashed).} \label{T2F}
\end{table}


\begin{prop}\label{11}
Let $n\geq 7$.  The graph (11) of Table~\ref{T2F} gives a string C-group for $A_n$, when $n\equiv 1 \mod 4$, and for  $S_n$, when $n\equiv 3\mod 4$ .
\end{prop}
\begin{proof}
Let $\Gamma=(G,S)$ be the sggi with the permutation representation graph (11) of Table~\ref{T2F}.
For $n=7$, $\Gamma$ is a string C-group for $S_7$. For $n=9$, $\Gamma$ is a string C-group for $A_9$. Both can be found in the Atlas of abstract regular polytopes for small almost simple groups~\cite{LVatlas}.
For $n> 9$,  $G$ is primitive and contains a cycle of order $5$, namely $(\rho_0\rho_1)^5$, hence $G$ is isomorphic to $S_n$ when $n\equiv 3\mod 4$ and $G$ is isomorphic to $A_n$ when $n\equiv 1\mod 4$.

By~\cite[Lemma 6.2]{flm2}, $\Gamma_0$ is a sesqui-extension of a string C-group with respect to $\rho_1$.
Moreover, by Lemma~\ref{sesqui}, $G_0\cong S_{\frac{n-3}{2}}\times S_{\frac{n-1}{2}}$. 
We assume by induction that $\Gamma_{r-1}$ is a sesqui-extension of a string C-group. By Proposition~\ref{arp} it remains to prove that $G_0\cap G_{r-1}=G_{0,r-1}$.
We have that $G_0\cap G_{r-1}\subseteq C_2\times S_{\frac{n-5}{2}}\times S_{ \frac{n-3}{2}}$, where $C_2$ is generated by a transposition that does not belong to $G_0$. Hence $G_0\cap G_{r-1}\subseteq G_{0,r-1}\cong S_{\frac{n-5}{2}}\times S_{\frac{n-3}{2}}$.
\end{proof}

\begin{prop}
All the graphs of Table~\ref{T2F} are permutation representation graphs of string C-groups.
\end{prop}
\begin{proof}
Whenever $r=3$, we may use Theorem 4.1 of \cite{CO}. The remaining sggi of small degree were computed with {\sf Gap}~\cite{GAP4} and they satisfy the intersection property.
The graph (11) gives a string C-group by Proposition~\ref{11} while graphs (1), (2) and (13) are known to be string C-groups by~\cite[Corollary 4.19]{2017CFLM}, \cite[Corollary 5.11 and 5.12]{flm}.
\end{proof}

\section{Transitive string C-groups having a fracture graph with a split}\label{SSplit}

Let $G$ be a permutation group of degree $n$. Let $S=\{\rho_0, \ldots, \rho_{r-1}\}$ be a set of involutions of $G$ such that $(G,S)$ is a sggi.
Let $\mathcal G$ be the permutation representation graph of $\Gamma = (G,S)$.
Suppose $\Gamma$ has a fracture graph.

We say that $i$ {\em is the label of a split} (or a $i$-split) if
\begin{itemize}
\item $G_i$ has exactly one more orbit than $G$; then there exists a partition of $\{1,\ldots,n\}$ into two sets $O_1$ and $O_2$, of sizes $n_1$ and $n_2$ ($n = n_1+n_2$) such that $\rho_i$ is the unique permutation swapping vertices in different blocks of this partition of  $\{1,\ldots,n\}$;
\item there is exactly one pair of points $(a,b) \in O_1\times O_2$ such that $a\rho_i=b$.
\end{itemize}
The pair $\{a,b\}$ is then called a {\em split} or $i$-split of $\Gamma$.

Suppose that $\Gamma$ has an $i$-split $\{a,b\}$. 
For $j\neq i$, $\rho_j=\alpha_j\beta_j$ where $\alpha_j$ acts on $O_1$ and $\beta_j$ acts on $O_2$, and $\rho_i=\alpha_i\beta_i(a,b)$ where $\alpha_i$ acts on $O_1$ and $\beta_i$ acts on $O_2$.
Let $J_A := \{j \in \{0, \ldots, r-1\}\setminus\{i\}\mid\alpha_j \neq 1_G\}$ and $J_B := \{j \in \{0, \ldots, r-1\}\setminus\{i\}\mid \beta_j \neq 1_G\}$. We then have $A=\langle \alpha_j\mid j\in J_A\rangle$ and $B=\langle \beta_j\mid j\in J_B\rangle $. If one of the groups is trivial then the corresponding set of indices is empty.
If $\alpha_j = 1_G$ for each $j\in \{i+1,\ldots, r-1\}$ and
$\beta_j = 1_G$ for each $j\in \{0,\ldots, i-1\}$
, then we say that the $i$-split $\{a,b\}$ is \emph{perfect}.

\begin{prop}\label{primpb}
If $G$ is transitive and $\Gamma$ has a perfect $i$-split, then $G$ is primitive.
\end{prop}

\begin{proof}
Suppose that $\rho_j$ moves $a$, with $j\ne i$. Then necessarily
$j\in\{i-1,i+1\}$, since otherwise there is an $\{i,j\}$ square, whose remaining
vertices must lie in different $G_i$-orbits and must be swapped by
$\rho_i$. Now our assumption shows that $\rho_{i-1}$ moves $a$, and similarly
$\rho_{i+1}$ moves $b$.

Suppose that a block of imprimitivity contains points in both $O_1$ and $O_2$.
Then it is fixed by both $A$ and $B$, and hence is the whole of
$\{1,\ldots,n\}$, a contradiction. So, in particular, $a$ and $b$ lie in
different blocks, which are interchanged by $\rho_i$. But then there is more
than one $i$-edge linking the orbits $O_1$ and $O_2$, contrary to assumption.
\end{proof}

In what follows we consider the case where $G$ is transitive and $\Gamma$ has an $i$-split. For this $i$-split we consider the notation introduced above.
In addition we assume that all maximal parabolic subgroups of $\Gamma$ (that is all subgroups $G_j$ with $j=0,\ldots, r-1$) are intransitive. 
In what follows we determine an upper bound for the rank of $\Gamma$ when the $i$-split is not perfect.
We first deal with the case where $A$ and $B$ are imprimitive, and after we assume that $B$ is either primitive or trivial.


\subsection{$A$ and $B$ imprimitive}

Let $A$ be embedded into $S_{k_1}\wr S_{m_1}$ and $B$ be embedded into $S_{k_2}\wr S_{m_2}$ with $n_1=m_1k_1$ and $n_2=m_2k_2$. 
Consider a minimal subset $M$ of the set of generators of $G_i$ generating the group induced on the two block systems. 
Let $R$ be the set containing the remaining generators of $G_i$.  

Consider the permutation representation graph  $\mathcal{X}$  for the block action, that is, a graph having  $m_1+m_2$ vertices, corresponding to the blocks, and with a $j$-edge between two blocks whenever $\rho_j$ swaps them.
As $G$ is transitive and has an $i$-split, $G_i$ has exactly two orbits and the graph $\mathcal{X}$ has two connected components.   
Also, consider the subgraph $\mathcal{\bar{X}}$ of $\mathcal{X}$ with the same vertices and with one $j$-edge for each element of $M$, between blocks in different $G_j$-orbits. This is a fracture subgraph of $\mathcal{X}$, particularly $\mathcal{\bar{X}}$ has no cycles. Hence $|M| \leq m_1+m_2- 2 $.

Similarly, consider the graph $\mathcal{Y}$ with $k$ vertices corresponding to the $\langle M\rangle$-orbits, with  a $j$-edge between a pair of $\langle M\rangle$-orbits $L$ and $L'$  whenever there is  $x\in L$ such that $x\rho_j\in L'$ with $\rho_j\in R$. Let $\mathcal{\bar{Y}}$ be a fracture subgraph of  $\mathcal{Y}$  having only one $j$-edge for each element $\rho_j\in R$ between $\langle M\rangle$-orbits in different $G_j$-orbits. As before, $\mathcal{\bar{Y}}$ has no cycles and has at least two components. Hence $|R| \leq k - 2$. Note that no element of $R$ can fix all the $\langle M \rangle$-orbits. Suppose on the contrary that $\rho_j \in R$ fixes all $\langle M\rangle$-orbits. Then for each $j$-edge of $\mathcal G$, there is a path with edges with labels corresponding to the elements of $M$ connecting the two vertices of that $j$-edge. Hence $\mathcal G$ does not have a fracture graph as $G_j$ is transitive, a contradiction.
Note also that each $\rho_j$ must have a transposition corresponding to an edge belonging to a fracture graph. Thus it must connect two different $\langle M\rangle$-orbits. Hence for each element $\rho_j\in R$ there is one $j$-edge of the fracture graph $\mathcal{\bar{Y}}$. As $\mathcal{\bar{Y}}$ is a forest with at least two components,  $|R| \leq k- 2$.   In addition  $k\leq k_1+k_2$, therefore $|R| \leq k_1+k_2- 2$.

In what follows let $L_a$ be the $\langle M\rangle$-orbit containing $a$ and $L_b$ the $\langle M\rangle$-orbit containing $b$.

In \cite{2017CFLM} the authors restricted their study to the case where $G$ is even, that is a subgroup of $A_n$. However, for the following three propositions this assumption was never used.


\begin{prop}\cite[Proposition 5.2]{2017CFLM}\label{M=m-2}
If $|M|= m_1+m_2-2$ then  $\mathcal{X}$ has two connected components and consecutive labels. Up to a duality,  $\mathcal{X}$ is the following graph.
$$\xymatrix@-1.3pc{ *+[F]{} \ar@{-}[rr]^{i-m_1+1} && *+[F]{ }  \ar@{.}[rr] && *+[F]{ }  \ar@{-}[rr]^{i-2} &&*+[F]{}  \ar@{-}[rr]^{i-1} && *+[F]{a} &&  *+[F]{b}  \ar@{-}[rr]^{i+1} && *+[F]{ } \ar@{-}[rr]^{i+2} && *+[F]{ }  \ar@{.}[rr]&&*+[F]{ } \ar@{-}[rr] ^{i+m_2-1}&&*+[F]{ }} $$
\end{prop}

\begin{prop}\cite[Proposition 5.3]{2017CFLM}\label{M=m-3(1)}
If $|M|= m_1+m_2-3$ then, up to duality, either $m_1=2$ or $m_1\geq 4$,  accordantly to one of the following graphs.
$$(1) \,\xymatrix@-1.2pc{  *+[F]{} \ar@{-}[rr]^{i-1} &&*+[F]{a}  &&*+[F]{b} \ar@{-}[rr]^{i-1} && *+[F]{ }  \ar@{-}[rr]^{i-2}& &*+[F]{} \ar@{.}[rr] &&*+[F]{} \ar@{-}[rr] ^{i-m_2-1}&&*+[F]{ } } $$
$$(2) \,\xymatrix@-1.3pc{  *+[F]{} \ar@{-}[rrr]^{i+m_1-3} &&&*+[F]{} \ar@{-}[rrr]^{i+m_1-2} &&&*+[F]{} \ar@{-}[rrr]^{i+m_1-3} &&&*+[F]{} \ar@{.}[rr] &&*+[F]{} \ar@{-}[rr] ^{i+1}&&*+[F]{a }  & &*+[F]{b} \ar@{-}[rr]^{i-1} &&*+[F]{}  \ar@{.}[rr]&&*+[F]{} \ar@{-}[rr] ^{i-m_2-1}&&*+[F]{ }} $$
\end{prop}
\begin{prop}\cite[Proposition 5.4]{2017CFLM}\label{R<=k-3}
If an element of $R$ interchanges points in more than one pair of $\langle M\rangle$-orbits, then $|R|\leq k -3$.
\end{prop}

Now let us compute an upper bound for the rank of $\Gamma$ in the case where $A$ and $B$ are imprimitive.  The following proposition improves~\cite[Proposition 5.7]{2017CFLM}. Parts of its proofs follow the proof of the latter reference.

\begin{prop}\label{imp}
Suppose that $i$ is label of a split of $\Gamma$ that is not perfect.
If $A$ and $B$ are both imprimitive then $r\leq \frac{n-2}{2}$ or $\Gamma$ has, up to duality, one of the following permutation representation graphs.

\begin{tiny}
$$\xymatrix@-1.57pc{*+[o][F]{} \ar@{-}[rr]^3\ar@{-}[d]_5&&*+[o][F]{} \ar@{-}[rr]^4\ar@{-}[d]^{5} &&*+[o][F]{} \ar@{-}[rr]^3&&*+[o][F]{}  \ar@{-}[rr]^2&&*+[o][F]{}\ar@{-}[d]^1\\
*+[o][F]{} \ar@{-}[rr]^3\ar@{-}[d]_6&& *+[o][F]{}\ar@{-}[d]^6&& *+[o][F]{} \ar@{-}[rr]^3\ar@{-}[d]_0&&  *+[o][F]{}*\ar@{-}[d]^0 \ar@{-}[rr]^2&&*+[o][F]{} \ar@{-}[d]^0\\
*+[o][F]{} \ar@{-}[rr]_3\ar@{.}[d]]&& *+[o][F]{}\ar@{.}[d]&&*+[o][F]{}\ar@{-}[rr]_3&&*+[o][F]{} \ar@{-}[rr]_2&&*+[o][F]{} \\
*+[o][F]{} \ar@{-}[rr]_3\ar@{-}[d]_{r-1}&&*+[o][F]{} \ar@{-}[d]^{r-1}&& &&&&\\
*+[o][F]{}\ar@{-}[rr]_3&& *+[o][F]{}&&&& &&&&}
\,
\xymatrix@-1.57pc{*+[o][F]{} \ar@{-}[rr]^{i-1}\ar@{-}[d]_{i+1}&&*+[o][F]{} \ar@{-}[rr]^i\ar@{-}[d]^{i+1} &&*+[o][F]{} \ar@{-}[rr]^{i-1}&&*+[o][F]{}\ar@{-}[d]^{i-2}\\
*+[o][F]{} \ar@{-}[rr]{i-1}\ar@{-}[d]_{i+2}&& *+[o][F]{}\ar@{-}[d]_{i+2}&& *+[o][F]{} \ar@{-}[rr]^{i-1}\ar@{-}[d]_{i-3}&& *+[o][F]{}*\ar@{-}[d]^{i-3}\\
*+[o][F]{} \ar@{-}[rr]_{i-1}\ar@{.}[d]]&& *+[o][F]{}\ar@{.}[d]&&*+[o][F]{}\ar@{.}[d]\ar@{-}[rr]^{i-1}&&*+[o][F]{}\ar@{.}[d]\\
*+[o][F]{} \ar@{-}[rr]_{i-1}\ar@{-}[d]_{r-1}&&*+[o][F]{} \ar@{-}[d]^{r-1}&&*+[o][F]{}\ar@{-}[d]^0\ar@{-}[rr]^{i-1}&&*+[o][F]{}\ar@{-}[d]^0\\
*+[o][F]{}\ar@{-}[rr]_{i-1}&& *+[o][F]{}&&*+[o][F]{}\ar@{-}[rr]_{i-1} &&*+[o][F]{} }
\,
 \xymatrix@-1.57pc{*+[o][F]{} \ar@{-}[rr]^{i-1}\ar@{-}[d]_{i+1}&&*+[o][F]{} \ar@{-}[rr]^i\ar@{-}[d]^{i+1} &&*+[o][F]{} \ar@{-}[rr]^{i-1} &&*+[o][F]{} \ar@{.}[rr]&&*+[o][F]{}\ar@{-}[rr]^1&&*+[o][F]{}\ar@{-}[d]^0\\
*+[o][F]{} \ar@{-}[rr]_{i-1}\ar@{-}[d]_{i+2}&& *+[o][F]{}\ar@{-}[d]^{i+2}&& *+[o][F]{} \ar@{-}[rr]^{i-1}&& *+[o][F]{} \ar@{.}[rr]&&*+[o][F]{} \ar@{-}[rr]^1&& *+[o][F]{}*\\
*+[o][F]{} \ar@{-}[rr]_{i-1}\ar@{.}[d]]&& *+[o][F]{}\ar@{.}[d]&&&&&&&&\\
*+[o][F]{} \ar@{-}[rr]_{i-1}\ar@{-}[d]_{r-1}&&*+[o][F]{} \ar@{-}[d]^{r-1}&&&& &&&&\\
*+[o][F]{}\ar@{-}[rr]_{i-1}&& *+[o][F]{} &&&& && && &&}
$$
\end{tiny}

\end{prop}
\begin{proof}
We have $r=|M|+|R|+1\leq (m_1+m_2- 2)+ (k_1+k_2- 2)+1$. As $ (m_1+m_2- 2)+ (k_1+k_2- 2)+1\leq \frac{m_1k_1+m_2k_2+2}{2} \Leftrightarrow (m_1-2)(k_1-2)+(m_2-2)(k_2-2)\geq 0$, we get $r \leq \frac{n+2}{2}$. 

Note that, if $r=|M|+|R|+1\leq (m_1+m_2- 2)+ (k_1+k_2- 2)+1-2$, then $r \leq \frac{n-2}{2}$, as wanted.
Hence if we can improve the bound of $|M| + |R|$ down by two, we are done.

Suppose that  $|M|=m_1+m_2- 2$. In this case $\mathcal{X}$ is as in Proposition~\ref{M=m-2}. In particular the set of labels of $M\cup \{\rho_i\}$ is an interval $\{i-m_1+1,\ldots,i+m_2-1\}$.
Now consider the graph $\mathcal{Y}$. The only way to connect $L_a$ to another $\langle M\rangle$-orbit is with and $(i-m_1)$-edge and the only way to connect $L_b$ to another  $\langle M\rangle$-orbit is with and $(i+m_2)$-edge.
As $i$ is the label of a non-perfect split there exists $\rho_h\in R$ acting nontrivially on both $G_i$-orbits. Suppose without loss of generality that $h>i$. Then as the set of labels of  the generators of $R$ acting on the first $G_i$-orbit must contain $i-m_1$ and $h$, the set of labels of $R$ is not an interval. We then have that $\bar{\mathcal{Y}}$ has at least three components, namely at least two in the first $G_i$ orbit and at least one in the second. Hence $|R|\leq k_1+k_2- 3$. Moreover if we have equality, $|R|= k_1+k_2- 3$ then  $\bar{\mathcal{Y}}$ has exactly three components, and there is only one possibility for the graph $\mathcal{Y}$, which is the following.
$$ \xymatrix@-1.57pc{*++[][F]{\;L_a\;} \ar@{-}[dd]_{i-m_1}&&*++[][F]{L_b} \ar@{-}[dd]^{i+m_2}\\
&&\\
*++[][F]{\quad\quad} \ar@{-}[dd]_{i-m_1-1}&&*++[][F]{\quad\quad} \ar@{-}[dd]^{i+m_2+1}\\
&&\\
*++[][F]{\quad\quad} \ar@{.}[dd]&&*++[][F]{\quad\quad} \ar@{.}[dd]\\
&&\\
*++[][F]{\quad\quad} \ar@{-}[dd]_1&&*++[][F]{\quad\quad} \ar@{-}[dd]^{r-1}\\
&&\\
*++[][F]{\quad\quad} \ar@{-}[dd]_0&&*++[][F]{\quad\quad} \\
&&\\
*++[][F]{\quad\quad} \ar@{-}[dd]_1&&\\
&&\\
*++[][F]{\quad\quad} &&
}$$
But then $J_A=\{0,\ldots, i-1\}$ and $J_B=\{i+1,\ldots,r-1\}$, hence $h\not\in J_A$, a contradiction. 
Thus if $|M| = m_1+m_2- 2$ then $|R|\leq k_1+k_2-4$.

Now assume that $|M|=m_1+m_2- 3$ and $|R|=k_1+k_2-2$. Recall that $k$ is the number of vertices of $\mathcal Y$.
If $k< k_1+k_2$ then $|R|\leq k_1+k_2-3$. Hence $k=k_1+k_2$, $|R| = k-2$ and by Proposition~\ref{R<=k-3}, each element of $R$ swaps at most two $\langle M\rangle$-orbits. Then $\mathcal{\bar{Y}}$ has exactly two components and there are at least two $\langle M\rangle$-orbits in each $\Gamma_i$-orbit. 
There are, up to duality, the two possibilities for $\mathcal{X}$ given in  Proposition~\ref{M=m-3(1)}. 

If $\mathcal X$ is graph (1) of Proposition~\ref{M=m-3(1)}, there is only one possibility to connect $L_a$ to another $\langle M\rangle$-orbit, that is using an $(i+1)$-edge. Then, by Proposition~\ref{R<=k-3}, there cannot be $(i+1)$-edges in the other $G_i$-orbit.

In this case, as $m_1=2$, we have  $r\leq \frac{n-2}{2}\Leftrightarrow (m_2-2)(k_2-2)\geq 2$.
Thus $r> \frac{n-2}{2}$ only when $m_2=k_2=3$, $m_2=2$ or $k_2=2$ corresponding, respectively, to the following three possibilities for $\mathcal{G}$.
\begin{tiny}
$$\xymatrix@-1.57pc{*+[o][F]{} \ar@{-}[rr]^3\ar@{-}[d]_5&&*+[o][F]{} \ar@{-}[rr]^{i=4}\ar@{-}[d]^{5} &&*+[o][F]{} \ar@{-}[rr]^3&&*+[o][F]{}  \ar@{-}[rr]^2&&*+[o][F]{}\ar@{-}[d]^{l=1}\\
*+[o][F]{} \ar@{-}[rr]^3\ar@{-}[d]_6&& *+[o][F]{}\ar@{-}[d]^6&& *+[o][F]{} \ar@{-}[rr]^3\ar@{-}[d]_0&&  *+[o][F]{}*\ar@{-}[d]^0 \ar@{-}[rr]^2&&*+[o][F]{} \ar@{-}[d]^0\\
*+[o][F]{} \ar@{-}[rr]_3\ar@{.}[d]]&& *+[o][F]{}\ar@{.}[d]&&*+[o][F]{}\ar@{-}[rr]_3&&*+[o][F]{} \ar@{-}[rr]_2&&*+[o][F]{} \\
*+[o][F]{} \ar@{-}[rr]_3\ar@{-}[d]_{r-1}&&*+[o][F]{} \ar@{-}[d]^{r-1}&& &&&&\\
*+[o][F]{}\ar@{-}[rr]_3&& *+[o][F]{}&&&& &&&&}
\,
\xymatrix@-1.57pc{*+[o][F]{} \ar@{-}[rr]^{i-1}\ar@{-}[d]_{i+1}&&*+[o][F]{} \ar@{-}[rr]^i\ar@{-}[d]^{i+1} &&*+[o][F]{} \ar@{-}[rr]^{i-1}&&*+[o][F]{}\ar@{-}[d]^{l=i-2}\\
*+[o][F]{} \ar@{-}[rr]^{i-1}\ar@{-}[d]_{i+2}&& *+[o][F]{}\ar@{-}[d]_{i+2}&& *+[o][F]{} \ar@{-}[rr]^{i-1}\ar@{-}[d]_{i-3}&& *+[o][F]{}*\ar@{-}[d]^{i-3}\\
*+[o][F]{} \ar@{-}[rr]_{i-1}\ar@{.}[d]]&& *+[o][F]{}\ar@{.}[d]&&*+[o][F]{}\ar@{.}[d]\ar@{-}[rr]^{i-1}&&*+[o][F]{}\ar@{.}[d]\\
*+[o][F]{} \ar@{-}[rr]_{i-1}\ar@{-}[d]_{r-1}&&*+[o][F]{} \ar@{-}[d]^{r-1}&&*+[o][F]{}\ar@{-}[d]^0\ar@{-}[rr]^{i-1}&&*+[o][F]{}\ar@{-}[d]^0\\
*+[o][F]{}\ar@{-}[rr]_{i-1}&& *+[o][F]{}&&*+[o][F]{}\ar@{-}[rr]_{i-1} &&*+[o][F]{} }
\,
 \xymatrix@-1.57pc{*+[o][F]{} \ar@{-}[rr]^{i-1}\ar@{-}[d]_{i+1}&&*+[o][F]{} \ar@{-}[rr]^i\ar@{-}[d]^{i+1} &&*+[o][F]{} \ar@{-}[rr]^{i-1} &&*+[o][F]{} \ar@{.}[rr]&&*+[o][F]{}\ar@{-}[rr]^1&&*+[o][F]{}\ar@{-}[d]^{l=0}\\
*+[o][F]{} \ar@{-}[rr]_{i-1}\ar@{-}[d]_{i+2}&& *+[o][F]{}\ar@{-}[d]^{i+2}&& *+[o][F]{} \ar@{-}[rr]^{i-1}&& *+[o][F]{} \ar@{.}[rr]&&*+[o][F]{} \ar@{-}[rr]^1&& *+[o][F]{}*\\
*+[o][F]{} \ar@{-}[rr]_{i-1}\ar@{.}[d]]&& *+[o][F]{}\ar@{.}[d]&&&&&&&&\\
*+[o][F]{} \ar@{-}[rr]_{i-1}\ar@{-}[d]_{r-1}&&*+[o][F]{} \ar@{-}[d]^{r-1}&&&& &&&&\\
*+[o][F]{}\ar@{-}[rr]_{i-1}&& *+[o][F]{} &&&& && && &&}
$$
\end{tiny}
If $\mathcal{X}$ is graph (2) of Proposition~\ref{M=m-3(1)}, there either $i$ is the label of a perfect split or, to connect  $L_a$ to another $\langle M\rangle$-orbit, we need to use a single edge with label $f=i+m_1-2$ or $f=i+m_1-4$. But then $\rho_f\in M$, a contradiction.
\end{proof}


\subsection{$B$ is either primitive or trivial}\label{Bprim}

\begin{table}[htbp]
\begin{tiny}
 \[\begin{array}{|l|} 
\hline
 (I) \xymatrix@-1.7pc{ *+[o][F]{}\ar@{-}[rr]^{r-1} \ar@{-}[d]_{0} &&  *+[o][F]{}\ar@{-}[d]_{0} \ar@{.}[rr] &&  *+[o][F]{}\ar@{-}[rr]^2\ar@{-}[d]_0 &&  *+[o][F]{}\ar@{-}[rr]^{1} \ar@{-}[d]_{0} && *+[o][F]{} \ar@{-}[rr]^{2}&&  *+[o][F]{}\\
*+[o][F]{}\ar@{-}[rr]_{r-1} && *+[o][F]{} \ar@{.}[rr] &&  *+[o][F]{}\ar@{-}[rr]_{2}&&  *+[o][F]{}&&}   
\textrm{ with } r = n/2
 \\

 (II) \xymatrix@-1.7pc{ *+[o][F]{}\ar@{-}[rr]^{r-1} \ar@{-}[d]_{0} && *+[o][F]{} *+[o][F]{}\ar@{-}[d]_{0} \ar@{.}[rr] &&  *+[o][F]{}\ar@{-}[rr]^2\ar@{-}[d]_0 &&  *+[o][F]{}\ar@{-}[rr]^{1} \ar@{-}[d]_{0} && *+[o][F]{} \ar@{=}[rr]^{2}_0&&  *+[o][F]{}\\
*+[o][F]{}\ar@{-}[rr]_{r-1} &&   *+[o][F]{} \ar@{.}[rr] &&  *+[o][F]{}\ar@{-}[rr]_{2}&&  *+[o][F]{}&&}
\textrm{ with } r = n/2
 \\
 (III) \xymatrix@-1.7pc{ *+[o][F]{}\ar@{-}[rr]^{r-1} \ar@{-}[d]_{0} &&  *+[o][F]{}\ar@{-}[d]_{0} \ar@{.}[rr] &&  *+[o][F]{}\ar@{-}[rr]^2\ar@{-}[d]_0 &&  *+[o][F]{}\ar@{-}[rr]^{1} \ar@{-}[d]_{0} && *+[o][F]{} \\
*+[o][F]{}\ar@{-}[rr]_{r-1} && *+[o][F]{} \ar@{.}[rr] &&  *+[o][F]{}\ar@{-}[rr]_{2}&&  *+[o][F]{}} 
\textrm{ with } r = (n+1)/2
\\
 (IV)\; \xymatrix@-1.7pc{*+[o][F]{} \ar@{-}[rr]^{r-3} &&*+[o][F]{} \ar@{.}[rr] &&*+[o][F]{} \ar@{-}[rr]^1 &&*+[o][F]{} \ar@{-}[rr]^0 &&*+[o][F]{}  \ar@{-}[rr]^1&&*+[o][F]{} \ar@{.}[rr]&&  *+[o][F]{} \ar@{-}[rr]^{r-3} &&*+[o][F]{} \ar@{-}[rr]^{r-2}&&*+[o][F]{}\ar@{-}[rr]^{r-1}&&*+[o][F]{}  \ar@{-}[rr]^{r-2}&&*+[o][F]{}\ar@{-}[rr]^{r-1}&&*+[o][F]{} }
\\\textrm{ with } r = n/2
 \\
(V)\; \xymatrix@-1.7pc{*+[o][F]{} \ar@{-}[rr]^{r-3} &&*+[o][F]{} \ar@{.}[rr] &&*+[o][F]{} \ar@{-}[rr]^1 &&*+[o][F]{} \ar@{-}[rr]^0 &&*+[o][F]{}  \ar@{-}[rr]^1&&*+[o][F]{} \ar@{.}[rr]&&  *+[o][F]{} \ar@{-}[rr]^{r-3} &&*+[o][F]{} \ar@{-}[rr]^{r-2}&&*+[o][F]{}\ar@{-}[rr]^{r-1}&&*+[o][F]{}  \ar@{-}[rr]^{r-2}&&*+[o][F]{}\ar@{=}[rr]^{r-1}_{r-3}&&*+[o][F]{} }
\\\textrm{ with } r = n/2
 \\
(VI)\; \xymatrix@-1.7pc{*+[o][F]{} \ar@{-}[rr]^{r-3} &&*+[o][F]{} \ar@{.}[rr] &&*+[o][F]{} \ar@{-}[rr]^1 &&*+[o][F]{} \ar@{-}[rr]^0 &&*+[o][F]{}  \ar@{-}[rr]^1&&*+[o][F]{} \ar@{.}[rr]&&  *+[o][F]{} \ar@{-}[rr]^{r-3} &&*+[o][F]{} \ar@{-}[rr]^{r-2}&&*+[o][F]{}\ar@{-}[rr]^{r-1}\ar@{-}[d]_{r-3}&&*+[o][F]{}\ar@{=}[d]^{r-3}_{r-2}\\
&& && && && && && && && *+[o][F]{} \ar@{-}[rr]_{r-1}&& *+[o][F]{}\\
 }
\\\textrm{ with } r = n/2
 \\
(VII)\; \xymatrix@-1.7pc{*+[o][F]{} \ar@{-}[rr]^{r-1} &&*+[o][F]{} \ar@{-}[rr]^{r-2} &&*+[o][F]{} \ar@{-}[rr]^{r-1} &&*+[o][F]{} \ar@{-}[rr]^{r-2} &&*+[o][F]{} \ar@{-}[rr]^{r-3} &&*+[o][F]{} \ar@{-}[rr]^{r-4} &&*+[o][F]{}\ar@{-}[rr]^{r-5}\ar@{-}[d]_{r-3}&&*+[o][F]{}\ar@{-}[d]_{r-3}*+[o][F]{} \ar@{.}[rr]&&*+[o][F]{}\ar@{-}[rr]^{0}\ar@{-}[d]_{r-3}&&*+[o][F]{}\ar@{-}[d]_{r-3}&&\\
 && && && && && && *+[o][F]{} \ar@{-}[rr]_{r-5}&& *+[o][F]{}*+[o][F]{} \ar@{.}[rr]&&*+[o][F]{} \ar@{-}[rr]_{0}&& *+[o][F]{}*&&\\
 }
\\\textrm{ with } r = n/2
 \\
(VIII)\; \xymatrix@-1.7pc{*+[o][F]{} \ar@{=}[rr]^{r-1}_{r-3} &&*+[o][F]{} \ar@{-}[rr]^{r-2} &&*+[o][F]{} \ar@{-}[rr]^{r-1} &&*+[o][F]{} \ar@{-}[rr]^{r-2} &&*+[o][F]{} \ar@{-}[rr]^{r-3} &&*+[o][F]{} \ar@{-}[rr]^{r-4} &&*+[o][F]{}\ar@{-}[rr]^{r-5}\ar@{-}[d]_{r-3}&&*+[o][F]{}\ar@{-}[d]_{r-3}*+[o][F]{} \ar@{.}[rr]&&*+[o][F]{}\ar@{-}[rr]^{0}\ar@{-}[d]_{r-3}&&*+[o][F]{}\ar@{-}[d]_{r-3}&&\\
 && && && && && && *+[o][F]{} \ar@{-}[rr]_{r-5}&& *+[o][F]{}*+[o][F]{} \ar@{.}[rr]&&*+[o][F]{} \ar@{-}[rr]_{0}&& *+[o][F]{}*&&\\
 }
 \\\textrm{ with } r = n/2
 \\
(IX)\; \xymatrix@-1.7pc{*+[o][F]{} \ar@{-}[rr]^{r-1}\ar@{=}[d]_{r-3}^{r-2} &&*+[o][F]{} \ar@{-}[rr]^{r-2}\ar@{-}[d]^{r-3} &&*+[o][F]{} \ar@{-}[rr]^{r-3} &&*+[o][F]{} \ar@{-}[rr]^{r-4} &&*+[o][F]{}\ar@{-}[rr]^{r-5}\ar@{-}[d]_{r-3}&&*+[o][F]{}\ar@{-}[d]_{r-3}*+[o][F]{} \ar@{.}[rr]&&*+[o][F]{}\ar@{-}[rr]^{0}\ar@{-}[d]_{r-3}&&*+[o][F]{}\ar@{-}[d]_{r-3}&&\\
  *+[o][F]{} \ar@{-}[rr]_{r-1}&& *+[o][F]{} && && && *+[o][F]{} \ar@{-}[rr]_{r-5}&& *+[o][F]{}*+[o][F]{} \ar@{.}[rr]&&*+[o][F]{} \ar@{-}[rr]_{0}&& *+[o][F]{}*&&\\
 }
 \\\textrm{ with } r = n/2
 \\

(X)\; \xymatrix@-1.7pc{*+[o][F]{} \ar@{-}[rr]^{h} &&*+[o][F]{} \ar@{.}[rr]&&*+[o][F]{} \ar@{-}[rr]^{0} &&*+[o][F]{} \ar@{.}[rr] &&*+[o][F]{} \ar@{-}[rr]^{h} &&*+[o][F]{} \ar@{-}[rr]^{h+1} &&*+[o][F]{} \ar@{-}[rr]^{h+2} &&*+[o][F]{}\ar@{-}[rr]^{h+3}\ar@{-}[d]_{h+1}&&*+[o][F]{}\ar@{-}[d]_{h+1}*+[o][F]{} \ar@{.}[rr]&&*+[o][F]{}\ar@{-}[rr]^{r-1}\ar@{-}[d]_{h+1}&&*+[o][F]{}\ar@{-}[d]_{h+1}&&\\
 && && && && && && && *+[o][F]{} \ar@{-}[rr]_{h+3}&& *+[o][F]{}*+[o][F]{} \ar@{.}[rr]&&*+[o][F]{} \ar@{-}[rr]_{r-1}&& *+[o][F]{}*&&\\
 }\\
\textrm{with } r=(n+1)/2 \textrm{ and } 1\leq h \leq r-2
 \\

(XI)\;  \xymatrix@-1.7pc{*+[o][F]{} \ar@{-}[rr]^{h} &&*+[o][F]{} \ar@{.}[rr] &&*+[o][F]{} \ar@{-}[rr]^1 &&*+[o][F]{} \ar@{-}[rr]^0 &&*+[o][F]{}  \ar@{-}[rr]^1&&*+[o][F]{} \ar@{.}[rr]&&  *+[o][F]{} \ar@{-}[rr]^{r-2}&&*+[o][F]{}\ar@{-}[rr]^{r-1}&&*+[o][F]{}  \ar@{-}[rr]^{r-2} &&*+[o][F]{} \ar@{.}[rr]&&  *+[o][F]{} \ar@{-}[rr]^{h}&&*+[o][F]{}} \\
\textrm{with } r=n/2 \textrm{ and }1 \leq h \leq r-2
 \\
(XII)\;  \xymatrix@-1.7pc{*+[o][F]{} \ar@{=}[rr]^{2}_{0} &&*+[o][F]{} \ar@{-}[rr]^1 &&*+[o][F]{} \ar@{-}[rr]^0 &&*+[o][F]{}  \ar@{-}[rr]^1&&*+[o][F]{} \ar@{.}[rr]&&  *+[o][F]{} \ar@{-}[rr]^{r-2}&&*+[o][F]{}\ar@{-}[rr]^{r-1}&&*+[o][F]{}  \ar@{-}[rr]^{r-2} &&*+[o][F]{} \ar@{.}[rr]&&  *+[o][F]{} \ar@{-}[rr]^{2}&&*+[o][F]{} }
 \\\textrm{ with } r = n/2
\\
 (XIII) \xymatrix@-1.7pc{ 
 *+[o][F]{}\ar@{-}[rr]^{4} \ar@{-}[d]_{1} &&  *+[o][F]{}\ar@{-}[d]_{1} \ar@{-}[rr]^3 &&  *+[o][F]{}\ar@{-}[rr]^2\ar@{-}[d]^1 && *+[o][F]{} \\
 *+[o][F]{}\ar@{-}[rr]^{4} \ar@{-}[d]_{0} &&  *+[o][F]{}\ar@{-}[d]_{0} \ar@{-}[rr]^3 &&  *+[o][F]{}\ar@{}[rr]\ar@{-}[d]^0 && \\
 *+[o][F]{}\ar@{-}[rr]_{4} &&  *+[o][F]{} \ar@{-}[rr]_3 &&  *+[o][F]{} &&}
 \textrm{ with } n=10 \textrm{ and } r = 5
\\
(XIV)\xymatrix@-1.7pc{ 
 *+[o][F]{}\ar@{-}[rr]^{0} \ar@{-}[d]_{h} &&  *+[o][F]{}\ar@{-}[d]_{h} \ar@{.}[rr] &&  *+[o][F]{}\ar@{-}[rr]^{h-3}\ar@{-}[d]^{h} &&*+[o][F]{}\ar@{-}[rr]^{h-2}\ar@{-}[d]^{h} && *+[o][F]{} \ar@{-}[d]^{h}\ar@{-}[rr]^{h-1}&&*+[o][F]{} \ar@{-}[rr]^{h}&&*+[o][F]{} \ar@{-}[rr]^{h+1}&&*+[o][F]{} \ar@{-}[rr]^{h+2}&&*+[o][F]{}\ar@{-}[d]^{h+1} \ar@{-}[rr]^{h+3}&&*+[o][F]{}\ar@{-}[d]^{h+1} \ar@{.}[rr]&&*+[o][F]{}\ar@{-}[d]^{h+1} \ar@{-}[rr]^{r-1}&&*+[o][F]{}\ar@{-}[d]^{h+1} \\
 *+[o][F]{}\ar@{-}[rr]_{0}&& *+[o][F]{}\ar@{.}[rr]&&*+[o][F]{}\ar@{-}[rr]_{h-3}&&*+[o][F]{}\ar@{-}[rr]_{h-2}&& *+[o][F]{}&& &&  && && *+[o][F]{}\ar@{-}[rr]_{h+3}&&*+[o][F]{}\ar@{.}[rr]&&  *+[o][F]{} \ar@{-}[rr]_{r-1}&&*+[o][F]{}}
 \\
\textrm{with } r=(n+1)/2 \textrm{ and } 2\leq h \leq r-2
 \\

(XV)\xymatrix@-1.7pc{ 
 *+[o][F]{}\ar@{-}[rr]^{r-1} \ar@{-}[d]_{h} &&  *+[o][F]{}\ar@{-}[d]_{h} \ar@{.}[rr] &&  *+[o][F]{}\ar@{-}[rr]^{h+3}\ar@{-}[d]^{h} &&*+[o][F]{}\ar@{-}[rr]^{h+2}\ar@{-}[d]^{h} && *+[o][F]{} \ar@{-}[d]^{h}\ar@{-}[rr]^{h+1}&&*+[o][F]{} \ar@{-}[rr]^{h}&&*+[o][F]{} \ar@{-}[rr]^{h-1}&&*+[o][F]{} \ar@{-}[rr]^{h-2}&&*+[o][F]{} \ar@{.}[rr]&&*+[o][F]{} \ar@{-}[rr]^1&&*+[o][F]{}\ar@{-}[d]^0 \\
 *+[o][F]{}\ar@{-}[rr]_{r-1}&& *+[o][F]{}\ar@{.}[rr]&&*+[o][F]{}\ar@{-}[rr]_{h+3}&&*+[o][F]{}\ar@{-}[rr]_{h+2}&& *+[o][F]{} &&*+[o][F]{}\ar@{-}[rr]_{h} && *+[o][F]{}\ar@{-}[rr]_{h-1}&&*+[o][F]{}\ar@{-}[rr]_{h-3} &&*+[o][F]{}\ar@{.}[rr]&&*+[o][F]{}\ar@{-}[rr]_1&&  *+[o][F]{}}
 \\
\textrm{with } r=n/2 \textrm{ and } 2\leq h \leq r-2
 \\
(XVI)\xymatrix@-1.7pc{  *+[o][F]{}\ar@{-}[rr]^0\ar@{=}[d]_{h}^k && *+[o][F]{}\ar@{.}[rr]\ar@{=}[d]_{h}^k && *+[o][F]{}\ar@{-}[rr]^{h-2}\ar@{-}[d]_{h} && *+[o][F]{}\ar@{-}[rr]^{h-1}\ar@{-}[d]_{h} &&  *+[o][F]{}\ar@{-}[rr]^h  && *+[o][F]{} \ar@{.}[rr]&&*+[o][F]{} \ar@{-}[rr]^{r-1}&&  *+[o][F]{} \ar@{.}[rr]&&  *+[o][F]{} \ar@{-}[rr]^h&&*+[o][F]{} \\
 *+[o][F]{}\ar@{-}[rr]_0&& *+[o][F]{}\ar@{.}[rr]&&  *+[o][F]{}\ar@{-}[rr]_{h-2}&&   *+[o][F]{} &&}
  \\
\textrm{with } r=n/2,\, 1\leq k\leq h-3 \textrm{ and } 2\leq h \leq r-2
 \\
(XVII)\xymatrix@-1.6pc{  *+[o][F]{}\ar@{-}[rr]^0\ar@{=}[d]_{h}^k &&*+[o][F]{}\ar@{-}[rr]^1\ar@{=}[d]_{h}^k && *+[o][F]{}\ar@{.}[rr]\ar@{=}[d]_{h}^k && *+[o][F]{}\ar@{-}[rr]^{h-2}\ar@{-}[d]_{h} &&  *+[o][F]{}\ar@{-}[rr]^{h-1}  \ar@{-}[d]_{h} && *+[o][F]{} \ar@{-}[rr]^{h}&&*+[o][F]{} \ar@{-}[rr]^{h+1}&&*+[o][F]{} \ar@{-}[rr]^{h+2} \ar@{-}[d]^{h}&&*+[o][F]{} \ar@{-}[d]^{h} \ar@{.}[rr]&&  *+[o][F]{} \ar@{-}[rr]^{r-1} \ar@{-}[d]^{h}&&  *+[o][F]{} \ar@{-}[d]^{h} \\
 *+[o][F]{}\ar@{-}[rr]_0&&*+[o][F]{}\ar@{-}[rr]_1&& *+[o][F]{}\ar@{.}[rr]&&   *+[o][F]{}\ar@{-}[rr]_{h-2}&&*+[o][F]{} &&&&&& *+[o][F]{} \ar@{-}[rr]_{h+2}&& *+[o][F]{} \ar@{.}[rr]&& *+[o][F]{} \ar@{-}[rr]_{r-1}&& *+[o][F]{} }
   \\
\textrm{with } r=n/2,\, 1\leq k\leq h-3 \textrm{ and } 2\leq h \leq r-2
 \\
 \\
 \hline
 \end{array}\]
 \end{tiny}
\caption{Permutation representation graphs for $S_n$ with rank $r\geq \frac{n}{2}$ having a fracture graph and a split that is not perfect.}\label{notperfect}
\end{table}


Before we analyse this case in general, we give in Table~\ref{notperfect} permutation representation graphs that will arise in this case.
We conjecture that these graphs define string C-groups for the symmetric group. Anyway, these graphs will all turn out to be excluded from our classification.
Hence proving that they give indeed string C-groups would not add much to our result.
But it is an interesting open problem.

\begin{prop}\cite[Proposition 5.1]{2017CFLM}\label{CCD}
If $B$ is primitive, then the set $J_B$  is an interval. The same result holds for $A$.
\end{prop}

In this case $J_B$ is either empty or an interval. As before let $\mathcal{G}$ be the permutation representation graph of $\Gamma$.

\begin{prop}\label{path}\cite[Proposition 5.18]{2017CFLM}
If $e$ is an $f$-edge of $\mathcal{G}$ not in an alternating square, then any path (not containing another $f$-edge) from $e$ to an edge with label $l$, with $l<f$ (resp. $l>f$), contains all labels between $l$ and $f$. 
Moreover, there exists a path from $e$ to an $l$-edge, that is fixed by $\Gamma_{>l}$  (resp. $\Gamma_{<l}$).
\end{prop}


Recall that we deal with the case where $i$ is the label of a split that is not perfect. In what follows the first $G_i$-orbit is the orbit corresponding to the group $A$, while the second $G_i$-orbit is the orbit corresponding to the group $B$.

In the following two propositions we consider the case $i=0$.

\begin{prop}\label{i=0andh=r-1} 
Let $i=0$. If $\rho_{r-1}$ acts non-trivially on both $G_i$-orbits then $r\leq n/2$. Moreover if $r=n/2$ then $\Gamma$ has the following permutation representation graph and $G$ is isomorphic to $C_2\wr S_{n/2}$.

\begin{tabular}{c}
$\xymatrix@-1.3pc{*+[o][F]{}\ar@{-}[rr]^{r-1} && *+[o][F]{} \ar@{-}[rr]^{r-2}&&*+[o][F]{} \ar@{.}[rr] && *+[o][F]{}\ar@{-}[rr]^1 && *+[o][F]{}\ar@{-}[rr]^0 &&*+[o][F]{} \ar@{-}[rr]^1&& *+[o][F]{}\ar@{.}[rr]&& *+[o][F]{} \ar@{-}[rr]^{r-2}&& *+[o][F]{} \ar@{-}[rr]^{r-1}&&*+[o][F]{}  }$\\
\end{tabular}

\end{prop}
\begin{proof}
To prove that  $r\leq n/2$ we use the argument of the proof of Proposition 5.19 of \cite{2017CFLM}:
by Proposition~\ref{path} there are two paths, one in the first and the other in the second 
$G_i$-orbit, each containing all labels from $1$ to $r-1$, thus $2(r-1)+1\leq n-1\Leftrightarrow r\leq n/2$.

Now suppose that $r=n/2$. In that case the paths have consecutive labels, and we get the following possibilities for the permutation representation graphs.
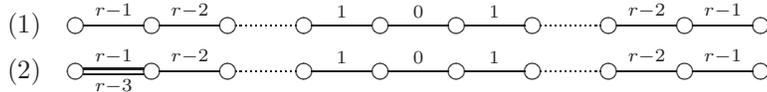
\begin{tabular}{cc}
(1)&$\xymatrix@-1.3pc{*+[o][F]{}\ar@{-}[rr]^{r-1} && *+[o][F]{} \ar@{-}[rr]^{r-2}&&*+[o][F]{} \ar@{.}[rr] && *+[o][F]{}\ar@{-}[rr]^1 && *+[o][F]{}\ar@{-}[rr]^0 &&*+[o][F]{} \ar@{-}[rr]^1&& *+[o][F]{}\ar@{.}[rr]&& *+[o][F]{} \ar@{-}[rr]^{r-2}&& *+[o][F]{} \ar@{-}[rr]^{r-1}&&*+[o][F]{}  }$\\
(2) &$ \xymatrix@-1.3pc{*+[o][F]{}\ar@{=}[rr]^{r-1}_{r-3} && *+[o][F]{} \ar@{-}[rr]^{r-2}&&*+[o][F]{} \ar@{.}[rr] && *+[o][F]{}\ar@{-}[rr]^1 && *+[o][F]{}\ar@{-}[rr]^0 &&*+[o][F]{} \ar@{-}[rr]^1&& *+[o][F]{}\ar@{.}[rr]&& *+[o][F]{} \ar@{-}[rr]^{r-2}&& *+[o][F]{} \ar@{-}[rr]^{r-1}&&*+[o][F]{}  }$
\end{tabular}

Graph (2) is not a string C-group as $G_{r-1} \cong S_{n-1}$, $G_{>r-4}\cong S_{4}\times S_{4}$ and $\langle \rho_{r-2},\rho_{r-3}\rangle$ is a dihedral group while $G_{r-1}\cap G_{>r-4} \cong S_4\times S_3$.

Graph (1) clearly generates an imprimitive group with blocks of size 2, isomorphic to $C_2\wr S_{n/2}$.
\end{proof}
The proof of the following result of \cite{2017CFLM} does not require that $G$ is even.

\begin{prop}\cite[Proposition 5.20]{2017CFLM}\label{i=0andh<>r-1}
Let $i=0$. Let $h\neq r-1$ be the maximal label such that $\rho_h$ acts non-trivially on both $G_i$-orbits.
There exists a set of vertices $X$, contained in the $G_i$-orbit fixed by $\rho_{r-1}$, such that $h\leq \frac{n-|X|-1}{2}$ and $G_{>h}$ fixes $\{1,\ldots, n\}\setminus X$. Moreover if $h= \frac{n-|X|-1}{2}$ then $\Gamma_{<h}$ has the following permutation representation graph, where the black dots represent the vertices of $\{1,\ldots, n\}\setminus X$. 
$$\xymatrix@-1.3pc{*{\bullet}\ar@{-}[rr]^h &&*{\bullet} \ar@{-}[rr]^{h-1}&&*{\bullet} \ar@{.}[rr] &&*{\bullet}\ar@{-}[rr]^1 &&*{\bullet}\ar@{-}[rr]^0 &&*{\bullet} \ar@{-}[rr]^1&& *{\bullet} \ar@{.}[rr]&&*{\bullet} \ar@{-}[rr]^{h-1}&& *{\bullet} \ar@{-}[rr]^{h}&&*+[o][F]{}  }
$$
\end{prop}

Propositions~\ref{i=0andh=r-1} and \ref{i=0andh<>r-1} can be rewritten for $i=r-1$ using duality.
In what follows we suppose that $i\neq 0,\, r-1$. By Proposition~\ref{CCD}, the set of indices $J_B$ is an interval, either $J_B=\{i+1,\ldots, m\}$  for some $m\leq r-1$ or, $J_B=\{m,\ldots, i-1\}$ for some $m\geq 0$.
As the $i$-split is not perfect, $J_A$ is not an interval. Suppose that $J_B=\{i+1,\ldots, m\}$  for some $m\leq r-1$ (the other case can be obtained by duality).
There are two possibilities, either $G_{<i}$ is transitive or intransitive on the first $G_i$-orbit. 
We consider these cases separately. First, let us restate Proposition 5.22 of~\cite{2017CFLM} so that it includes the case where $G$ is not an even group.

\begin{prop}\label{<itran}
Let $r\geq n/2$. Let $h>i$ be the maximal label of  a permutation acting non-trivially on both $G_i$-orbits.
If $G_{<i}$ is transitive on the first $G_i$-orbit then $h=i+1$ and either  $h=r-1$ is the label of a perfect  split or
 $h<r-1$ and there exists a set of vertices $X$, contained in the second $G_i$-orbit, such that $h\leq \frac{n-|X|-1}{2}$ and $G_{>h}$ fixes $\{1,\ldots, n\}\setminus X$. 
Moreover if $h= \frac{n-|X|-1}{2}$ then $\Gamma_{<h}$ (with $h=i+1$)  has the following permutation representation graph for some $k\in\{2,\ldots,i-1\}$ where the black dots represent the vertices of $\{1,\ldots, n\}\setminus X$. 
 $$
\xymatrix@-1.3pc{*{\bullet}\ar@{-}[rr]^0 \ar@{=}[d]_{i+1}^k &&*{\bullet} \ar@{-}[rr]^1\ar@{=}[d]_{i+1}^k&&*{\bullet}\ar@{=}[d]_{i+1}^k\ar@{.}[rr] && *{\bullet}\ar@{-}[rr]^{i-1}\ar@{-}[d]_{i+1} && *{\bullet}\ar@{-}[rr]^i \ar@{-}[d]_{i+1} &&*{\bullet} \ar@{-}[rr]^{i+1}&&*+[o][F]{} \\
 *{\bullet}\ar@{-}[rr]_0 && *{\bullet}\ar@{-}[rr]_1&&*{\bullet}\ar@{.}[rr] && *{\bullet}\ar@{-}[rr]_{i-1}&&*{\bullet}&&&&
 }
  $$
\end{prop}
\begin{proof}
The case where $h\neq r-1$ is~\cite[Proposition 5.22]{2017CFLM}.

Let $h=r-1$. In this case, as shown in the proof of ~\cite[Proposition 5.22]{2017CFLM}, $r=n/2$ and $\mathcal G$ is as follows.
 \[\begin{array}{c}
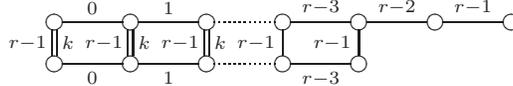
\xymatrix@-1.3pc{*+[o][F]{}\ar@{-}[rr]^0 \ar@{=}[d]_{r-1} ^k && *+[o][F]{} \ar@{-}[rr]^1\ar@{=}[d]_{r-1} ^k&&*+[o][F]{}\ar@{=}[d]_{r-1} ^k\ar@{.}[rr] && *+[o][F]{}\ar@{-}[rr]^{r-3}\ar@{-}[d]_{r-1} && *+[o][F]{}\ar@{-}[rr]^{r-2} \ar@{-}[d]_{r-1} &&*+[o][F]{} \ar@{-}[rr]^{r-1}&&*+[o][F]{} \\
 *+[o][F]{}\ar@{-}[rr]_0 && *+[o][F]{} \ar@{-}[rr]_1&&*+[o][F]{} \ar@{.}[rr] && *+[o][F]{}\ar@{-}[rr]_{r-3}&& *+[o][F]{}&&&&}
 \end{array}\]
where $k< r-2$ or, the graph above with some extra $(r-2)$-edges. In any case $(r-1)$ is the label of a perfect split.
\end{proof}

The following proposition and its proof are similar to~\cite[Proposition 5.23]{2017CFLM}. As we do not assume that $G$ is an even group, we redo what was done in  \cite{2017CFLM} with small adaptations to our setting.

\begin{prop}\label{<iint}
Let $r\geq n/2$ and let $h>i$ be the maximal label of  a permutation acting non-trivially on both $G_i$-orbits.  
If $G_{<i}$ is intransitive in the first $G_i$-orbit, then  one of the following situations occurs.
\begin{enumerate}
\item $h=r-1$, $r= n/2$ and $\mathcal{G}$ is either graph (I) or graph (II) of Table~\ref{notperfect}.
\item $h < r-1$ and there exists a set of vertices $X$, contained in the second $G_i$-orbit,  such that $h\leq \frac{n-|X|-1}{2}$ and $G_{>h}$ fixes $\{1,\ldots, n\}\setminus X$. 
Moreover if $h= \frac{n-|X|-1}{2}$ then $h=i+1$ and $\Gamma_{<h}$ has the following permutation representation graph, where the black dots represent the vertices of $\{1,\ldots, n\}\setminus X$. 
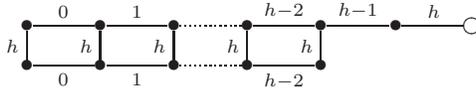
$$\xymatrix@-1.3pc{*{\bullet}\ar@{-}[rr]^0 \ar@{-}[d]_{h} &&*{\bullet} \ar@{-}[rr]^1\ar@{-}[d]_{h}&&*{\bullet}\ar@{-}[d]_{h} \ar@{.}[rr] && *{\bullet}\ar@{-}[rr]^{h-2}\ar@{-}[d]_{h} && *{\bullet}\ar@{-}[rr]^{h-1} \ar@{-}[d]_{h} &&*{\bullet} \ar@{-}[rr]^{h}&&  *+[o][F]{}\\
 *{\bullet}\ar@{-}[rr]_0 && *{\bullet} \ar@{-}[rr]_1&&*{\bullet} \ar@{.}[rr] && *{\bullet}\ar@{-}[rr]_{h-2}&& *{\bullet}&&}
$$
\end{enumerate}
\end{prop}
\begin{proof}
In this case $J_A$ is not an interval thus, by Proposition~\ref{CCD},  $A$ is  imprimitive embedded into  $S_k\wr S_m$ and $G_{<i}$ is fixing all the blocks (with $k,m\geq 2$).
By Proposition~\ref{path} there exists a path $\mathcal{P}_1$ from the $h$-edge in the first $G_i$-orbit to the vertex $a$, and a path $\mathcal{P}_2$ from the $h$-edge in the second $G_i$-orbit to the vertex $b$, each of them containing all labels from $i+1$ to $h-1$, and fixed by $G_{>h}$. 
In addition there is a path $\mathcal{P}_3$ in the block $\beta$ containing the vertex $a$, and edges with all labels from $0$ to $i-1$.
Moreover, there is also a path $\mathcal{P}_4$ in the block $\beta\rho_{i+1}$ also containing edges with all labels from $0$ to $i-1$.

In Proposition 5.23 of \cite{2017CFLM} the authors proved that  $\mathcal{P}_1$, $\mathcal{P}_2$, $\mathcal{P}_3$ and $\mathcal{P}_4$ have no edge in common, as shown in the following figure. 

 $$
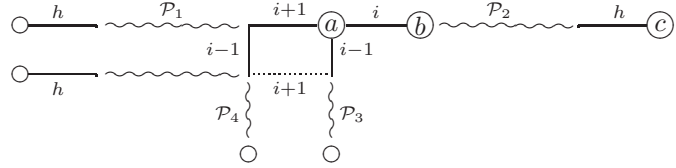
\xymatrix@-1.3pc{  *+[o][F]{} \ar@{-}[rr]^{h}&&*+[.][F]{} \ar@{~}[rrrr]^{\mathcal{P}_1}   &&&&*+[.][F]{}\ar@{-}[d]_{i-1}\ar@{-}[rr]^{i+1} && *+[o][F]{a}\ar@{-}[d]^{i-1} \ar@{-}[rr]^i&& *+[o][F]{b}\ar@{~}[rrrr]^{\mathcal{P}_2}&&&& *+[.][F]{}\ar@{-}[rr]^h&& *+[o][F]{c}\\
   *+[o][F]{} \ar@{-}[rr]_{h}&&*+[.][F]{} \ar@{~}[rrrr]  &&&&*+[.][F]{}\ar@{~}[dd]_{\mathcal{P}_4}\ar@{.}[rr]_{i+1}&&*+[.][F]{}\ar@{~}[dd]^{\mathcal{P}_3}&&&&&&\\
  &&&&&&&&&&&&&&\\
  &&&&&&*+[o][F]{}&& *+[o][F]{} &&&&&&}
 $$
 
Suppose that $h=r-1$.  
If $k=2$ then $n\geq 2 |\mathcal{P}_1|+ |\mathcal{P}_2|\geq 2(r-1)+|\mathcal{P}_2| $, hence $r\leq \frac{n-(|\mathcal{P}_2|-2)}{2}$ which gives a contradiction when $|\mathcal{P}_2|>2$.  
If $|\mathcal{P}_2|=2$, then $r=n/2$ and $\mathcal{G}$ is one of the two graphs (I) or (II) of Table~\ref{notperfect} with $r=3$ and $n=6$.  Now consider that  $k>2$, then $n\geq 3(r-1)$ and since $r\geq n/2$, we get a contradiction.

From now on assume that $h\neq r-1$.
Suppose that  $\rho_{h+1}$ acts trivially on the first $G_i$-orbit.
 Let $P$ be the set of vertices of $\mathcal{P}_1\cup\mathcal{P}_2\cup\mathcal{P}_3\cup \mathcal{P}_4$  excluding the vertex $c$ on the right hand side of the diagram above, and $X=\{1,\ldots,n\}\setminus P$. 
The set $P$ is fixed by $G_{>h}$ and $2h+1\leq |P|=n-|X|$. 
Moreover when $2h+1= |P|$ the  permutation representation graph of $\Gamma_{<h}$ is the one given in this proposition with $h=i+1$. 

Finally suppose that $\rho_{h+1}$ does not act trivially on the first $G_i$-orbit. Then $G_{>h}$  does not fix the first $G_i$-orbit.
 Moreover $\rho_{r-1}$ acts nontrivially on the first $G_i$-orbit.
 In this case consider the path $\mathcal{P}_1$ as before, with the label of its first edge being $r-1$ instead of $h$.
 In this case $n_1\geq 2(r-1)$ and $n_2\geq 2$, hence $r\leq n/2$.
If $n_1\neq  2(r-1)$ then $n_1\geq 2(r-1)+2$ and therefore $r\leq \frac{n-2}{2}$, a contradiction.
Consequently, $n_1= 2(r-1)$ and $\mathcal{G}$ is as follows.
$$
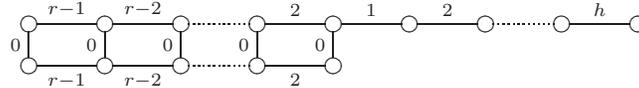
\xymatrix@-1.3pc{ *+[o][F]{}\ar@{-}[rr]^{r-1} \ar@{-}[d]_{0} && *+[o][F]{} \ar@{-}[rr]^{r-2}\ar@{-}[d]_{0}&& *+[o][F]{}\ar@{-}[d]_{0} \ar@{.}[rr] &&  *+[o][F]{}\ar@{-}[rr]^2\ar@{-}[d]_0 &&  *+[o][F]{}\ar@{-}[rr]^{1} \ar@{-}[d]_{0} && *+[o][F]{} \ar@{-}[rr]^{2}&&  *+[o][F]{} \ar@{.}[rr]&&  *+[o][F]{} \ar@{-}[rr]^{h}&&  *+[o][F]{}\\
  *+[o][F]{}\ar@{-}[rr]_{r-1} &&  *+[o][F]{} \ar@{-}[rr]_{r-2}&& *+[o][F]{} \ar@{.}[rr] &&  *+[o][F]{}\ar@{-}[rr]_{2}&&  *+[o][F]{}&&}$$
Hence $n\geq 2(r-1)+h$, which gives a contradiction with $r\geq n/2$ when $h>2$. For $h=2$ we get the possibilities given by graphs  (I) and (II) of Table~\ref{notperfect}. 
\end{proof}


Let us now summarize Propositions~\ref{<itran} and \ref{<iint}.

\begin{prop}\label{h}
Let $r\geq n/2$. Let $i\neq 0$ and $J_B$ be an interval with labels $>i$.
If  $h>i$ is the maximal label of  a permutation acting non-trivially on both $G_i$-orbits then one of the following situations occurs.
\begin{enumerate}
\item  $h\neq r-1$ and there exists a set of vertices $X$, contained in the second $G_i$-orbit, such that $h\leq \frac{n-|X|-1}{2}$ and 
$G_{>h}$ acts trivially on $\{1,\ldots, n\}\setminus X$;
\item $\Gamma$ has a perfect split; 
\item $h=r-1$, $r=n/2$ and $\mathcal{G}$ is one of the graphs (I) or (II) of Table~\ref{notperfect}.
\end{enumerate}
\end{prop}

 
 The result of Proposition 5.27 of \cite{2017CFLM} can be extended to a group $\Gamma$ that is not even.
 \begin{prop}\label{t}
Let $t\in\{0,\ldots,r-2\}$ and $U:=\{1,\ldots, n\}\setminus {\rm Fix}(G_{>t})$. If $t$ is such that
\begin{itemize}
\item  $t\leq \frac{n-|U|-1}{2}$,
\item $\Gamma_{>t}$ has a $2$-fracture graph,
\item $G_{>t}$  acts intransitively on $U$,
\end{itemize}
then $r\leq \frac{n-1}{2}$.
\end{prop}

\begin{proof}
First, $|U|$ must be at least $4$, otherwise $\Gamma_{>t}$ would not admit a 2-fracture graph. Hence the proposition holds for $t=r-2$. 
For $t<r-2$, the construction used in Proposition 5.27 of \cite{2017CFLM} does not use the assumption that $G$ is even. Thus Proposition 5.27 holds for any string C-group $\Gamma$.
\end{proof}

\begin{prop}\label{ineq0Btrivial}
If  $i\neq 0$ and $J_B=\emptyset$ then $r\leq n/2$. Moreover if $r=n/2$ then $\Gamma$ has, up to duality, the permutation representation graph (III) or (XIII) of Table~\ref{notperfect}.
\end{prop}
\begin{proof}
The proof uses the same arguments as those of Proposition~5.25 of \cite{2017CFLM}.

Since $A$ is imprimitive, $A$ is embedded into a wreath product $S_k\wr S_m$ with $n=km+1$ and $G_{<i}$ fixing the blocks.
As $G_{<i}$ commutes with $G_{>i}$ which acts transitively on the blocks, $G_{<i}$ acts in the same way on each block.

If there is a label $l>i$ such that $\rho_l$ acts nontrivially inside a block then the orbits of $\rho_l$ form sub-blocks of 2 points inside each block of size $k$. So $k$ must be even and
$G_{<i}$ is embedded into $S_2\wr S_{k/2}$.
Then by Theorem~\ref{bigimp}, $i\leq (k+2)/2$.  But if $i= (k+2)/2$ then $\Gamma$ does not have a fracture graph, a contradiction. Hence, using the fact that $k$ is even, we get
$i\leq k/2$. Moreover, $r-i-1 \leq m-1$ as the generators of $G_{>i}$ act independently on the $m$ orbits of $G_{<i}$.
So $$r \leq k/2+m \leq (n-1)/4 + (n-1)/k.$$
Now, if $k = 2$, we have $G_{i-1}$ transitive, contradicting the fact that we have a fracture graph.
Hence $k\geq 4$. Therefore,  we get
$r\leq (n-1)/2$.

Suppose then that
every generator with label $>i$ acts trivially on the blocks it fixes.
Thus there are four paths  $\mathcal{P}_1$, $\mathcal{P}_2$, $\mathcal{P}_3$ and $\mathcal{P}_4$, as in the following graph, containing all but one label twice and one cycle that is an alternating square. 
$$
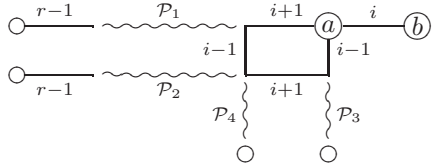
\xymatrix@-1.3pc{  *+[o][F]{} \ar@{-}[rr]^{r-1}&&*+[.][F]{} \ar@{~}[rrrr]^{\mathcal{P}_1}   &&&&*+[.][F]{}\ar@{-}[d]_{i-1}\ar@{-}[rr]^{i+1} && *+[o][F]{a}\ar@{-}[d]^{i-1} \ar@{-}[rr]^i&& *+[o][F]{b}\\
   *+[o][F]{} \ar@{-}[rr]_{r-1}&&*+[.][F]{} \ar@{~}[rrrr]_{\mathcal{P}_2}   &&&&*+[.][F]{}\ar@{~}[dd]_{\mathcal{P}_4}\ar@{-}[rr]_{i+1}&&*+[.][F]{}\ar@{~}[dd]^{\mathcal{P}_3}&&\\
  &&&&&&&&&&\\
  &&&&&&*+[o][F]{}&& *+[o][F]{} &&}
 $$
 Hence $2(r-1)+1\leq n$.
If equality holds, $\Gamma$ has the graph (III) of Table~\ref{notperfect}.
 Otherwise there are $(k-2)(m-2)$ points not covered by these paths and hence $n-(k-2)(m-2)\geq 2(r-1)+1$.
 Then $n/2 \leq r \leq (n-(k-2)(m-2)+1)/2$. 
 This forces $k = m=3$, $n=10$ and $r=5$, corresponding to graph (XIII) of Table~\ref{notperfect}.
\end{proof}

\begin{prop}\label{IPfails2}
If $\Gamma$ has one of the following permutation representation graphs (where $h\in\{1,\ldots, r-1\}$), then $\Gamma$ is not a string C-group.
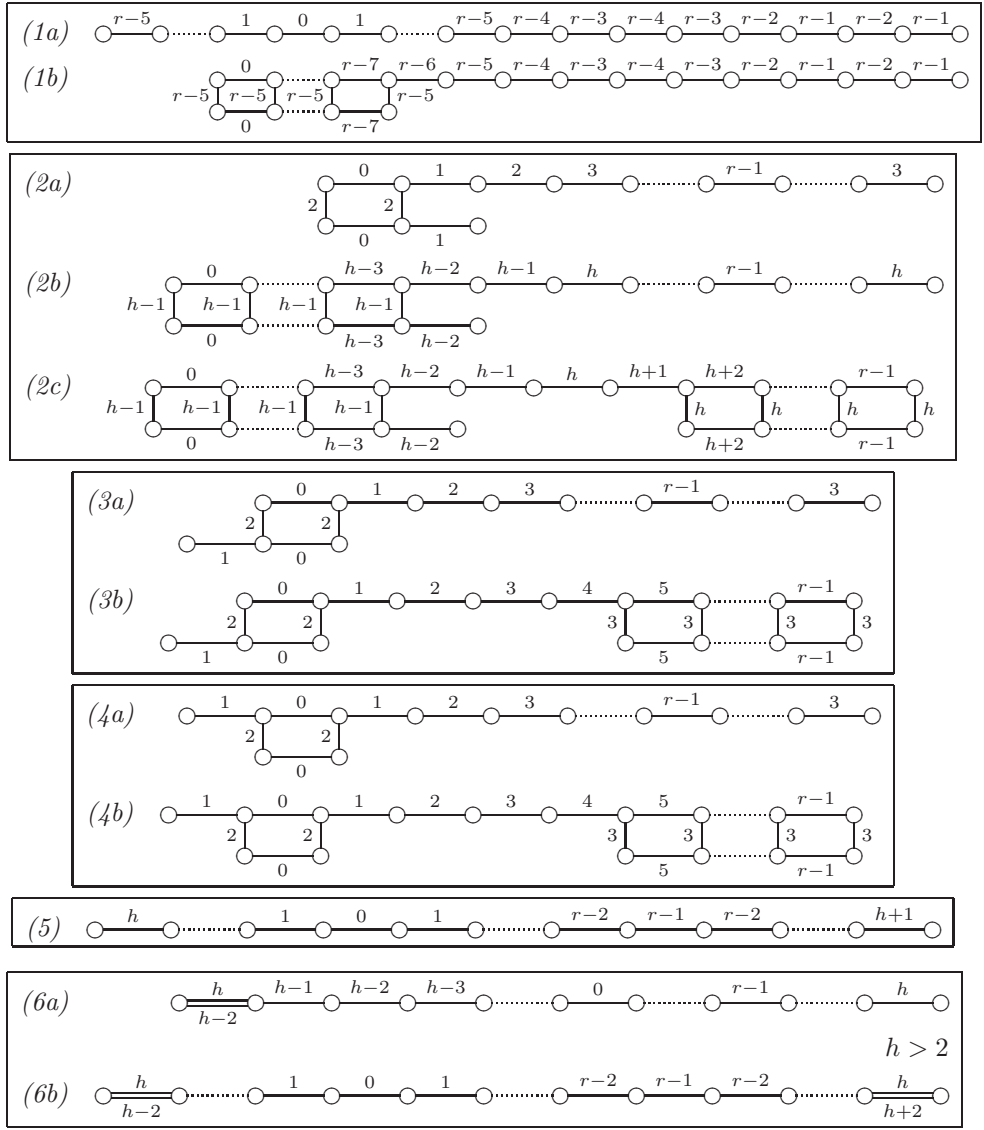
\begin{center}
\begin{tabular}{|cr|}
\hline
(1a)& $\xymatrix@-1.6pc{*+[o][F]{}\ar@{-}[rr]^{r-5} &&*+[o][F]{} \ar@{.}[rr] && *+[o][F]{}\ar@{-}[rr]^1 && *+[o][F]{}\ar@{-}[rr]^0 &&*+[o][F]{} \ar@{-}[rr]^1&& *+[o][F]{}\ar@{.}[rr]&&  *+[o][F]{} \ar@{-}[rr]^{r-5}&& *+[o][F]{}\ar@{-}[rr]^{r-4} && *+[o][F]{}\ar@{-}[rr]^{r-3} && *+[o][F]{}\ar@{-}[rr]^{r-4} &&*+[o][F]{} \ar@{-}[rr]^{r-3}&& *+[o][F]{}\ar@{-}[rr]^{r-2}&& *+[o][F]{} \ar@{-}[rr]^{r-1}&& *+[o][F]{} \ar@{-}[rr]^{r-2}&&*+[o][F]{}  \ar@{-}[rr]^{r-1}&&*+[o][F]{}  }$\\

(1b)&$\xymatrix@-1.6pc{  *+[o][F]{}\ar@{-}[rr]^0\ar@{-}[d]_{r-5} &&*+[o][F]{}\ar@{.}[rr]\ar@{-}[d]_{r-5} && *+[o][F]{} \ar@{-}[d]_{r-5} \ar@{-}[rr]^{r-7} && *+[o][F]{}\ar@{-}[d]^{r-5}  \ar@{-}[rr]^{r-6}&&  *+[o][F]{} \ar@{-}[rr]^{r-5}&& *+[o][F]{}\ar@{-}[rr]^{r-4} && *+[o][F]{}\ar@{-}[rr]^{r-3} && *+[o][F]{}\ar@{-}[rr]^{r-4} &&*+[o][F]{} \ar@{-}[rr]^{r-3}&& *+[o][F]{}\ar@{-}[rr]^{r-2}&& *+[o][F]{} \ar@{-}[rr]^{r-1}&& *+[o][F]{} \ar@{-}[rr]^{r-2}&&*+[o][F]{}  \ar@{-}[rr]^{r-1}&&*+[o][F]{} \\
 *+[o][F]{}\ar@{-}[rr]_0&& *+[o][F]{}\ar@{.}[rr] && *+[o][F]{} \ar@{-}[rr]_{r-7}&&   *+[o][F]{}&&}$\\
 \hline
 \end{tabular}
 \end{center}
 \begin{center}
\begin{tabular}{|cr|}
\hline
(2a) & $\xymatrix@-1.3pc{  *+[o][F]{}\ar@{-}[rr]^0\ar@{-}[d]_2 &&  *+[o][F]{}\ar@{-}[rr]^{1} \ar@{-}[d]_{2} && *+[o][F]{} \ar@{-}[rr]^{2}&&  *+[o][F]{} \ar@{-}[rr]^3&&*+[o][F]{} \ar@{.}[rr]&&*+[o][F]{} \ar@{-}[rr]^{r-1}&&  *+[o][F]{} \ar@{.}[rr]&&  *+[o][F]{} \ar@{-}[rr]^3&&*+[o][F]{} \\
  *+[o][F]{}\ar@{-}[rr]_{0}&& *+[o][F]{}\ar@{-}[rr]_{1}&&  *+[o][F]{}&&}$\\

(2b)&  $\xymatrix@-1.3pc{  *+[o][F]{}\ar@{-}[rr]^0\ar@{-}[d]_{h-1} &&*+[o][F]{}\ar@{.}[rr]\ar@{-}[d]_{h-1} && *+[o][F]{}\ar@{-}[rr]^{h-3}\ar@{-}[d]_{h-1} &&  *+[o][F]{}\ar@{-}[rr]^{h-2} \ar@{-}[d]_{h-1} && *+[o][F]{} \ar@{-}[rr]^{h-1}&&  *+[o][F]{} \ar@{-}[rr]^h&&*+[o][F]{} \ar@{.}[rr]&&*+[o][F]{} \ar@{-}[rr]^{r-1}&&  *+[o][F]{} \ar@{.}[rr]&&  *+[o][F]{} \ar@{-}[rr]^h&&*+[o][F]{} \\
 *+[o][F]{}\ar@{-}[rr]_0&& *+[o][F]{}\ar@{.}[rr]&&   *+[o][F]{}\ar@{-}[rr]_{h-3}&& *+[o][F]{} \ar@{-}[rr]_{h-2}&& *+[o][F]{} &&}$
 \\
 (2c)&  $\xymatrix@-1.3pc{  *+[o][F]{}\ar@{-}[rr]^0\ar@{-}[d]_{h-1} &&*+[o][F]{}\ar@{.}[rr]\ar@{-}[d]_{h-1} && *+[o][F]{}\ar@{-}[rr]^{h-3}\ar@{-}[d]_{h-1} &&  *+[o][F]{}\ar@{-}[rr]^{h-2} \ar@{-}[d]_{h-1} && *+[o][F]{} \ar@{-}[rr]^{h-1}&&  *+[o][F]{} \ar@{-}[rr]^h
 &&*+[o][F]{} \ar@{-}[rr]^{h+1} &&*+[o][F]{}\ar@{-}[d]^{h}   \ar@{-}[rr]^{h+2}&&*+[o][F]{} \ar@{.}[rr]\ar@{-}[d]^{h} &&  *+[o][F]{} \ar@{-}[rr]^{r-1}\ar@{-}[d]^{h}&&*+[o][F]{} \ar@{-}[d]^{h}\\
 *+[o][F]{}\ar@{-}[rr]_0&& *+[o][F]{}\ar@{.}[rr]&&   *+[o][F]{}\ar@{-}[rr]_{h-3}&& *+[o][F]{} \ar@{-}[rr]_{h-2}&& *+[o][F]{} && &&&&*+[o][F]{} \ar@{-}[rr]_{h+2}&&*+[o][F]{} \ar@{.}[rr]&&*+[o][F]{} \ar@{-}[rr]_{r-1}&&*+[o][F]{}}$
\\
  \hline
  \end{tabular}
 \end{center}
 \begin{center}
\begin{tabular}{|cr|}
\hline
 (3a)& $\xymatrix@-1.3pc{&&  *+[o][F]{}\ar@{-}[rr]^0\ar@{-}[d]_2 &&  *+[o][F]{}\ar@{-}[rr]^{1} \ar@{-}[d]_{2} && *+[o][F]{} \ar@{-}[rr]^{2}&&  *+[o][F]{} \ar@{-}[rr]^3&&*+[o][F]{} \ar@{.}[rr]&&*+[o][F]{} \ar@{-}[rr]^{r-1}&&  *+[o][F]{} \ar@{.}[rr]&&  *+[o][F]{} \ar@{-}[rr]^3&&*+[o][F]{} \\
    *+[o][F]{}\ar@{-}[rr]_{1}&&*+[o][F]{}\ar@{-}[rr]_{0}&& *+[o][F]{}&& }$
  \\
  (3b)& $\xymatrix@-1.3pc{&&  *+[o][F]{}\ar@{-}[rr]^0\ar@{-}[d]_2 &&  *+[o][F]{}\ar@{-}[rr]^{1} \ar@{-}[d]_{2} && *+[o][F]{} \ar@{-}[rr]^{2}&&  *+[o][F]{} \ar@{-}[rr]^3&&  *+[o][F]{} \ar@{-}[rr]^4 &&  *+[o][F]{} \ar@{-}[rr]^5\ar@{-}[d]_{3} &&  *+[o][F]{} \ar@{.}[rr] \ar@{-}[d]_{3}&&*+[o][F]{}  \ar@{-}[d]^{3}\ar@{-}[rr]^{r-1}&&  *+[o][F]{} \ar@{-}[d]^{3} \\
    *+[o][F]{}\ar@{-}[rr]_{1}&&*+[o][F]{}\ar@{-}[rr]_{0}&& *+[o][F]{}&&&&&& &&  *+[o][F]{} \ar@{-}[rr]_5&&  *+[o][F]{} \ar@{.}[rr]&&*+[o][F]{} \ar@{-}[rr]_{r-1}&&  *+[o][F]{} }$
  \\
\hline
  \end{tabular}
 \end{center}
 \begin{center}
\begin{tabular}{|cr|}
\hline
(4a)& $\xymatrix@-1.3pc{ *+[o][F]{}\ar@{-}[rr]^1&& *+[o][F]{}\ar@{-}[rr]^0\ar@{-}[d]_2 &&  *+[o][F]{}\ar@{-}[rr]^{1} \ar@{-}[d]_{2} && *+[o][F]{} \ar@{-}[rr]^{2}&&  *+[o][F]{} \ar@{-}[rr]^3&& *+[o][F]{} \ar@{.}[rr]&&*+[o][F]{} \ar@{-}[rr]^{r-1}&&  *+[o][F]{} \ar@{.}[rr]&&  *+[o][F]{} \ar@{-}[rr]^3&&*+[o][F]{} \\
   &&*+[o][F]{}\ar@{-}[rr]_{0}&&*+[o][F]{} &&}$
   \\
  (4b)& $\xymatrix@-1.3pc{*+[o][F]{}\ar@{-}[rr]^1&&  *+[o][F]{}\ar@{-}[rr]^0\ar@{-}[d]_2 &&  *+[o][F]{}\ar@{-}[rr]^{1} \ar@{-}[d]_{2} && *+[o][F]{} \ar@{-}[rr]^{2}&&  *+[o][F]{} \ar@{-}[rr]^3&&  *+[o][F]{} \ar@{-}[rr]^4 &&  *+[o][F]{} \ar@{-}[rr]^5\ar@{-}[d]_{3}&&  *+[o][F]{} \ar@{.}[rr] \ar@{-}[d]_{3}&&*+[o][F]{}  \ar@{-}[d]^{3}\ar@{-}[rr]^{r-1}&&  *+[o][F]{} \ar@{-}[d]^{3} \\
   &&*+[o][F]{}\ar@{-}[rr]_{0}&& *+[o][F]{}&&&&&&&&  *+[o][F]{} \ar@{-}[rr]_5&&  *+[o][F]{} \ar@{.}[rr]&&*+[o][F]{} \ar@{-}[rr]_{r-1}&&  *+[o][F]{} }$
 \\
 \hline
 \end{tabular}
 \end{center}
 \begin{center}
\begin{tabular}{|cr|}
\hline
(5)& $ \xymatrix@-1.3pc{*+[o][F]{} \ar@{-}[rr]^{h} &&*+[o][F]{} \ar@{.}[rr] &&*+[o][F]{} \ar@{-}[rr]^1 &&*+[o][F]{} \ar@{-}[rr]^0 &&*+[o][F]{}  \ar@{-}[rr]^1&&*+[o][F]{} \ar@{.}[rr]&&  *+[o][F]{} \ar@{-}[rr]^{r-2}&&*+[o][F]{}\ar@{-}[rr]^{r-1}&&*+[o][F]{}  \ar@{-}[rr]^{r-2} &&*+[o][F]{} \ar@{.}[rr]&&*+[o][F]{}\ar@{-}[rr]^{h+1}&&*+[o][F]{} }
$
\\
\hline
 \end{tabular}
 \end{center}
 \begin{center}
\begin{tabular}{|cr|}
\hline

(6a)& $\xymatrix@-1.3pc{*+[o][F]{} \ar@{=}[rr]^{h}_{h-2} &&*+[o][F]{} \ar@{-}[rr]^{h-1} &&*+[o][F]{} \ar@{-}[rr]^{h-2} &&*+[o][F]{} \ar@{-}[rr]^{h-3} &&*+[o][F]{} \ar@{.}[rr] &&*+[o][F]{} \ar@{-}[rr]^0 &&*+[o][F]{} \ar@{.}[rr]&&*+[o][F]{}\ar@{-}[rr]^{r-1} &&*+[o][F]{} \ar@{.}[rr]&&*+[o][F]{}\ar@{-}[rr]^{h}&&*+[o][F]{} }$
\\
&$h>2$\\
(6b) &$ \xymatrix@-1.3pc{*+[o][F]{} \ar@{=}[rr]^{h}_{h-2}&&*+[o][F]{} \ar@{.}[rr] &&*+[o][F]{} \ar@{-}[rr]^1 &&*+[o][F]{} \ar@{-}[rr]^0 &&*+[o][F]{}  \ar@{-}[rr]^1&&*+[o][F]{} \ar@{.}[rr]&&  *+[o][F]{} \ar@{-}[rr]^{r-2}&&*+[o][F]{}\ar@{-}[rr]^{r-1}&&*+[o][F]{}  \ar@{-}[rr]^{r-2} &&*+[o][F]{} \ar@{.}[rr]&&*+[o][F]{}\ar@{=}[rr]^{h}_{h+2}&&*+[o][F]{} }$
\\
\hline
 \end{tabular}
 \end{center}
\end{prop}

\begin{proof}
Let us deal with each graph separately.

(1a) Observe that $\Gamma_{\{r-5,r-4,r-3,r-2\}}$ is a sesqui-extension with respect to $\rho_{r-2}$ of a sesqui-extension with respect to $\rho_{r-5}$ of the sggi (3) of Proposition~\ref{IPfails} thus, by Propositions~\ref{sesqui1} and~\ref{IPfails}, this graph does not give a permutation representation graph of a string C-group. We may deal  with  graph (1b)  similarly.

(2a) In this case $\Gamma_{\{0,1,2,3\}}$ is a sesqui-extension with respect to $\rho_3$ of the graph (6a) of Proposition~\ref{IPfails}. Thus, by Propositions~\ref{sesqui1} and  ~\ref{IPfails}, $\Gamma$ is not a string C-group. We may deal with graph (3a), (3b), (4a) and (4b) similarly (for that, graphs (7) and (8) of Proposition~\ref{IPfails} may be used).

(2b) In this case $\Gamma_{\{h-3,h-2, h-1,h\}}$ is a sesqui-extension  with respect to $\rho_{h-1}$ of a sesqui-extension with respect to $h$ of the sggi having the permutation representation graph (6b) of Proposition~\ref{IPfails}.
Then, by item (d) of Proposition~\ref{sesqui} and Proposition~\ref{IPfails}, we conclude that $\Gamma$ is not a string C-group. We may deal with graphs (2c) similarly.

(5) We have that  $G_0\cap G_{r-1} > G_{0,r-1}$ as the former is a direct product of symmetric groups acting on the intersections of the respective orbits of $G_0$ and $G_{r-1}$ and  the latter is generated by even permutations and thus it has to be of index at least 2 in the former. Hence this case does not satisfy intersection property.

(6a) Let $h > 2$. In this case $G_{>h-3}$ acts as a symmetric group on the points on the left of the first $(h-3)$-edge, $G_{<h}$ acts as a symmetric group on the points on the left of the second $h$-edge and therefore their intersection will be a symmetric group on 4 points, not a dihedral group. Hence the intersection condition is not satisfied. We may deal with graph (6b) similarly.
\end{proof}

\begin{prop}\label{split}
Let  $r\geq n/2$ and $\Gamma$ be a string C-group for $S_n$ having a fracture graph. 
Then either $\Gamma$ has a perfect split or $\Gamma$ has, up to duality, one of the permutation representation graphs of Table~\ref{notperfect} and is of rank $\leq (n+1)/2$.
\end{prop}
\begin{proof}
Let us assume that $\Gamma$ is not one of the string C-group of Table~\ref{notperfect} (nor their duals). In addition suppose that $\Gamma$ does not have a perfect split.

Among the string C-groups for $S_n$ found in Table~\ref{T2F}, all of them have rank $r=(n-1)/2$. Thus there are no string C-groups for $S_n$, with rank 
$r\geq n/2$, having a 2-fracture graph. Hence since $\Gamma$ has a fracture graph but does not have a 2-fracture graph, this means that $\Gamma$ has a split which is not perfect.

Let $i$ be the label of a split $e=\{a,\,b\}$ that is not perfect. Up to duality, assume that $i\in\{0,\ldots,\,\lfloor r/2\rfloor\}$.
Consider the groups $A$ and $B$ corresponding to the pair of $G_i$-orbits. 
If $A$ and $B$ are imprimitive then, by Proposition~\ref{imp}, $\Gamma$ has permutation representation graph (XIV) of Table~\ref{notperfect} or the following permutation representation graph. 
$$
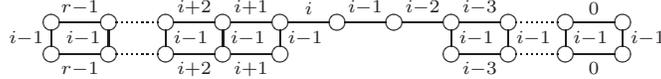
\xymatrix@-1.6pc{ 
 *+[o][F]{}\ar@{-}[rr]^{r-1} \ar@{-}[d]_{i-1} &&  *+[o][F]{}\ar@{-}[d]_{i-1} \ar@{.}[rr] &&  *+[o][F]{}\ar@{-}[rr]^{i+2}\ar@{-}[d]^{i-1} &&*+[o][F]{}\ar@{-}[rr]^{i+1}\ar@{-}[d]^{i-1} && *+[o][F]{} \ar@{-}[d]^{i-1}\ar@{-}[rr]^i&&*+[o][F]{} \ar@{-}[rr]^{i-1}&&*+[o][F]{} \ar@{-}[rr]^{i-2}&&*+[o][F]{}\ar@{-}[d]^{i-1} \ar@{-}[rr]^{i-3}&&*+[o][F]{}\ar@{-}[d]^{i-1} \ar@{.}[rr]&&*+[o][F]{}\ar@{-}[d]^{i-1} \ar@{-}[rr]^0&&*+[o][F]{}\ar@{-}[d]^{i-1} \\
 *+[o][F]{}\ar@{-}[rr]_{r-1}&& *+[o][F]{}\ar@{.}[rr]&&*+[o][F]{}\ar@{-}[rr]_{i+2}&&*+[o][F]{}\ar@{-}[rr]_{i+1}&& *+[o][F]{}&& && && *+[o][F]{}\ar@{-}[rr]_{i-3}&&*+[o][F]{}\ar@{.}[rr]&&*+[o][F]{}\ar@{-}[rr]_0&&  *+[o][F]{}}
 $$
But the latter corresponds to a sggi whose group is embedded into $S_{n/2}\wr S_2$ with $\rho_{i-1}$ swapping the blocks, a contradiction.

Let us assume that $B$ is either primitive or trivial.

 If $B$ is trivial, we distinguish two cases according to whether $i=0$ or not. If $i=0$, then $e$ is a perfect split, a contradiction. If $i\neq 0$ then, by Proposition~\ref{ineq0Btrivial}, $\Gamma$ has the permutation representation graph given in (III) or (XIII) of Table~\ref{notperfect}, a contradiction.
 Hence  $B$ cannot be trivial.
 
 If $B$ is not trivial, it is primitive and, moreover, as the $i$-split cannot be perfect by hypothesis, there exists $h\neq i$ such that $\rho_h$ acts non-trivially on both $G_i$-orbits. 
 Let $h$ be such that $|h-i|$ is maximal.

Let us first assume that $i=0$. 
Then $h>i$ and, by Proposition~\ref{i=0andh=r-1}, as $G\cong S_n$, $h\neq r-1$. 
Then,  by Proposition~\ref{i=0andh<>r-1}, there exists a set of vertices $X$, contained in the $G_i$-orbit fixed by $\rho_{r-1}$, such that
 $h\leq \frac{n-|X|-1}{2}$ with $X:=\{1,\ldots, n\}\setminus {\rm Fix}(G_{>h})$. 
 
Let us now assume $i\neq 0$. 
If $h > i$, we are then in case (a) of Proposition~\ref{h} and we have  $h\leq \frac{n-|X|-1}{2}$  with $X:=\{1,\ldots, n\}\setminus {\rm Fix}(\Gamma_{>h})$. 
For $h<i$ we get a dual result and the rest of the proof can also be adapted to this case accordantly.

We may then assume  $h>i$ and  $h\leq \frac{n-|X|-1}{2}$  with $X:=\{1,\ldots, n\}\setminus {\rm Fix}(\Gamma_{>h})$. We now consider two cases separately: (a) $\Gamma_{>h}$ admits a 2-fracture graph and (b) $\Gamma_{>h}$ does not admit a 2-fracture graph. Additionally  assume that $i$ is the maximal label of a split with the property that there exist $l>i$ such that $\rho_l$ acts non trivially in both $G_i$-orbits.

In case (a) we use Proposition~\ref{t} to conclude that $G_{>h}$ is transitive on $X$. 
The rank of $\Gamma_{>h}$ is equal to $r-1-h$ and, as $r\geq n/2$ and  $h\leq \frac{n-|X|-1}{2}$, we have 
$$r-1-h\geq \frac{|X|-1}{2}.$$
Therefore $\Gamma_{>h}$ must appear in Table~\ref{T2F}, and because of the string condition, $\Gamma_{>h}$ must be one of graphs (9), (10), (13), (14) or (15) of Table \ref{T2F} (the only ones that have a pendant edge with label $0$ or $r-1$, for otherwise we have a contradiction with the definition of $X$). Moreover, as  $r-1-h = (|X|-1)/2$ in these five cases, we have $r=n/2$ and $h= \frac{n-|X|-1}{2}$, which implies that 
$\Gamma_{<h}$ has, as permutation representation graph, one of the graphs with black dots (representing the elements that are not in $X$), given in Propositions~\ref{i=0andh<>r-1}, ~\ref{<itran} and ~\ref{<iint}. 

Combining graph (9) with the graph given by Proposition~\ref{i=0andh<>r-1} gives cases (IV) and (V)  of Table~\ref{notperfect}, a case where there is a perfect split with label $r-1$ (if we put a double edge with labels $(r-3)$ and $(r-1)$ not at the end of the graph) and a case where $G_{r-1}$ is transitive and we do not have a fracture graph (where each $(r-1)$-edge is double with an $(r-3)$-edge).
Combining graph (9) with the graph given by Proposition~\ref{<itran} always gives a perfect split with label $r-3$ (which corresponds to label $i+1$ in the picture of Proposition~\ref{<itran}).
Combining graph (9) with the graph given by Proposition~\ref{<iint}  gives cases (VII), (VIII) and (IX) of Table~\ref{notperfect}, and a case with a perfect split with label $r-1$, a contradiction.

Combining graph (10) with the graphs of Propositions~\ref{i=0andh<>r-1} and ~\ref{<iint} we get the graphs (1a) and (1b) of Proposition~\ref{IPfails2}, and combining graph (10) with the graph of Proposition~\ref{<itran} we get a sggi having a perfect split with label $r-5$. In any case this gives a contradiction.

Combining  the graph (13) with the graphs of Propositions~\ref{i=0andh<>r-1} or ~\ref{<iint} we get one of the graphs, (2a), (2b) or (2c) of Proposition~\ref{IPfails2}, thus the intersection property fails. 
If we combine (13) with the graph of Proposition~\ref{<itran} we get that $\Gamma$ has a perfect split. In any case we get a contradiction.

Finally if we combine graphs (14) or (15) with any of the graphs of Propositions~\ref{i=0andh<>r-1} or ~\ref{<iint} we get one of the graphs, (3a), (3b), (4a) or (4b)  of Proposition~\ref{IPfails2}. If $\Gamma_{>h}$ is as in Proposition~\ref{<itran} then $\Gamma$ has a perfect split. In any case we get a contradiction.


In case (b),  $\Gamma_{>h}$ does not admit a 2-fracture graph.
Now there exists $j>h$ such that $\rho_j$ only swaps one pair of vertices in different $G_j$-orbits. 
Choose $j$ minimal with this property. 
Let $\rho_l=\gamma_l\delta_l$ with $\gamma_l$ and $\delta_l$ being the permutations in each $G_j$-orbit.
Let $J_C := \{l \in \{0, \ldots, r-1\}\setminus\{j\}| \gamma_l \neq 1_G\}$ and $J_D := \{l \in \{0, \ldots, r-1\}\setminus\{j\}| \delta_l \neq 1_G\}$. We then have $C=\langle \gamma_l\mid l\in J_C\rangle$ and $D=\langle \delta_l\mid l\in J_D\rangle $.
Let $L_1$ be the orbit of $G_j$ containing the $i$-edge $\{a,b\}$ and let $L_2$ be the other $G_j$-orbit.

If one of $C$ or $D$ is trivial then, as $\Gamma$ has no perfect split, $j\neq r-1$ and therefore $r\leq \frac{n-1}{2}$ by the dual of Proposition~\ref{ineq0Btrivial}, a contradiction.
Hence both $C$ and $D$ are nontrivial sggi's.

By Propositions~\ref{imp} and~\ref{CCD} either $J_C$ or $J_D$ is an interval.
If $C=G_{<j}$ and $D=G_{>j}$, the $j$-split is perfect, a contradiction.
It remains to consider the case where there exists a permutation $\rho_g$ acting non-trivially on both $G_j$-orbits. 
Thanks to the maximality of $i$,  $g<j$. 
Choose $g$ minimal.
As $g\in J_B$ and $J_B$ must be an interval (as $B$ is primitive), we have that $g>i$. We now consider three cases:


(1) If $j=r-1$, by Proposition~\ref{path}, both $J_C$ and $J_D$ are intervals and $g$ is the minimal label of a permutation acting nontrivially on $L_2$. 
Then we can use the dual of Proposition~\ref{i=0andh<>r-1} to conclude that 
$r-g\leq \frac{n-|Y|-1}{2}$ with $Y:= \{1,\ldots, n\}\setminus {\rm Fix}(G_{<g})\subseteq L_1$. 
If $g< h$ then, by Proposition~\ref{path}, there is a path in $O_2$ from the vertex $x\in L_1$ to the vertex $y\in L_2$ (as in the picture below), containing all labels from $g$ to $r-2$ twice.
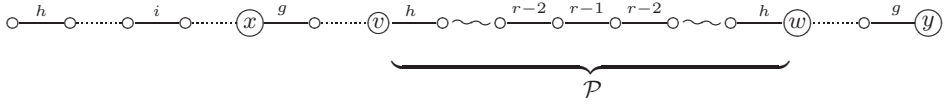
\begin{small}
\[\begin{array}{c}
\xymatrix@-1.55pc{  *+[o][F]{} \ar@{-}[rr]^{h}&&*+[o][F]{} \ar@{.}[rr] &&*+[o][F]{}\ar@{-}[rr]^i&&*+[o][F]{}\ar@{.}[rr] && *+[o][F]{x}\ar@{-}[rr]^g&& *+[o][F]{}\ar@{.}[rr] &&*+[o][F]{v}\ar@{-}[rr]^h&&*+[o][F]{} \ar@{~}[rr] &&*+[o][F]{}\ar@{-}[rr]^{r-2}&&*+[o][F]{}\ar@{-}[rr]^{r-1}&&*+[o][F]{}\ar@{-}[rr]^{r-2}&& *+[o][F]{}\ar@{~}[rr] &&*+[o][F]{}\ar@{-}[rr]^h&&*+[o][F]{w}\ar@{.}[rr] &&*+[o][F]{}\ar@{-}[rr]^g&& *+[o][F]{y}}\\
\hspace{88pt}\underbrace{\hspace{150pt}}\\
\hspace{90pt}\mathcal{P}\\
\end{array}\] 
\end{small}
 This path is an extension of a path $\mathcal P$ containing all labels from $h$ to $r-2$ twice. 
Let $V$ be the set of vertices of $\mathcal P$.
Then $r-h\leq\frac{|V|}{2}$. As $h\leq \frac{n-|X|-1}{2}$, we have $n/2\leq r\leq \frac{n-|X|+|V|-1}{2}$, implying that $|V| \geq |X| +1$.
Let $v$ and $w$ be as in the figure above.
As $V\setminus\{v,w\}\subseteq (\{1, \ldots, n\}\setminus Fix(G_{>h}))$,  $|V| \leq |X| + 2$. 
Moreover there must be at least two vertices in $L_2$ that belong to the extended path, namely those to the left of $v$ in the picture, and they cannot be taken into account when bounding $h$. Thus we must have that $h \leq ((n-2)-|X|-1)/2$ and $r\leq (n-|X|-3+|V|)/2 \leq (n-1)/2$, a contradiction.
Hence $g \geq h$.

(2) Let $j\neq r-1$ and $J_D$  be an interval. In this case both $\rho_0$ and $\rho_{r-1}$ act nontrivially on the first $G_j$-orbit. Hence we can use the dual argument to the one used in the last paragraph of the proof of Proposition~\ref{<iint}, to conclude that  we get the permutation representations $(I)$ or $(II)$ of Table~\ref{notperfect}, a contradiction.

(3) Let $j\neq r-1$ and $J_C$ be an interval. By the dual of Proposition~\ref{h}, $g>0$ and  $r-g\leq \frac{n-|Y|-1}{2}$ with $Y:= \{1,\ldots, n\}\setminus {\rm Fix}(G_{<g})\subseteq L_1$. By the same reasoning as in (1), $h\leq g$.

In conclusion, only cases (1) and (3) are possible, and both give the inequality  $r-g\leq \frac{n-|Y|-1}{2}$ with $Y:= \{1,\ldots, n\}\setminus {\rm Fix}(G_{<g})\subseteq L_1$ and  $h\leq g$.

By choice of $j$, there is no other label $l$ between $h$ and $g$ having only one pair of vertices in different $G_l$-orbits, hence there are three possibilities: either $\Gamma_{\{h+1, \ldots, g-1\}}$ admits a 2-fracture graph, or $h+1=g$ or $h=g$. 

Suppose first that $\Gamma_{\{h+1,\ldots, g-1\}}$ has a 2-fracture graph (which has $|X\cap Y|$ vertices).
If that 2-fracture graph is disconnected, by Proposition~\ref{frac1}, $g-h-1\leq \frac{|X\cap Y|-1}{2}$. 
We have the following inequalities.
$$h\leq \frac{n-|X|-1}{2},\;g-h\leq \frac{|X\cap Y|+1}{2}\mbox{ and }  r-g \leq \frac{n-|Y|-1}{2}.$$
As in addition, $n= |X| + |Y| - |X \cap Y|$, $r\leq (n-1)/2$, a contradiction.
Hence the 2-fracture graph is connected. To avoid the same contradiction as before, we have to get, out of Table~\ref{T2F}, those graphs that have rank $|X \cap Y|/2$. Thus we  must have one of the first three graphs of Table~\ref{T2F}. Also $\Gamma_{<h}$ and $\Gamma_{>g}$ have one of the permutation representation graphs, up to duality, having black dots, given in Propositions~\ref{i=0andh<>r-1}, ~\ref{<itran} and ~\ref{<iint}, which is not possible.

Suppose next that $h+1= g$. 
If either $h < (n-|X|-1)/2$ or $r-g< (n-|Y|-1)/2$, then  we get $r<  \frac{n -|X \cap Y|}{2}$ as shown below.
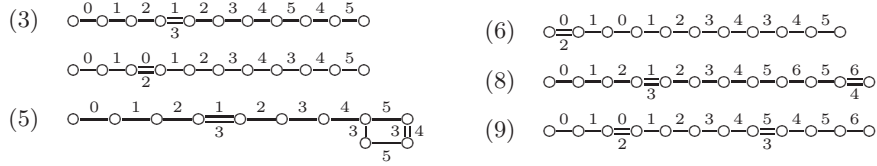
\begin{center}
\begin{tabular}{lcl}
\begin{tabular}{rl}
$r$&$\leq \frac{n-|Y|-1}{2} + g$\\[5pt]
&$ =  \frac{n-|Y|-1}{2} + h + 1 $\\[5pt]
&$< \frac{n-|Y|-1}{2} + \frac{n-|X|-1}{2}+1$ \\[5pt]
&$= \frac{2n - |X| - |Y|}{2}$ \\[5pt]
&$= \frac{n -|X \cap Y|}{2}$
\end{tabular}
&
\begin{tabular}{rl}
$r$&$< \frac{n-|Y|-1}{2} + g$\\[5pt]
&$=\frac{n-|Y|-1}{2} + h + 1 $\\[5pt]
&$\leq \frac{n-|Y|-1}{2} + \frac{n-|X|-1}{2}+1$ \\[5pt]
&$= \frac{2n - |X| - |Y|}{2}$ \\[5pt]
&$= \frac{n -|X \cap Y|}{2}$
\end{tabular}\\
\end{tabular}
\end{center}
So we can assume that $h = (n-|X|-1)/2$ and $r-g = (n-|Y|-1)/2$.
Then the possibilities for $\Gamma_{<h}$ and $\Gamma_{>h+1}$ are given by Propositions~\ref{i=0andh<>r-1}, ~\ref{<itran} or ~\ref{<iint}.
But we can exclude the possibility of having the graph of  Proposition~\ref{<itran}, as in that case either $h+1$ or $h$ is the label of a perfect split.
Suppose $\Gamma_{<h}$ and $\Gamma_{>h+1}$ are as in Proposition~~\ref{i=0andh<>r-1}. 
Then the permutation representation graph the graph (5) of Proposition~\ref{IPfails2}. Suppose that $\Gamma_{<h}$ as in Proposition~\ref{<iint}, then we get graphs (X)  and (XIV)   of Table~\ref{notperfect}. In any case we get a contradicition.

Finally, suppose $h=g$. Using similar arguments as in the case $h+1=g$, we see that we only need to consider the cases where $\Gamma_{<h}$ and $\Gamma_{>h}$ are as in Propositions~\ref{i=0andh<>r-1}, ~\ref{<itran} or ~\ref{<iint}.

Suppose first that $\Gamma_{<h}$ and $\Gamma_{>h}$ are as in Proposition~\ref{i=0andh<>r-1}. Then we get the following possibilities: the graph (XI) or the graph (XII)  of Table~\ref{notperfect}; the graphs (6a) or (6b) of Proposition~\ref{IPfails2}; or $h=2$ and $2$ is the label of a perfect split. In any case we get a contradiction. 
If both $\Gamma_{<h}$ and $\Gamma_{>h}$ are as in Proposition~\ref{<itran} then $\Gamma$ has a perfect split, a contradiction.
If both are as in are as in Proposition~\ref{<iint} then we get the following permutation representation graph.
$$\xymatrix@-1.6pc{ 
 *+[o][F]{}\ar@{-}[rr]^{0} \ar@{-}[d]_{h} &&  *+[o][F]{}\ar@{-}[d]_{h} \ar@{.}[rr] &&  *+[o][F]{}\ar@{-}[rr]^{h-3}\ar@{-}[d]^{h} &&*+[o][F]{}\ar@{-}[rr]^{h-2}\ar@{-}[d]^{h} && *+[o][F]{} \ar@{-}[d]^{h}\ar@{-}[rr]^{h-1}&&*+[o][F]{} \ar@{-}[rr]^{h}&&*+[o][F]{} \ar@{-}[rr]^{h+1}&&*+[o][F]{}\ar@{-}[d]^{h} \ar@{-}[rr]^{h+3}&&*+[o][F]{}\ar@{-}[d]^{h} \ar@{.}[rr]&&*+[o][F]{}\ar@{-}[d]^{h} \ar@{-}[rr]^{r-1}&&*+[o][F]{}\ar@{-}[d]^{h} \\
 *+[o][F]{}\ar@{-}[rr]_{0}&& *+[o][F]{}\ar@{.}[rr]&&*+[o][F]{}\ar@{-}[rr]_{h-3}&&*+[o][F]{}\ar@{-}[rr]_{h-2}&& *+[o][F]{}&& &&  && *+[o][F]{}\ar@{-}[rr]_{h+3}&&*+[o][F]{}\ar@{.}[rr]&&  *+[o][F]{} \ar@{-}[rr]_{r-1}&&*+[o][F]{}}
 $$
But then $G$ is embedded into $S_{n/2}\wr S_2$ with $\rho_h$ swapping the blocks, a contradiction.
The remaining possibilities correspond to graphs  (XV), (XVI) and (XVII) of Table~\ref{notperfect}, a contradiction.
\end{proof}

A string C-group $\Gamma = (G, \{ \rho_0, \ldots, \rho_{r-1}\})$ is {\em split-attachable} if its permutation representation graph has either a 0-edge or a $(r-1)$-edge that is pendant (i.e. has a vertex of degree one) and if $\Gamma$ can be extended to a larger string C-group that has a perfect split attached to the vertex of degree one.

\begin{prop}\label{att}
Let $\Gamma$ be one the  sggi  of Tables~\ref{an}, \ref{BI}, \ref{primPolys}, \ref{T2F} or \ref{notperfect}.
If $\Gamma$ is a string C-group that is split-attachable then $\Gamma$ has one of the following permutation representation graphs:
the graph (9) of Table~\ref{T2F} or  the graphs (2), (3), (5) and (8) of Table~\ref{primPolys}.

\end{prop}
\begin{proof}
 Let $\Gamma$ be one of the string C-groups of Table~\ref{an}. The sggi's with permutation representation graphs (1), (2), (4) and (7) are not split-attachable thanks to the commuting property. 
Let us consider the remaining permutation representation graphs of Table~\ref{an}. 
Attaching one perfect split to these graphs we get, up to a duality, the following  sggi's.

\begin{small}
\begin{tabular}{cc}
\begin{tabular}{cl}
(3)& $\xymatrix@-2pc{*+[o][F]{} \ar@{-}[rr]^0 &&*+[o][F]{} \ar@{-}[rr]^1 && *+[o][F]{}  \ar@{-}[rr]^2&&*+[o][F]{} \ar@{=}[rr]^1_3 && *+[o][F]{} \ar@{-}[rr]^2 && *+[o][F]{} \ar@{-}[rr]^3&&*+[o][F]{}  \ar@{-}[rr]^4 && *+[o][F]{} \ar@{-}[rr]^5&& *+[o][F]{} \ar@{-}[rr]^4&&*+[o][F]{}  \ar@{-}[rr]^5&&*+[o][F]{}  }$ \\
& $\xymatrix@-2pc{*+[o][F]{} \ar@{-}[rr]^0 && *+[o][F]{}  \ar@{-}[rr]^1&&*+[o][F]{} \ar@{=}[rr]^0_2 && *+[o][F]{} \ar@{-}[rr]^1 && *+[o][F]{} \ar@{-}[rr]^2&&*+[o][F]{}  \ar@{-}[rr]^3 && *+[o][F]{} \ar@{-}[rr]^4&& *+[o][F]{} \ar@{-}[rr]^3&&*+[o][F]{}  \ar@{-}[rr]^4&&*+[o][F]{}  \ar@{-}[rr]^5&&*+[o][F]{}}$ \\
(5)&$\xymatrix@-1.8pc{*+[o][F]{} \ar@{-}[rr]^0 &&*+[o][F]{} \ar@{-}[rr]^1&& *+[o][F]{}  \ar@{-}[rr]^2&&*+[o][F]{} \ar@{=}[rr]^1_3 && *+[o][F]{} \ar@{-}[rr]^2 && *+[o][F]{} \ar@{-}[rr]^3&&*+[o][F]{}  \ar@{-}[rr]^4 && *+[o][F]{} \ar@{-}[rr]^5&& *+[o][F]{}\\
&&&& && && && &&  &&*+[o][F]{}  \ar@{-}[rr]_5 \ar@{-}[u]^3&&*+[o][F]{}  \ar@{=}[u]_4^3}$
\end{tabular}
&
\begin{tabular}{cl}
(6)&$\xymatrix@-2pc{*+[o][F]{} \ar@{=}[rr]^0_2  && *+[o][F]{}  \ar@{-}[rr]^1&&*+[o][F]{} \ar@{-}[rr]^0 && *+[o][F]{} \ar@{-}[rr]^1 && *+[o][F]{} \ar@{-}[rr]^2&&*+[o][F]{}  \ar@{-}[rr]^3 && *+[o][F]{} \ar@{-}[rr]^4&& *+[o][F]{} \ar@{-}[rr]^3&&*+[o][F]{}  \ar@{-}[rr]^4&&*+[o][F]{} \ar@{-}[rr]^5&&*+[o][F]{} }$\\
(8)& $\xymatrix@-2pc{*+[o][F]{} \ar@{-}[rr]^0  &&*+[o][F]{} \ar@{-}[rr]^1  && *+[o][F]{}  \ar@{-}[rr]^2&&*+[o][F]{} \ar@{=}[rr]_3^1 && *+[o][F]{} \ar@{-}[rr]^2 && *+[o][F]{} \ar@{-}[rr]^3&&*+[o][F]{}  \ar@{-}[rr]^4 && *+[o][F]{} \ar@{-}[rr]^5&& *+[o][F]{} \ar@{-}[rr]^6&&*+[o][F]{}  \ar@{-}[rr]^5&&*+[o][F]{}  \ar@{=}[rr]^6_4&& *+[o][F]{}}$\\
(9)& $\xymatrix@-2pc{*+[o][F]{} \ar@{-}[rr]^0&& *+[o][F]{}  \ar@{-}[rr]^1&&*+[o][F]{} \ar@{=}[rr]_2^0 && *+[o][F]{} \ar@{-}[rr]^1 && *+[o][F]{} \ar@{-}[rr]^2&&*+[o][F]{}  \ar@{-}[rr]^3 && *+[o][F]{} \ar@{-}[rr]^4&& *+[o][F]{} \ar@{=}[rr]^5_3&&*+[o][F]{}  \ar@{-}[rr]^4&&*+[o][F]{}  \ar@{-}[rr]^5&& *+[o][F]{} \ar@{-}[rr]^6&& *+[o][F]{}}$ \\
\end{tabular}
\end{tabular}
\end{small}\\
In any case either we find  a set of indices $J$ such that the permutation representation graph of $\Gamma_J$ is either the graph (3) or the graph (9) of Proposition~\ref{IPfails} or a sesqui-extension of one of those sggi with respect to the first or the last generator. Thus by Propositions~\ref{sesqui1} and~\ref{IPfails} these are not string C-groups.

The graphs of Table~\ref{BI} cannot be attached to a perfect split thanks to the commuting property.

Let $\Gamma$ be one of the string C-groups of Table~\ref{primPolys}. 
By the commuting property, the string C-groups with permutation representation graphs (1), (4), (6) and (7) are not split attachable. All the remaining graphs (2), (3), (5) and (8) are listed in the statement of this proposition.

Let $\Gamma$ be one of the string C-groups of Table~\ref{T2F}.
 The only candidates to be split-attachable are graphs (9),(10), (13), (14) and (15). 
 In what follows we rule all these cases out except (9) which is in the statement of the proposition.
The possibilities that we now rule out are the following.
 
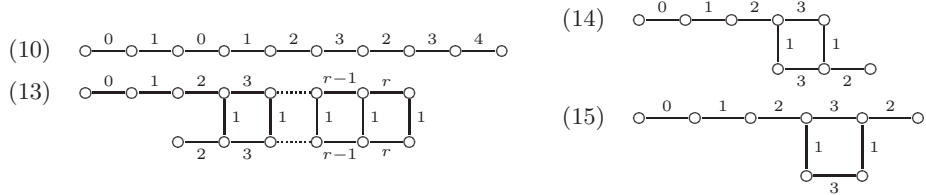
\begin{small}
\begin{tabular}{cc}
\begin{tabular}{cl}
 (10)&$ \xymatrix@-1pc{ *+[o][F]{}  \ar@{-}[r]^0 & *+[o][F]{}  \ar@{-}[r]^1 &*+[o][F]{}  \ar@{-}[r]^0 & *+[o][F]{} \ar@{-}[r]^1 & *+[o][F]{}  \ar@{-}[r]^2 & *+[o][F]{} \ar@{-}[r]^3 & *+[o][F]{}   \ar@{-}[r]^2  & *+[o][F]{}  \ar@{-}[r]^3  & *+[o][F]{}  \ar@{-}[r]^4  & *+[o][F]{} } $\\
 (13)&$\xymatrix@-1pc{*+[o][F]{}  \ar@{-}[r]^0 & *+[o][F]{}  \ar@{-}[r]^1 & *+[o][F]{} \ar@{-}[r]^2 & *+[o][F]{} \ar@{-}[r]  \ar@{-}[r]^3 & *+[o][F]{}  \ar@{.}[r] & *+[o][F]{}  \ar@{-}[r]^{r-1}  & *+[o][F]{}  \ar@{-}[r]^{r}  & *+[o][F]{}  &\\
&&*+[o][F]{}  \ar@{-}[r]_2 & *+[o][F]{} \ar@{-}[u]_1  \ar@{-}[r]_3  & *+[o][F]{}  \ar@{.}[r] \ar@{-}[u]_1 &  *+[o][F]{}  \ar@{-}[r]_{r-1}\ar@{-}[u]_1&*+[o][F]{}  \ar@{-}[r]_{r}\ar@{-}[u]_1&  *+[o][F]{} \ar@{-}[u]_1  } $\\
\end{tabular}
&
\begin{tabular}{cl}
 (14) &$ \xymatrix@-1pc{ *+[o][F]{}  \ar@{-}[r]^0&*+[o][F]{}  \ar@{-}[r]^1 &   *+[o][F]{}  \ar@{-}[r]^2 & *+[o][F]{}  \ar@{-}[r]^3 & *+[o][F]{}   &\\
& &  &  *+[o][F]{}  \ar@{-}[r]_3\ar@{-}[u]_1 &*+[o][F]{}  \ar@{-}[r]_2\ar@{-}[u]_1&  *+[o][F]{} }$\\
  (15) &$ \xymatrix@-0.7pc{*+[o][F]{}  \ar@{-}[r]^0& *+[o][F]{}  \ar@{-}[r]^1 &   *+[o][F]{}  \ar@{-}[r]^2 & *+[o][F]{}  \ar@{-}[r]^3 & *+[o][F]{}   \ar@{-}[r]^2 & *+[o][F]{} \\
& &  &  *+[o][F]{}  \ar@{-}[r]_3\ar@{-}[u]_1 &*+[o][F]{} \ar@{-}[u]_1& } $\\
 \end{tabular}
  \end{tabular}
\end{small}

In any case there exists $J$ such that $\Gamma_J$ has, as permutation representation graph, one of the graphs of Proposition~\ref{IPfails} or a sesqui-extension, with respect to the first or the last generator, of one of the graphs of Proposition~\ref{IPfails}. 
For graph (10) $\Gamma_{0}$ is a sesqui-extension of graph (3) of Proposition~\ref{IPfails};  for graph (13) $\Gamma_{<4}$ has the graph (6a), (6b) or is a sesqui-extension with respect to $\rho_1$ of the graph (6b) of Proposition~\ref{IPfails}; The graphs (14) and (15) are, respectively, the graphs (7) and (8) of  Proposition~\ref{IPfails}.

Finally let $\Gamma$ be one of the sggi's of Table~\ref{notperfect}. Only the graphs (IV) and (VII) can be extended with a split in such a way that the result is a sggi. But as intersection property fails, these graphs are not split-attachable. Indeed, we have the following graphs.\

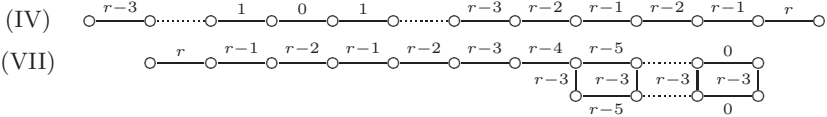
\begin{small}
\begin{tabular}{cr}
 (IV)&$ \xymatrix@-1.5pc{*+[o][F]{} \ar@{-}[rr]^{r-3} &&*+[o][F]{} \ar@{.}[rr] &&*+[o][F]{} \ar@{-}[rr]^1 &&*+[o][F]{} \ar@{-}[rr]^0 &&*+[o][F]{}  \ar@{-}[rr]^1&&*+[o][F]{} \ar@{.}[rr]&&  *+[o][F]{} \ar@{-}[rr]^{r-3} &&*+[o][F]{} \ar@{-}[rr]^{r-2}&&*+[o][F]{}\ar@{-}[rr]^{r-1}&&*+[o][F]{}  \ar@{-}[rr]^{r-2}&&*+[o][F]{}\ar@{-}[rr]^{r-1}&&*+[o][F]{}\ar@{-}[rr]^{r}&&*+[o][F]{}  }$\\
 
(VII)&$\xymatrix@-1.5pc{*+[o][F]{} \ar@{-}[rr]^{r} &&*+[o][F]{} \ar@{-}[rr]^{r-1} &&*+[o][F]{} \ar@{-}[rr]^{r-2} &&*+[o][F]{} \ar@{-}[rr]^{r-1} &&*+[o][F]{} \ar@{-}[rr]^{r-2} &&*+[o][F]{} \ar@{-}[rr]^{r-3} &&*+[o][F]{} \ar@{-}[rr]^{r-4} &&*+[o][F]{}\ar@{-}[rr]^{r-5}\ar@{-}[d]_{r-3}&&*+[o][F]{}\ar@{-}[d]_{r-3}*+[o][F]{} \ar@{.}[rr]&&*+[o][F]{}\ar@{-}[rr]^{0}\ar@{-}[d]_{r-3}&&*+[o][F]{}\ar@{-}[d]_{r-3}&&\\
&& && && && && && && *+[o][F]{} \ar@{-}[rr]_{r-5}&& *+[o][F]{}*+[o][F]{} \ar@{.}[rr]&&*+[o][F]{} \ar@{-}[rr]_{0}&& *+[o][F]{}*&&\\
 }$
\end{tabular}
\end{small}\\

In both cases $\Gamma_{>r-4}$ is a sesqui-extension with respect to  $\rho_{r-3}$ of the graph (3) of Proposition~\ref{IPfails}.
\end{proof}

\section{r\&d extension}\label{RD}

In this section we construct a string C-group of rank $r+1$ and degree $n+1$ from a string C-group of rank $r$ and degree $n$. We will refer to this construction as the $r\&d$-extension.

We assume in this section that $\Gamma = (G,\{ \rho_0,\ldots, \rho_{r-1}\})$ is a sggi whose permutation representation graph, with vertex-set $\{1,\ldots, n\}$, and that $\Gamma$ has a perfect $i$-split $\{a,b\}$.
In addition, let $\{O_1,\,O_2\}$ be a partition of $\{1,\ldots, n\}$ with $a\in O_1$, $b\in O_2$ such that $G_{<i}$ and $G_{>i}$ are, respectively, the actions of $G_i$ on $O_1$ and $O_2$.
Let $\rho_i=\alpha_i\beta_i(a,b)$ with $\alpha_i$ and $\beta_i$ being the actions of $\rho_i$ in $O_1$ and $O_2$ respectively.  
Now we extend the sggi $\Gamma$ to another sggi $\Gamma^{i\uparrow} = (G^{i\uparrow}, S^{i\uparrow})$ where $G$ is acting on a set $\{1,\ldots, n\}\cup \{c\}$ of size $n+1$, $\Gamma^{i\uparrow}$ has two adjacent perfect splits $\{a,c\}$ and $\{c,b\}$, and such that $G^{i\uparrow}_{>i+1}\cong G_{>i}$ and $G^{i\uparrow}_{<i}\cong G_{<i}$.
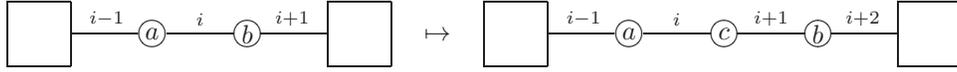
$$\xymatrix@-1.2pc{*++++[F]{}\ar@{-}[rr]^(.6){i-1} &&*+[o][F]{a} \ar@{-}[rr]^i&&*+[o][F]{b} \ar@{-}[rr]^(.4){i+1}&& *++++[F]{} &\mapsto &
*++++[F]{}\ar@{-}[rr]^(.6){i-1} &&*+[o][F]{a} \ar@{-}[rr]^i&& *+[o][F]{c}  \ar@{-}[rr]^{i+1} &&*+[o][F]{b} \ar@{-}[rr]^(.4){i+2}&&*++++[F]{} }
$$ 
More precisely,
$S^{i\uparrow} = \{\delta_0, \ldots, \delta_r\}$ and 
$G^{i\uparrow}=\langle S \rangle$ where
$$\delta_j=\left\{\begin{array} {l}\rho_j, \mbox{ if }   j\in \{0,\ldots, i-1\},\\[5pt]
 \alpha_i(a,c),\mbox{ if } j=i,\\[5pt]
 \beta_i(c,b),\mbox{ if }j=i+1,\\[5pt]
 \rho_{j-1}, \mbox{ if } j\in\{i+2,\ldots, r\}.
 \end{array}\right..
 $$
By construction $\Gamma^{i\uparrow} := (G^{i\uparrow},S^{i\uparrow})$ is a sggi, $G^{i\uparrow}$ is of degree $n+1$ and $\Gamma^{i\uparrow}$ has rank $r+1$. We will say that  $\Gamma^{i\uparrow}$ is a \emph{rank  and degree extension of $\Gamma$}, for short  \emph{r\&d extension of $\Gamma$}, with respect to the perfect $i$-split}.

We now prove that the r\&d extension of a sggi is either a symmetric or an alternating group.
%
%
%
We first recall some facts about permutation groups containing a $3$-cycle.

\begin{prop}\label{3cyc}
\begin{itemize}
\item[(a)] Let $G$ be a permutation group containing a $3$-cycle. Then $G$
has a normal subgroup which is a direct product of alternating groups in
their natural representations.
\item[(b)] Let $G$ be a primitive permutation group containing a $3$-cycle.
Then $G$ is the alternating or symmetric group.
\item[(c)] Let $G$ be an intransitive permutation group containing a 
$3$-cycle $\alpha$. Let $X$ be the orbit of one of the points of $\alpha$,
and $H$ the group induced on $X$ by $G$. If $A_X\le H$, then $A_X\le G$.
\end{itemize}
\end{prop}

\begin{proof}
Let $G$ act on $\{1,\ldots,n\}$. Define a relation $\equiv$ on this set by
the rule
\[(a\equiv b) \Leftrightarrow (a=b\hbox{ or }(\exists c)((a,b,c)\in G)).\]
We claim that this is an equivalence relation. It is clearly reflexive and
symmetric. Suppose that $a\equiv b$ and $b\equiv c$. Then we have two
$3$-cycles whose supports both contain $b$, and hence which generate a group
on $3$, $4$ or $5$ points. It is readily checked that this group is the
alternating group. Hence $a\equiv c$; and we also see that $(a,b,c)\in G$.

Now it follows that, if $X$ is an equivalence class of $\equiv$, then $G$
contains every $3$-cycle with support contained in $X$; these generate the
alternating group on $X$. So part (a) of the Proposition follows.

If $G$ is primitive and contains a $3$-cycle, then the relation $\equiv$ is
trivial but is not the relation of equality; so there is a single equivalence
class, and $G$ contains the alternating group $A_n$. So (b) holds.

Finally we prove (c). Under the hypothesis of this part, all $3$-cycles on
$X$ are conjugates of $\alpha$ or $\alpha^{-1}$ (since, for $\beta\in G$, the
conjugate $\alpha^\beta$ depends only on the action of $\beta$ on $X$). So $X$
is an equivalence class of $\equiv$, and the result follows from the proof of (a).
\end{proof}

\begin{prop}\label{sntosn}
Let $\Gamma = (G,S)$ be a sggi generated by $r$ involutions with a perfect $i$-split.
If $G$ is a transitive group of degree $n$ then $A_{n+1}\leq G^{i\uparrow}$. 
\end{prop}
\begin{proof}
By Proposition~\ref{primpb},  $G^{i\uparrow}$ is primitive. Moreover $(\delta_i\delta_{i+1})^2$ is a 3-cycle.
By Proposition~\ref{3cyc}(b),  $G^{i\uparrow}\geq A_{n+1}$.
\end{proof}

 

The r\&d-extension of a string C-group is not in general a string C-group. For instance $\Gamma= (S_4, \{(1,2)(3,4),\,(2,3), \,(3,4)\})$ is the string C-group obtained from the automorphism group of the hemi-cube acting on its vertices. However $\Gamma^{0\uparrow} = (S_5,\{(5,1), (1,2)(3,4),\,(2,3), \,(3,4)\})$ is not a string C-group.
In what follows we give sufficient conditions under which this operation preserves the intersection property.
In what follows transitivity of $G$ is not required.


\begin{prop}\label{i1r3}
Let $\Gamma = (G,S)$  be a string C-group of rank $3$ with a perfect $1$-split $\{a,b\}$. If $\{a,b\}\cap Fix(G_1)=\emptyset$ then $\Gamma^{1\uparrow}$ is a string C-group.  \end{prop}
\begin{proof}
If $G$ is transitive, there is only one possibility for the permutation representation graph on $\Gamma$ and thus 
$$\Gamma = (S_4,\{(1,2),\, (2,3),\, (3,4)\}).$$
Then 
$$\Gamma^{1\uparrow}=(S_5,\{ (1,2),\, (2,5),\, (5,3), (3,4)\})$$
which is a string C-group, namely the 4-simplex. 

Suppose now that $G$ is intransitive.
 There is only one possibility for the connected component of the permutation representation graph of $\Gamma$  containing the perfect $1$-split which is the following.
$$
\xymatrix@-1.3pc{*+[o][F]{} \ar@{-}[rr]^0 &&*+[o][F]{} \ar@{-}[rr]^1 &&*+[o][F]{} \ar@{-}[rr]^2 &&*+[o][F]{}}
$$

The remaining components are either single $0$-edges, single $2$-edges, double edges with labels $\{0,1\}$ or $\{1,2\}$. 
In any case $\Gamma^{1\uparrow}$ is a sesqui-extension of a string C-group representation of $S_5$, which is a string C-group by Proposition~\ref{sesqui0}.
\end{proof}

\begin{thm}\label{RDE}
Let $\Gamma = (G,S)$  be a string C-group with a perfect $i$-split and let $\rho_i=\alpha_i(a,b)\beta_i$ as defined at the beginning of this section.  Suppose that $\Gamma$ has a fracture graph and that  $\{a,b\}\cap Fix(G_i)=\emptyset$.
If either $\alpha_i$ or $\beta_i$ is trivial, then $\Gamma^{i\uparrow}$ is a string C-group.
\end{thm}
\begin{proof}

As $\{a,b\}\cap Fix(G_i)=\emptyset$, $i\notin\{ 0,\,\,r-1\}$ and $r\geq 3$. If $r=3$ then the theorem holds by Proposition~\ref{i1r3}. Let $r> 3$.
Assume, by induction, that this theorem holds for any rank smaller than $r$. 
Suppose that $\alpha_i$ is trivial (the same arguments will be valid if, instead, $\beta_i$ is trivial). 
Since $i\not\in\{0, r-1\}$, we have the following equalities.
\[(\Gamma^{i\uparrow})_0\cong (\Gamma_0)^{i\uparrow}\mbox{ and }
(\Gamma^{i\uparrow})_r\cong (\Gamma_{r-1})^{i\uparrow}.\]
Suppose that $\{a,b\}\cap Fix(G_{0,i})\neq \emptyset$. 
In that case $i=1$. 
Let $\mathcal G$ be the permutation representation graph of $\Gamma$.
If $\mathcal G$ is connected then $0$ is also a perfect split, and $(\Gamma^{i\uparrow})_0$ is isomorphic to $\Gamma$, hence a string C-group.
If $\mathcal G$ is disconnected then,  one of the components of $\mathcal G$ has the $1$-split connected to a 0-edge, 
and the remaining components not fixed by $\rho_0$ have exactly two vertices, that are either connected by a single $0$-edge or with a double $\{0,1\}$-edge.  Hence the connected components having 0-edges are of the following forms.

$$
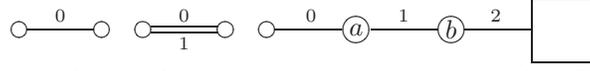
\xymatrix@-1.2pc{*+[o][F]{}\ar@{-}[rr]^0&&*+[o][F]{} & *+[o][F]{}\ar@{=}[rr]^0_1&&*+[o][F]{}  & *+[o][F]{}\ar@{-}[rr]^{0} &&*+[o][F]{a} \ar@{-}[rr]^1&&*+[o][F]{b} \ar@{-}[rr]^(.4){2}&& *++++[F]{}}$$

Thus $\Gamma$ is a  sesqui-extension with respect to $\rho_0$ of a string C-group $\Delta$ whose permutation representation graph is obtained from the one above by removing the 0-edges from the connected components of size 2.

$$
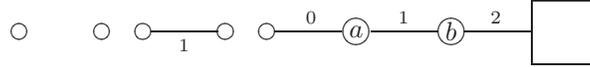
\xymatrix@-1.2pc{*+[o][F]{}\ar@{}[rr]&&*+[o][F]{} & *+[o][F]{}\ar@{-}[rr]_1&&*+[o][F]{}  & *+[o][F]{}\ar@{-}[rr]^{0} &&*+[o][F]{a} \ar@{-}[rr]^1&&*+[o][F]{b} \ar@{-}[rr]^(.4){2}&& *++++[F]{}}$$

But then $\Gamma^{1\uparrow}$ has the following graph.

$$
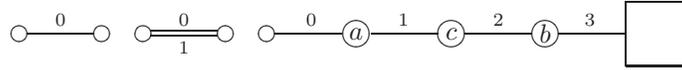
\xymatrix@-1.2pc{*+[o][F]{}\ar@{-}[rr]^0&&*+[o][F]{} & *+[o][F]{}\ar@{=}[rr]^0_1&&*+[o][F]{} &*+[o][F]{}\ar@{-}[rr]^{0} &&*+[o][F]{a} \ar@{-}[rr]^1&& *+[o][F]{c}  \ar@{-}[rr]^{2} &&*+[o][F]{b} \ar@{-}[rr]^(.4){3}&&*++++[F]{} }$$ 

Hence $(\Gamma^{1\uparrow})_0$ is precisely $\Delta$ (up to a renumbering of the labels on the edges and an extra isolated vertex) and thus it is a string C-group as well.

The same argument can be used when $\{a,b\}\cap Fix(G_{r-1,i})\neq \emptyset$.
 
Assume that $\{a,b\}\cap Fix(G_{0,i})= \emptyset$ and $\{a,b\}\cap Fix(G_{r-1,i})= \emptyset$.
Then, by induction,  $(\Gamma^{i\uparrow})_0$ and $(\Gamma^{i\uparrow})_r$ are string C-groups. 
   
By Proposition~\ref{arp}, it remains to prove that $(G^{i\uparrow})_0\cap( G^{i\uparrow})_r=(G^{i\uparrow})_{0,r}$.

Note that $G_{0,r-1}$  has at least three orbits. 
Let $X:=a^{G_{0,r-1}}$, $Z := Fix(G_{<i}) \setminus X$ and $Y := Fix(G_{>i}) \setminus X$.
Obviously, $\{X,\,Y,\,Z\}$ is a partition of $\{1,\ldots, n\}$.
Consider the group actions $H$ and $K$ of $G_{0,r-1}$ on $Y$ and $Z$ respectively. 
Let $G^{i\uparrow}=\langle\delta_0,\ldots,\delta_r\rangle$.

As $(\delta_i\delta_{i+1})^2=(a,b,c)\in (G^{i\uparrow})_{0,r}$, by Proposition~\ref{sntosn} the group action of $(G^{i\uparrow})_{0,r}$ on $X':=X\cup \{c\}$ contains the alternating group $A_{X'}$.
Thus, by Proposition~\ref{3cyc}(c),
$(G^{i\uparrow})_{0,r}$ contains $A_{X'}$. 
As $\alpha_i$ is trivial, $\delta_i=(a,c)\in (G^{i\uparrow})_{0,r}$, hence  $(G^{i\uparrow})_{0,r}$ is  isomorphic to $S_{X'}\times H\times K$.
The group action of  $(G^{i\uparrow})_0$ on $Y\cup Z$ is $H\times K$.
Hence $(G^{i\uparrow})_0\cap( G^{i\uparrow})_r$ is a subgroup of $(G^{i\uparrow})_{0,r}$, as wanted and therefore $(G^{i\uparrow})_0\cap( G^{i\uparrow})_r =(G^{i\uparrow})_{0,r}$. 
\end{proof}
\section{Sesqui-extensions of symmetric and alternating groups}\label{se}

We say that a sesqui-extension of a string C-group $\Gamma$ is \emph{proper} if the resulting group is $C_2\times \Gamma$, otherwise it is  \emph{improper}.
\begin{prop}\label{properAn}
Any sesqui-extension of an alternating group is proper.
\end{prop}
\begin{proof}
Suppose that $\Gamma = (G,\{\rho_0, \ldots, \rho_{r-1}\})$ is a sggi with $G$ a permutation group of degree $n$ isomorphic to $A_n$.
Let $\Gamma^*$ be a sesqui-extension of $\Gamma$ with respect to $\rho_i$.
Let $G'$ be the subgroup of $G$ whose elements are written as a product of generators with an even number of $\rho_i$'s.
As the alternating group has no subgroups of index $2$, $G'=G$. Hence, any generator $\rho_i$ can be written as a word $w$ with an even number of generators $\rho_i$ in it. Thus the identity can be written as $w\rho_i$, which has an odd number of $\rho_i$'s.
By Lemma~\ref{sesqui}(b),  we then have that $G^* \cong C_2\times G$ meaning that $\Gamma^*$ is proper.
\end{proof}


\begin{prop}\label{properSn}
Let $G$ be a transitive permutation group of degree $n$.
Let $\Gamma = (G,\{ \rho_0,\ldots,\rho_{r-1}\})$ be a string C-group.
Suppose that $0$ is the label of a perfect split of $\Gamma$.
If $r\geq \frac{n+3}{2}$, then $G\cong S_n$ and a sesqui-extension of $\Gamma$ with respect to $\rho_0$ is proper.
\end{prop}
\begin{proof}
By Proposition~\ref{primpb} $G$ is primitive and by Theorem~\ref{altn} and Theorem~\ref{bigprim}, $G\cong S_n$.
Consider the subgroup $G'=\langle \rho_1^{\rho_0},\rho_1,\ldots,\rho_{r-1}\rangle$. 
If a sesqui-extension with respect to $\rho_0$ is improper, then $|G:G'|=2$. Hence $G'\cong A_n$.
This implies that $G_0$ has one orbit of length 1 and one orbit of length $n-1$. Therefore $G_0$ is a subgroup of $A_{n-1}$. 
As $r\geq \frac{n+3}{2}$, 
combining Theorems~\ref{altn} to~\ref{bigprim} we readily see that $G_0$ has to be transitive imprimitive on the orbit of $n-1$ points and that the permutation representation graph on that orbit must be one of graphs (2) and (3) of Table~\ref{BI}.
But, by Proposition~\ref{att}, none of these graphs are split-attachable,  a contradiction.
\end{proof}


\section{Transitive string C-groups having  perfect splits}\label{tranS}

In what follows let $G$ be a transitive permutation group of degree $n$ and let $\Gamma = (G,\{ \rho_0,\ldots,\rho_{r-1}\})$ be a string C-group. In this case, 
when $i$ is a perfect split of $\Gamma$, then $G_i$ has exactly two orbits. As before let $n_1:=|O_1|$, where $O_1$ is the $G_{<i}$-orbit and $n_2:=|O_2|$,  where $O_2$ is the $G_{>i}$-orbit.
Let $\rho_i=\alpha_i\beta_i(a,b)$ with $\alpha_i$ and $\beta_i$ being the actions of $\rho_i$ in $O_1$ and $O_2$ respectively.  

When $\alpha_i$ is nontrivial $\Gamma_{>i-1}$ is, by Proposition~\ref{sesqui1}, a sesqui-extension of a transitive string C-group on $n_2+1$ points, having the same rank as $\Gamma_{>i-1}$.
Analogously, when $\beta_i$ is nontrivial  $\Gamma_{<i+1}$ is a sesqui-extension of a transitive string C-group on $n_1+1$ points, having the same rank as $\Gamma_{<i+1}$.

\begin{prop}\label{OneisSn}
Let  $r\geq \frac{n+3}{2}$.
Suppose that $i$ is the label of a perfect split of $\Gamma$.
Then either $G_{<i+1}$ or $G_{>i-1}$ acts as the symmetric group on the orbit containing the  perfect $i$-split $\{a,b\}$.
\end{prop}
\begin{proof}
Suppose that neither of the groups $G_{<i+1}$ or $G_{>i-1}$ acts as a symmetric group on the orbit containing $\{a,b\}$.
Hence, by Theorem~\ref{altn}, Theorem~\ref{bigprim} and Proposition~\ref{primpb}, $i+1\leq \frac{(n_1+1)+1}{2}$ and  $r-i\leq \frac{(n_2+1)+1}{2}$.
Therefore $r\leq  \frac{n+2}{2}$, a contradiction.
\end{proof}

\begin{prop}\label{SnBigrank}
Let  $r\geq \frac{n+3}{2}$.
Suppose that $i$ is the label of a perfect split of $\Gamma$.
If  $G_{<i+1}$ does not contain $A_{n_1+1}$, then one of the following situations occurs.
\begin{enumerate}
\item $G_{>i}\cong S_{n_2}$ and $r-(i+1)\geq \frac{n_2+3}{2}$;
\item $\Gamma_{<i+1}$ is one of the string C-groups (2), (3) and (8) of Table~\ref{primPolys}.
\end{enumerate}
\end{prop}
\begin{proof}
Suppose first that $\Gamma_{<i+1}$ is one of the string C-groups of Table~\ref{primPolys}. Since, by Proposition~\ref{int}, $\Gamma$ has a fracture graph, only string C-groups (2), (3) and (8) are possible for $\Gamma_{<i+1}$, which gives situation (b).
As $G_{<i+1}$ does not contain $A_{n_1+1}$, we have that $\Gamma_{<i+1}$ is not one of the string C-groups of Table~\ref{an}. Consequently,  by Proposition~\ref{primpb} and Theorem~\ref{bigprim}, $i+1\leq \frac{(n_1+1)-1}{2}$. Hence $r-(i+1)\geq \frac{n_2+3}{2}$.
Moreover by Theorems~\ref{altn} to~\ref{bigprim}, we have $G_{>i}\cong S_{n_2}$. This gives situation (a).
\end{proof}

\begin{prop}\label{SnBigrankdual}
Let  $r\geq \frac{n+3}{2}$.
Suppose that $i$ is the label of a perfect split of $\Gamma$.
If $G_{>i-1}$ does not contain $A_{n_2+1}$, then one of the following situations occurs.
\begin{enumerate}
\item $G_{<i}\cong S_{n_1}$ and $i\geq \frac{n_1+3}{2}$;
\item $\Gamma_{>i-1}$ is the dual of one of the string C-groups (2), (3) or (8) of Table~\ref{primPolys}.
\end{enumerate}
\end{prop}
\begin{proof}
This is the dual of Proposition~\ref{SnBigrank} thus its proof is analogous.
\end{proof}


\begin{prop}\label{bothAnSn}
Let  $r\geq \frac{n+3}{2}$.
Suppose that $i$ is the label of a perfect split of $\Gamma$.
If  $A_{n_1+1}\leq G_{<i+1}$ and $A_{n_2+1}\leq G_{>i-1}$ then either $\alpha_i$ or $\beta_i$ is trivial.
\end{prop}
\begin{proof}
Suppose that neither $\alpha_i$ nor $\beta_i$ is trivial.
By Propositions~\ref{properAn} and \ref{properSn}, $\Gamma_{<i+1}$ (resp. $\Gamma_{>i-1}$) is a proper sesqui-extension of  a string C-group whose group is $A_{n_1+1}$ or $S_{n_1+1}$ (resp. $A_{n_2+1}$ or $S_{n_2+1}$).
By Proposition~\ref{OneisSn}  we may assume (up to duality) that $G_{<i+1}\cong S_{n_1+1}\times \langle \beta_i\rangle$ and $G_{>i-1}$ contains a subgroup $A_{n_2+1}\times \langle \alpha_i\rangle$.
But then $\alpha_i\in G_{<i+1}\cap G_{>i-1}=\langle \rho_i\rangle$, a contradiction.
\end{proof}

\begin{prop}\label{4ps}
Suppose that $i$ is the label of a perfect split of $\Gamma$.
If $j\neq i$ is such that $\Gamma_i$ has a perfect $j$-split,
then either $\Gamma$ has a perfect $j$-split or $j=i\pm 1$.
\end{prop}
\begin{proof}
Suppose without loss of generality that $j>i$.
If $\Gamma_i$ has a perfect $j$-split that is not a perfect $j$-split for $\Gamma$, 
then $\rho_i$ acts non-trivially on both orbits of $\Gamma_j$. 
The commuting property then implies that $j=i+1$.
\end{proof}

\begin{prop}\label{5ps}
Let $r\geq \frac{n+3}{2}$. Suppose that $\Gamma$ has a perfect $0$-split. If $1$ is the label of a perfect split of $\Gamma_0$ then $1$ is the label of a perfect split of $\Gamma$.
\end{prop}
\begin{proof}
Suppose $1$  is a label of perfect split of $\Gamma_0$ but not of $\Gamma$. Then  $\rho_0$ acts nontrivially on the $G_{>1}$-orbit. Then, on that orbit, as $\rho_0$ commutes with all $\rho_j$'s, with $j>1$, there must be a block system for $G_{>1}$ with blocks of size 2 determined by the action of $\rho_0$.
Moreover $\rho_0$ is fix-point-free on the $G_{>1}$-orbit. 
In this case  $G_{>1}$ is imprimitive but, by Proposition~\ref{att}, $\Gamma_{>1}$ cannot be any of the string C-groups of highest rank having one of the permutation representation graphs of Table~\ref{BI}.
Hence $r-2\leq \frac{n-2}{2}$, a contradiction.
\end{proof}

\begin{prop}\label{7ps}
Suppose that $\Gamma$ has a perfect $i$-split. If $j$ is the label of a perfect split of $\Gamma_{>i}$ then, either $j=i+1$ or $j$ is the label of a perfect split of $\Gamma$.
\end{prop}
\begin{proof}
Suppose that $j$ is the label of a perfect split of $\Gamma_{>i}$ but not of $\Gamma$. 
Then there exists a permutation $\rho_l$ with $l\leq i$ acting nontrivially on the $G_{>j}$-orbit.
Then  there is a block system for $G_{>j}$ with blocks of size 2 determined by the action of $\rho_l$.
Moreover $\rho_l$ is fix-point-free on the $G_{>j}$-orbit. Particularly it acts nontrivially on one of the vertices of the perfect  $j$-split, which implies that  $l=j-1$.
As $l\leq i$ and $j=l+1>i$, we must have $l=i$ and $j=i+1$, as wanted.
\end{proof}



\begin{prop}\label{6ps}
Suppose that $\Gamma$ has perfect splits with labels $i$ and $j$ with $i<j$.
If $l$ is a perfect split of $\Gamma_{\{i,\ldots,j\}}$ then $l$ is a perfect split of $\Gamma$.
\end{prop}
\begin{proof}
Suppose that $l$ is not a perfect split of $\Gamma$.
Then there exists $k\notin\{i,\ldots,j\}$ such that $\rho_k$ acts on both $G_l$-orbits.
Suppose without loss of generality that $k<i$. This contradicts the fact that $i$ is the label of a perfect split.
\end{proof}


\begin{prop}\label{3ps}
Let  $r\geq \frac{n+3}{2}$. 
Then $\Gamma$ has a perfect split with label $i\notin\{0,r-1\}$.
\end{prop}
\begin{proof}
Suppose that if $i$ is the label of a perfect split of $\Gamma$ then either $i=0$ or $i=r-1$.

Since $r\geq \frac{n+3}{2}$, by Proposition~\ref{int}, $G$ is isomorphic to $S_n$ and $\Gamma$ has a fracture graph and, by Proposition~\ref{split}, $\Gamma$ has a perfect $i$-split.
 Without loss of generality we may assume that $0$ is the label of a perfect split of $\Gamma$.
Then $G_0$ is a transitive group on $n-1$ points and has rank $r-1\geq \frac{(n-1)+2}{2}$. 
By the results of Section~\ref{T} and the fact that $\Gamma_0$ is split-attachable, we have that $G_0\cong S_{n-1}$.
Hence, by Proposition~\ref{split}, $\Gamma_0$ has a perfect $j$-split.

Suppose that $j\neq r-1$. 
Then $j$ cannot be the label of a perfect split of $\Gamma$. 
By Proposition~\ref{4ps}, $j=1$. Then by Proposition~\ref{5ps}, $1$ is the label of a perfect split of $\Gamma$, a contradiction. 
Hence $j=r-1$ and both $0$ and $r-1$ are labels of perfect splits of $\Gamma$.

Now consider $\Gamma_{0,r-1}$ acting transitively on $n-2$ points.
Suppose that $\Gamma_{0,r-1}$ has a perfect $j$-split. Then it cannot be a perfect split of $\Gamma$.
By Proposition~\ref{4ps} it is either a perfect split of $\Gamma_{r-1}$ or $j=1$.
But if  $1$ is the label of a perfect split of $\Gamma_{0,r-1}$ it must be the label of a perfect split of $\Gamma_0$, by Proposition~\ref{4ps}. 
Then, by Proposition~\ref{5ps}, $1$ is the label of a perfect split of $\Gamma$, a contradiction.
Hence $j$ is the label of a perfect split of $\Gamma_{r-1}$.
By the dual of Proposition~\ref{4ps}, as $j$ is not the label of perfect split of $\Gamma$, $j=r-2$.
But the dual of Proposition~\ref{5ps}, gives a contradiction as well.
Consequently $\Gamma_{0,r-1}$  has no perfect splits.

By Proposition~\ref{att}, all the sggi's of Table~\ref{notperfect} are not split-attachable. 
Hence, by Proposition~\ref{split}, $G_{0,r-1}\not\cong S_{n-2}$.
Similarly, 
by Proposition~\ref{att}, 
$\Gamma_{0,r-1}$ cannot be one of the string C-groups of Table~\ref{an} 
nor one of the imprimitive string C-groups of highest rank of Table~\ref{BI}. Hence the rank of $\Gamma_{0,r-1}$ is at most $\frac{n-2}{2}$, meaning $r-2\leq \frac{n-2}{2}$, a contradiction.

This shows that $\Gamma$ has a perfect split with label $i\notin\{0,r-1\}$.
\end{proof}

\begin{prop}\label{3ps++}
Suppose that $\Gamma$ has a fracture graph.
Let $r\geq \frac{n+1}{2}$ and $n>3$. If $0$ and $r-1$ are perfect splits, then $\Gamma_{0,r-1}$ has a perfect split.
\end{prop}
\begin{proof}
Consider the group $\Gamma_{0,r-1}$ acting on $n-2$ points and rank $r-2\geq \frac{(n-2)-1}{2}$.
Suppose that $\Gamma_{0,r-1}$ does not have a perfect split.
Then, either $\Gamma_{0,r-1}$ has a 2-fracture graph, or it has a split that is not perfect.
First assume that $\Gamma_{0,r-1}$ admits a 2-fracture graph.
Then $\Gamma_{0,r-1}$ must be one of the string C-groups of Table~\ref{T2F}. By Proposition~\ref{att}, it must then have permutation representation graph number (9), but then $\Gamma$ has the permutation representation graph (3) of Proposition~\ref{IPfails}. Hence $\Gamma$
does not satisfy the intersection property.
Thus $\Gamma_{0,r-1}$ must have a split of label $i$ and this split is not perfect.
Suppose there exists $h$ with $i<h<r-1$ such that $\rho_h$ acts nontrivially in both $G_i$-orbits.
Then $G_{>i}$ and $G_{<i}$, both act nontrivially on the
 $G_i$-orbit not fixed pointwise by $\rho_0$.
 Let $A$ be the group action of $G_i$ on that orbit.
By Proposition~\ref{CCD}, $A$ cannot be primitive.
But then an argument similar to the one in the proof of Proposition~\ref{primpb}, gives a contradiction. 
The case where $h<i$ is handled by dualizing all the arguments.
\end{proof}

\begin{prop}\label{3ps+}
 If  $r\geq \frac{n+3}{2}$ then there exists a perfect split of $\Gamma$ with label $i\notin\{ 0,\,r-1\}$ such that either $\alpha_i$ or $\beta_i$ is trivial.
\end{prop}
\begin{proof}
For $r=4$, $n\leq 5$ hence there is only one possibility for  $\Gamma$, namely the string C-group of degree 5 corresponding to the 4-simplex. In this case, the proposition holds trivially.
Suppose that the proposition holds for every rank, from $4$ up to $r-1$.
Let $i\in\{1,\ldots, r-2\}$ be the label of a perfect split (whose existence is guaranteed by Proposition~\ref{3ps}). Assume that $\rho_i$ acts nontrivially in both $G_i$ orbits.
By Proposition~\ref{bothAnSn} we need to consider the cases $A_{n_2+1}\nleq G_{>i-1}$ or $A_{n_1+1}\nleq G_{<i+1}$, which can be dealt similarly by duality.
Assume that $A_{n_1+1}\nleq G_{<i+1}$. By Proposition~\ref{SnBigrank} we have two possible situations, (a) and (b). 

If we have the situation (a) of Proposition~\ref{SnBigrank}, which is $G_{>i}\cong S_{n_2}$ and $r-(i+1)\geq \frac{n_2+3}{2}$, then, by induction, $\Gamma_{>i}$ has a perfect $j$-split with $j\not\in \{i+1, r-1\}$ and either $\alpha_j$ or $\beta_j$ is trivial. 
Hence we have what we want if $j$ is the label of a perfect split of $\Gamma$. 
Suppose the contrary. Then, by Proposition~\ref{7ps}, $j=i+1$, a contradiction.
Hence, $j$ must be the label of a perfect split of $\Gamma$ satisfying what we need.

Consider that we have situation (b) of Proposition~\ref{SnBigrank}. Then we have the following three possibilities.

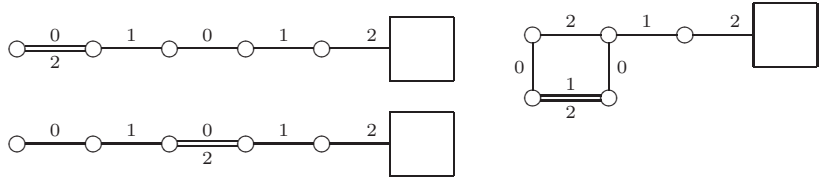
\begin{center}
\begin{tabular}{cc}
\begin{tabular}{c}
$\xymatrix@-1.3pc{*+[o][F]{} \ar@{=}[rr]^0_2 && *+[o][F]{}  \ar@{-}[rr]^1&&*+[o][F]{} \ar@{-}[rr]^0 && *+[o][F]{} \ar@{-}[rr]^1 && *+[o][F]{} \ar@{-}[rr]^2&&*++++[][F]{} }$\\
\\
$\xymatrix@-1.3pc{*+[o][F]{} \ar@{-}[rr]^0 && *+[o][F]{}  \ar@{-}[rr]^1&&*+[o][F]{} \ar@{=}[rr]^0_2 && *+[o][F]{} \ar@{-}[rr]^1 && *+[o][F]{} \ar@{-}[rr]^2&&*++++[][F]{} }$\\
\end{tabular}
&
\begin{tabular}{c}
$\xymatrix@-1.3pc{*+[o][F]{}  \ar@{-}[rr]^2&&*+[o][F]{}  \ar@{-}[rr]^1 && *+[o][F]{} \ar@{-}[rr]^2&& *++++[][F]{} \\
                               *+[o][F]{}  \ar@{=}[rr]_2^1 \ar@{-}[u]^0&&*+[o][F]{}  \ar@{-}[u]_0&& &&}$\\
 \\
\\
\end{tabular}
\end{tabular}
\end{center}

Moreover $\Gamma_{0,1}$ is a sesqui-extension of a string C-group of rank $r-2\geq\frac{ (n-4)+3}{2}$ acting transitively on $n-4$ points. 

Then by induction $\Gamma_{0,1}$ has a perfect $j$-split with $j\in\{3,\ldots, r-2\}$
with  $\alpha_j$ or $\beta_j$ trivial.
Again if $j$ is the label of a perfect split of $\Gamma$ we are done. 
Otherwise, by Proposition~\ref{7ps} $j=2$, a contradiction.
\end{proof}

\section{The proof of Theorem~\ref{main}}\label{theTh}
 
In what follows let $G$ be the symmetric group $S_n$ acting on $n$ points.
Let $r=n-\kappa$ with $n\geq 2\kappa+3$.
Finally let $\Gamma=(G,\{ \rho_0,\ldots, \rho_{r-1}\}) \in \Sigma^{{\kappa}}(n)$ (see Section~\ref{intrud} for the definition of $\Sigma^\kappa(n)$).

By Proposition~\ref{int}, $G_i$ is intransitive for all $i\in\{0,\ldots, r-1\}$.
Here we will continue to use the notation described at the begining of Section~\ref{tranS}.
 
\begin{prop}\label{imink'}
Let $i$ be the minimal label of a perfect split of $\Gamma$.
If $i\neq 0$ then $(n_2+1)-(r-i)<\kappa$.
\end{prop}
\begin{proof}
Suppose that $(n_2+1)-(r-i)\geq \kappa$. Then $i\geq n_1-1$, and therefore $\Gamma_{<i}$ is either trivial or the $i$-simplex. It cannot be the string C-group corresponding to the hemicube as all $G_i$'s have to be intransitive and the hemicube does not satisfy that. Hence either $i=0$ or $\Gamma_{<i}$ has a perfect split, a contradition.
Hence if $i\neq 0$, $(n_2+1)-(r-i)<\kappa$.  
\end{proof}
\begin{prop}\label{imin}
Let $i\neq 0$.
If $i$ is the minimal label of a perfect split  of $\Gamma$  then $i\leq \frac{n_1-1}{2}$. 
\end{prop}
\begin{proof}
Suppose that $i\geq \frac{n_1}{2}$.
The sggi's with the permutation representations given in Tables~\ref{an}, \ref{BI}, \ref{primPolys} and \ref{notperfect}, either have a perfect split, or are not split-attachable by Proposition~\ref{att}.
Hence $\Gamma_{<i}$ cannot be any of the string C-groups of these tables. 
Therefore by Proposition~\ref{split}, $G_{<i}$  is  not isomorphic to $S_{n_1}$,  and by  Theorem~\ref{altn} $G_{<i}$  is not isomorphic to $A_{n_1}$.
 By Theorem~\ref{bigprim} $G_{<i}$ cannot act primitively on the orbit of size $n_1$ of $G_i$, but by Theorem~\ref{bigimp} it cannot act imprimitively either, which gives a contradiction.
\end{proof}

\begin{prop}\label{i2}
Let $i\neq 0$.
If $i$ is the minimal label of a perfect split, then either  $\Gamma_{>i-1}\in \Sigma^{\kappa'}(n_2+1)$ or $\Gamma_{>i-1}$ is a sesqui-extension, with respect to  the first generator,  of a string C-group $\Phi$ with $\Phi\in \Sigma^{\kappa'}(n_2+1)$.
In any case $n_2+1\geq 2\kappa'+l$ and $\kappa'<\kappa$.
\end{prop}
\begin{proof}
As $n\geq 2\kappa+l$ (or  equivalently $r\geq \frac{n+l}{2}$) and, by Proposition~\ref{imin}, $i\leq \frac{n_1-1}{2}$, we have $r-i\geq \frac{(n_2+1)+l}{2}$. 
Let $\kappa'=(n_2+1)-(r-i)$. If we consider the equivalence
$$n\geq 2\kappa+l\Leftrightarrow r \geq \frac{n+l}{2}$$
and replace $\kappa$ by $\kappa'$, $n$ by $n_2+1$ and $r$ by $r-1$, one get 
$$n_2+1\geq 2\kappa'+l.$$
As  $r-i\geq \frac{(n_2+1)+l}{2}$ and $l\in\{3,4\}$,  $G_{>i-1}$ acts as the symmetric group on the orbit of size $n_2+1$.
The fact that $\kappa'<\kappa$ follows from Proposition~\ref{imink'}.
\end{proof}
 
 For each string C-group $\Phi=(S_n, \{ \gamma_0,\ldots, \gamma_{r-1} \}) \in \Sigma^{{\kappa}}(n)$ consider the following set.
 
\begin{center}

\begin{tabular}{rl}
$P_{\Phi}:=\{i\; |$ & $i$ is the label of a perfect split of $\Phi$ and\\
			  & $\gamma_i$ fixes each point on one of the two\\
			  & $\Phi_i$-orbit unless the point belongs to the split\}.
\end{tabular}

 \end{center}

\begin{prop}\label{1}
 $|\Sigma^{\kappa}(n)|\leq| \Sigma^{\kappa}(n+1)|$.
\end{prop}
\begin{proof} 
Let $H$ be the symmetric group $S_n$ of degree $n$ and let $\Phi= (H,\{ \gamma_0,\ldots, \gamma_{r-1} \}) \in \Sigma^{\kappa}(n)$.
By Proposition~\ref{3ps+},  $P_{\Phi}\setminus\{0,r-1\}\neq \emptyset$.
Moreover if $j\in P_{\Phi}\setminus\{0,r-1\}$ then, $Fix(H_j)=\emptyset$. 
Thus by Theorem~\ref{RDE},  $\Phi^{j\uparrow}\in \Sigma^{\kappa}(n+1)$.
Now consider the following correspondence from $\Sigma^{\kappa}(n)$ to $\Sigma^{\kappa}(n+1)$.
 $$\Phi\mapsto \Phi^{j\uparrow}\mbox{ where }j=\min(P_{\Phi}\setminus\{0,r-1\})$$
 The minimality of $j$ forces this mapping to be injective as $j$ being minimal on $\Phi$ will force $j$ to be also minimal in $\Phi^{j\uparrow}$.
 Therefore $|\Sigma^{\kappa}(n)|\leq| \Sigma^{\kappa}(n+1)|$.
\end{proof}

In what follows we will prove that the correspondence given in the proof of Proposition~\ref{1} is a bijection, meaning the following: for  each string C-group $\Phi=(S_n, \{ \gamma_0,\ldots, \gamma_{r-1} \}) \in \Sigma^{\kappa}(n)$ such that  $n\geq 2\kappa+4$, if $j=\min(P_{\Phi}\setminus\{0,r-1\})$ then $j+1\in P_{\Phi}\setminus\{0,r-1\}$. Moreover either $\gamma_j$ or $\gamma_{j+1}$ is a transposition. This is true when $\kappa=1$ as in this case, $n\geq 5$ and we have the simplex.

We now proceed by induction on $\kappa$.

\begin{induction} \label{ind}  If $\kappa'<\kappa$,  $n\geq 2\kappa'+4$, $\Phi=\langle \gamma_0,\ldots, \gamma_{r-1} \rangle$ and  $j=\min(P_{\Phi}\setminus\{0,r-1\})$  then $j+1\in P_{\Phi}\setminus\{0,r-1\}$ and either $\gamma_j$ or $\gamma_{j+1}$ is a transposition. \end{induction}

\begin{prop}\label{0r-1}
Let $r\geq \frac{n+2}{2}$.
If  $\{0,r-1\} \subseteq P_{\Gamma}$ and  $\{1,r-2\} \cap P_{\Gamma}=\emptyset $ then one of the following situations occurs.
\begin{enumerate}
 \item there exists $j\in P_{\Gamma}\setminus\{0,r-1\}$ such that $j+1\in P_{\Gamma}\setminus\{0,r-1\}$  and either $\rho_j$ or $\rho_{j+1}$ is a transposition.
 \item $3$ is the label of a perfect split of $\Gamma$ and $\Gamma_{<3}$ is the string C-group (3) of Table~\ref{primPolys}.
 \end{enumerate}
\end{prop}
\begin{proof}
Suppose that this proposition is false and let  $\Gamma$ be the smallest string C-group with $r\geq \frac{n+2}{2}$, $\{0,r-1\} \subseteq P_{\Gamma}$, $\{1,r-2\} \cap P_{\Gamma}=\emptyset $ and such that neither (a) nor (b) hold.

Consider $\Gamma_{0,r-1}$. It is such that $G_{0,r-1}$ acts transitively on $n-2$ points and it has rank $r-2\geq \frac{n-2}{2}$. By Proposition~\ref{3ps++}, $\Gamma_{0,r-1}$ has a perfect split.


Let $i$ and $j$ be, respectively, the minimal  and the maximal labels of  perfect splits of $\Gamma_{0,r-1}$.
Let $n_1$ be the size of the $G_i$-orbit that is not fixed by $G_{<i}$ and  $n_2$ be the size of the $G_j$-orbit that is not fixed by $G_{>j}$. Notice that as $\{1,r-2\} \cap P_{\Gamma}=\emptyset $, $n_1>2$ and $n_2>2$.

\[\begin{array}{c}
\xymatrix@-1.3pc{*+[o][F]{} \ar@{-}[rr]^(.4)0 && *++++[][F]{}  \ar@{-}[rr]^i&&*++++[][F]{} \ar@{-}[rr]^j && *++++[][F]{} \ar@{-}[rr]^(.65){r-1} && *+[o][F]{} }\\
\underbrace{\hspace{50pt}}\hspace{70pt}\underbrace{\hspace{50pt}}\\
\quad n_1\hspace{110pt}n_2\\
\end{array}\]

By construction, the string C-group $\Gamma_{\{1,\ldots, i-1\}}$ has no perfect splits on the orbit of size $n_1-1$. Therefore
Proposition~\ref{3ps++} applied to $\Gamma_{\{0,\ldots,i\}}$ on the orbit of size $n_1+1$ gives
$$ (E1) \qquad i+1\leq \frac{n_1+1}{2}.$$

Similarly,
Proposition~\ref{3ps++} applied to $\Gamma_{\{j,\ldots,r-1\}}$ on the orbit of size $n_2+1$ gives
$$(E2) \qquad r-j\leq \frac{n_2+1}{2}.$$

 
If $i=j$, then  $n=n_1+n_2$ and  $r\leq \frac{n}{2}$, a contradiction.
Thus $i\neq j$.

In what follows let $\Psi$ be the string C-group corresponding to the action of $G_{\{i,\ldots, j\}}$ on $n-n_1-n_2+2$ points and with rank $j-i+1$ satisfying the following inequality.
$$j-i+1\geq \frac{(n-n_1-n_2+2)+2}{2}$$

Suppose first that  $\{i+1, j-1\}\cap P_{\Psi}=\emptyset$.
As $\Gamma$ is a smallest counter-example to this proposition, then $\Psi$ must satisfy either (a) or (b).
Suppose that $\Psi$ satisfies (a).
Hence there exists  $l\in P_{\Psi}\setminus\{i,j\}$ such that $l+1\in P_{\Psi}\setminus\{i,j\}$  and either $\rho_l$ or $\rho_{l+1}$ is a transposition.
But then by Proposition~\ref{6ps}, $\Gamma$ also satisfies (a), a contradiction.

If $\Psi$ satisfies (b), then $\Gamma$ has a permutation representation graph as follows.
$$
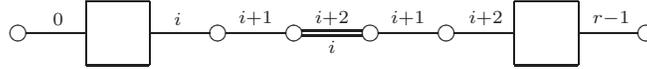
\xymatrix@-1.3pc{*+[o][F]{} \ar@{-}[rr]^(.4)0 && *++++[][F]{}  \ar@{-}[rr]^(.6)i&& *+[o][F]{}  \ar@{-}[rr]^{i+1}&&*+[o][F]{} \ar@{=}[rr]^{i+2}_i && *+[o][F]{} \ar@{-}[rr]^{i+1} && *+[o][F]{} \ar@{-}[rr]^(.4){i+2}&& *++++[][F]{} \ar@{-}[rr]^(.65){r-1} && *+[o][F]{} }$$
Then the intersection property fails, as $A_{n-n_1-1}\leq G_{>i}$, $A_{n_1+5}\leq G_{<i+3}$ and $\langle \rho_{i+1},\rho_{i+2}\rangle\cong D_5$. 
Thus $\{i+1, j-1\}\cap P_{\Psi}\neq \emptyset$.

By duality we may assume that  $i+1$ is the label of a perfect split of $\Psi$, while $j-1$ might be the label of a perfect split or not.
Then by Proposition~\ref{6ps}  $i+1$ must be the label of a perfect split of $\Gamma$.
Hence we have that $i$ and $i+1$ are consecutive labels of perfect splits of $\Gamma$ as shown in the following graph.
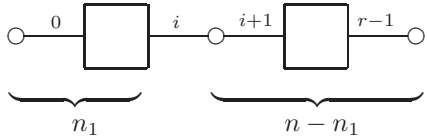
\[\begin{array}{c}
\xymatrix@-1.3pc{*+[o][F]{} \ar@{-}[rr]^(.4)0 && *++++[][F]{}  \ar@{-}[rr]^(.6)i&&*+[o][F]{}  \ar@{-}[rr]^(.4){i+1}&&*++++[][F]{} \ar@{-}[rr]^(.6){r-1} && *+[o][F]{} }\\
\underbrace{\hspace{50pt}}\hspace{25pt}\underbrace{\hspace{80pt}}\\
n_1\hspace{70pt}n-n_1\\
\end{array}\]

As by assumption $\Gamma$ is a counter-example to this proposition, neither $\rho_{i+1}$ nor $\rho_i$ is a transposition. 
Particularly if $\{a,b\}$ is the $i$-split, $(a,b)\notin \langle \rho_i\rangle$.
If $S_{n_1+1}\leq G_{<i+1}$ and $ S_{n-n_1+1}\leq G_{>i-1}$ then  $(a,b)\in G_{<i+1}\cap G_{>i-1}$, a contradiction. 
Hence either  $S_{n_1+1}\not\leq G_{<i+1}$ or $S_{n-n_1+1}\not\leq G_{>i-1}$.

If $S_{n_1+1}\not\leq G_{<i+1}$ we can improve $(E1)$ in the following way.
By Proposition~\ref{primpb}, we have that $G_{<i+1}$ is primitive on the orbit of size $n_1$. 
Suppose we have $i+1\geq \frac{n_1+1}{2}$.
In that case, $\Gamma_{<i+1}$ has, on the orbit of $n_1+1$ points, a permutation representation graph that is one of Table~\ref{an} or Table~\ref{primPolys}. 
By Proposition~\ref{att}, the graphs of Table~\ref{an} are excluded and only graphs (2), (3) (5) and (8) of Table~\ref{primPolys} need to be considered.
Note that the only possible perfect splits of $\Gamma_{<i+1}$ have labels $0$ and $i$ (thanks to the minimality of $i$).
Hence $\Gamma_{<i+1}$ needs to have permutation representation graph (3) on the orbit of $n_1+1$ points.
Therefore $\Gamma$ satisfies (b) of this proposition, a contradiction.
Thus  $S_{n_1+1}\not\leq G_{<i+1}$ implies that  
$$(E1^+)\qquad i+1\leq \frac{n_1}{2}.$$

Dually  $S_{n-n_1+1}\not\leq G_{>i-1}$ implies $(r-1)-i+1=r-i\leq \frac{n-n_1}{2}$. 
But then, combining this with the inequality (E1), we have $r\leq  \frac{n-n_1}{2}+\frac{n_1-1}{2}= \frac{n-1}{2}$, a contradiction. Thus  we necessarily have $S_{n-n_1+1}\leq G_{>i-1}$.
Thus, the intersection property forces that  $S_{n_1+1}\not\leq G_{<i+1}$ and $(E1^+)$ holds. Therefore combining $(E2)$ with $(E1^+)$, we get
$$(E3)\qquad j-i+1\geq \frac{(n-n_1-n_2+2)+3}{2}.$$

If we have equality, then $j-i+1= \frac{(n-n_1-n_2+2)+3}{2}$. But this can only happen when we have the equality in $(E2)$, namely that  $r-j=\frac{n_2+1}{2}$ (and in particular, $n_2$ is odd).

Now either $j-1\in P_{\Psi}$ or not.
Suppose first that $j-1\in P_{\Psi}$. 
Then, we have that $j-1$ and $j$ are consecutive labels of perfect splits, as shown in the following graph.
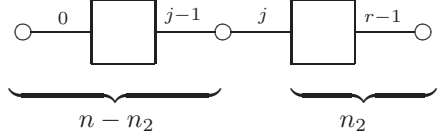
\[\begin{array}{c}
\xymatrix@-1.3pc{*+[o][F]{} \ar@{-}[rr]^(.4)0 && *++++[][F]{}  \ar@{-}[rr]^(.6){j-1}&&*+[o][F]{}  \ar@{-}[rr]^(.4)j&&*++++[][F]{} \ar@{-}[rr]^(.6){r-1} && *+[o][F]{} }\\
\underbrace{\hspace{80pt}}\hspace{25pt}\underbrace{\hspace{55pt}}\\
n-n_2\hspace{70pt}n_2\\
\end{array}\]
 Then assuming that neither $\rho_{j-1}$ nor $\rho_{j}$ is a transposition, the intersection property implies that $S_{n_2+1}\not\leq G_{>j-1}$. Dually to what we did before, with the consecutive labels of perfect splits $i$ and $i+1$, we get an improvement of  the inequality (E2).
$$r-j\leq \frac{n_2}{2}.$$
So $r-j \neq \frac{n_2+1}{2}$, a contradiction.
Hence $j-1\notin P_{\Psi}$.

Now consider the string C-group $\Phi$ corresponding to the action of $\Gamma_{\{i+1,\ldots,j\}}$ on $n-n_1-n_2+1$ points with rank $j-(i+1)+1=j-i$.
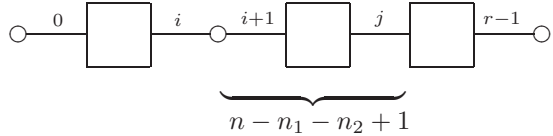
\[\begin{array}{c}
\xymatrix@-1.3pc{*+[o][F]{} \ar@{-}[rr]^(.4)0 && *++++[][F]{}  \ar@{-}[rr]^(.6)i&&*+[o][F]{}  \ar@{-}[rr]^(.4){i+1}&&*++++[][F]{} \ar@{-}[rr]^(.5){j} &&*++++[][F]{} \ar@{-}[rr]^(.6){r-1} && *+[o][F]{} }\\

\hspace{30pt}\underbrace{\hspace{70pt}}\hspace{5pt}\\
\hspace{30pt}n-n_1-n_2+1\\
\end{array}\]

If $i+2$ is the label of a perfect split then $\rho_{i+1}$ is a transposition, which implies that $\Gamma$ satisfies (a), a contradiction. Thus $\{i+2,j-1\}\cap P_{\Phi}=\emptyset$.
The inequalities  (E1$^+$) and (E2) give the following.
$$j-i\geq \frac{(n-n_1-n_2+1) +2}{2}$$
If  $\Phi$ satisfies (a) of this proposition then $\Gamma$ also satisfies (a), a contradiction.
If $\Phi$ satisfies (b), we get a contradiction with the intersection property.
But then $\Phi$ is a smaller counter-example than $\Gamma$ to the proposition, a contradiction.

We have thus proven that $(E3)$ is a strict inequality, namely that $$j-i+1\geq \frac{(n-n_1-n_2+2)+4}{2}.$$ 
Let us now prove that  $\kappa':= (n-n_1-n_2+2)-(j-i+1)<\kappa$. 
\[\begin{array}{rl}
(n-n_1-n_2+2)-(j-i+1)&=(n-r)-n_1-n_2+(r-j)+(i+1) \\
&\leq \kappa-n_1-n_2 +  \frac{n_2+1}{2}+\frac{n_1}{2}\\
&=\kappa-\frac{n_1+n_2-1}{2}<\kappa
\end{array}\]
By Induction~\ref{ind}, we get that there exists $j\in P_{\Gamma}\setminus\{0,r-1\}$ 
such that $j+1\in P_{\Gamma}\setminus\{0,r-1\}$  and either $\rho_j$ or $\rho_{j+1}$ is a transposition, a contradiction.
\end{proof}


In what follows let  $n\geq 2\kappa+4$ (or $r\geq \frac{n+4}{2}$).

\begin{prop}\label{remaining}
If  $\{0,1, r-2,\,r-1\} \subseteq P_{\Gamma}$  then there exists $j\in P_{\Gamma}\setminus\{0,r-1\}$ 
such that $j+1\in P_{\Gamma}\setminus\{0,r-1\}$  and either $\rho_j$ or $\rho_{j+1}$ is a transposition.
\end{prop}
\begin{proof}
If either $2$ or $r-3$ are labels of perfect splits of $\Gamma_{0,r-1}$, and therefore, by Proposition~\ref{6ps}, of $\Gamma$, we get  that $j=1$ is as wanted.
Suppose that $\{2,r-3\}\cap P_{\Gamma}=\emptyset$.
Consider $\Gamma_{0,r-1}$ with rank $r-2$ acting on $n-2$ points. As $r-2\geq \frac{(n-2)+2}{2}$, by Proposition~\ref{0r-1}, there are two possibilities. 
If we have case (b), $\Gamma$ has a perfect split with label $4$ and is as follows.
$$
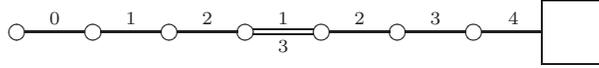
\xymatrix@-1.3pc{*+[o][F]{} \ar@{-}[rr]^0 &&*+[o][F]{} \ar@{-}[rr]^1 && *+[o][F]{}  \ar@{-}[rr]^2&&*+[o][F]{} \ar@{=}[rr]^1_3 && *+[o][F]{} \ar@{-}[rr]^2 && *+[o][F]{} \ar@{-}[rr]^3&& *+[o][F]{} \ar@{-}[rr]^(.4)4&&*++++[][F]{} }$$
But then $S_7\leq G_{<4}$ and $G_0\cong S_{n-1}$ (as $r-1\geq \frac{(n-1)+3}{2}$), so $G_{<4} \cap G_0 \geq S_5$ and $G_{\{1,2,3\}}\cong A_5$, contradicting the intersection property. Thus only Case (a) may occur and this finishes the proof.
\end{proof}

\begin{prop}\label{0r-1(2)}
If  $\{0,\,r-1\} \subseteq P_{\Gamma}$ and  $|\{1,\,r-2\}\cap P_{\Gamma}|\leq 1$ then there exists $j\in P_{\Gamma}\setminus\{0,r-1\}$ 
such that $j+1\in P_{\Gamma}\setminus\{0,r-1\}$  and either $\rho_j$ or $\rho_{j+1}$ is a transposition.
\end{prop}
\begin{proof}
Up to duality we may assume that $1$ is the label of a perfect split of $\Gamma$.
In this case, as 0 is the label of a perfect split, $G_0$ is acting transitively on $n-1$ points. Also $\Gamma_0$ is of rank $r-1\geq \frac{(n-1)+3}{2}$.
Now Proposition~\ref{0r-1} applies. Case (b) is excluded by the intersection property as in the previous proof. Case (a) gives what we want.
\end{proof}
\
\begin{prop}\label{indnot0}
If $0\notin P_{\Gamma}$ then there exists $j\in P_{\Gamma}\setminus\{0,r-1\}$ 
such that $j+1\in P_{\Gamma}\setminus\{0,r-1\}$  and either $\rho_j$ or $\rho_{j+1}$ is a transposition.
\end{prop}
\begin{proof}
By Proposition~\ref{i2}  there exists $l>0$ and  $\Phi=\langle \gamma_l,\ldots,\gamma_{r-1}\rangle\in \Sigma^{\kappa'}(n')$ with $\kappa'<\kappa$ and $n'\geq 2\kappa'+4$, such that $\Gamma_{>l-1}$ is either isomorphic to $\Phi$ or to a sesqui-extention of $\Phi$ with respect to its first generator.
By Induction~\ref{ind},  if  $j=\min(P_{\Phi}\setminus\{l,r-1\})$  then $j+1\in P_{\Phi}\setminus\{l,r-1\}$ 
and either $\gamma_j$ or $\gamma_{j+1}$ is a transposition.  As $\rho_j=\gamma_j$ and $\rho_{j+1}=\gamma_{j+1}$  we get what we want.
Note that $j$ and $j+1$ must be labels of perfect splits of $\Gamma$, by Proposition~\ref{6ps}.
\end{proof}


\begin{thm}
 $|\Sigma^{\kappa}(n)|$ is a constant for $n\geq 2\kappa+ 3$.

\end{thm}
\begin{proof}
By Propositions~\ref{0r-1} to~\ref{indnot0}, the correspondence from $\Sigma^{\kappa}(n)$  to $\Sigma^{\kappa}(n+1)$, given in Proposition~\ref{1} is a bijection, hence  $|\Sigma^{\kappa}(n)|=| \Sigma^{\kappa}(n+1)|$ for $n \geq 2\kappa+ 3$.
This is equivalent to say that $|\Sigma^{\kappa}(n)|$ is constant for $n\geq  2\kappa+ 3$.
\end{proof}

\section{Acknowledgements}
The authors thank Natalia Garcia Colin and Mark Mixer for useful comments on a preliminary version of this paper.

\bibliographystyle{amsplain}

\end{document}